\documentclass[12pt]{article}

\usepackage{bm}
\usepackage{geometry}
\usepackage{amsmath}
\usepackage{amsfonts}
\usepackage{amssymb}
\usepackage{amsthm}
\usepackage{appendix}
\usepackage{mathtools}
\DeclarePairedDelimiterX{\inp}[2]{\langle}{\rangle}{#1, #2}
\usepackage{mathrsfs}
\usepackage{authblk}
\usepackage{hyperref}
\usepackage{xcolor}
\usepackage[capitalize]{cleveref}
\usepackage{microtype}
\crefname{prop}{Prop.}{Props.}
\usepackage{tikz}
\usetikzlibrary{arrows.meta, positioning}

\usepackage[style=numeric,sorting=nyt,maxnames=6,giveninits=true]{biblatex}
\DeclareFieldFormat[article,thesis]{title}{#1}
\DeclareFieldFormat[article,thesis]{citetitle}{#1}

\renewbibmacro{in:}{}
\DeclareFieldFormat{pages}{#1}

\usepackage{enumitem}

\hypersetup{hidelinks}

\newtheorem{defn}{Definition}[section]
\newtheorem{nota}{Notation}[section]
\newtheorem{theorem}{Theorem}[section]
\newtheorem{prop}{Proposition}[section]
\newtheorem{lemma}{Lemma}[section]
\newtheorem{coro}{Corollary}[section]
\newtheorem{remark}{Remark}[section]

\newcommand{\p}{\mathfrak{p}}
\renewcommand{\d}{\mathrm{d}}

\newcommand{\R}{\mathbb{R}}
\newcommand{\tg}{\text{tan}}

\newcommand{\esssup}{\operatorname{ess\,sup}}

\newcommand{\normmm}[1]{\left|\mkern-2mu\left|\mkern-2mu\left|#1\right|\mkern-2mu\right|\mkern-2mu\right|}
\numberwithin{equation}{section}

\DeclareMathOperator{\supp}{supp}
\DeclareMathOperator{\dist}{dist}
\DeclareMathOperator{\diag}{diag}
    
\begin{document}
    
    \title{Well-posedness of  initial boundary value problems for 2D compressible MHD equations in domains with corners
    }






\author{
    Wen GUO\thanks{
        Corresponding author. 
        School of Mathematical Sciences, Shanghai Jiao Tong University, Shanghai, China. Email: \href{mailto:guo-wen@sjtu.edu.cn}{guo-wen@sjtu.edu.cn}}
     \quad Ya-Guang WANG\thanks{
        School of Mathematical Sciences, MOE-LSC and SHL-MAC, Shanghai Jiao Tong University, Shanghai, China. Email:
        \href{mailto:ygwang@sjtu.edu.cn}{ygwang@sjtu.edu.cn}}
        }

\maketitle
    
\vspace{.1in}

    \selectfont
    \begin{abstract}
    In this paper, the well-posedness is studied for the initial boundary value problem of the two-dimensional compressible ideal magnetohydrodynamic (MHD) equations in bounded perfectly conducting domains with corners. The presence of corners yields intrinsic analytic obstacles: the lack of smooth tangential vectors to the boundary prevents the use of classical anisotropic Sobolev spaces, and due to the coupling of normal derivatives near corners, one can not follow the usual way to estimate the normal derivatives of solutions from the equations. To overcome these difficulties, a new class of anisotropic Sobolev spaces $H^m_*(\Omega)$ is introduced to treat corner geometries. Within this framework, the well-posedness theory is obtained for both linear and nonlinear problems of the compressible ideal MHD equations with the impermeable and perfectly conducting boundary conditions. The associated linearized problem is studied in several steps: first one deduces the existence of weak solutions by using a duality argument in high order tangential spaces, then verifies that it is indeed a strong solution by several smoothing procedures preserving traces to get a weak–strong uniqueness result, afterwards the estimates of normal derivatives are obtained by combining the structure of MHD equations with the Helmholtz-type decomposition for both of velocity and magnetic fields. 
    \end{abstract}

~

    \tableofcontents
    
    \newpage
    \section{Introduction}\label{section-intro}
    
    In this paper, we study the following initial boundary value problem for the two-dimensional compressible ideal magnetohydrodynamic (MHD) equations in $\{t>0, x\in \Omega\}$, with $\Omega$ being a bounded, simply connected domain of $\R^2$ with finitely many corners: 
    \begin{equation}\label{MHD}
        \begin{cases}
            \frac{1}{Q(p,s)} (\partial_t  + u \cdot \nabla) p + \nabla \cdot u = 0,\\
            R(p,s)(\partial_t  + u \cdot \nabla) u + \nabla p + b\times (\nabla\times b) = 0, \\
            (\partial_t  + u \cdot \nabla) b - b \cdot \nabla  u + b(\nabla \cdot u) = 0,\\
            (\partial_t  + u \cdot \nabla) s = 0,\\
            \nabla \cdot b = 0,
        \end{cases}
    \end{equation}
    supplemented with the initial condition
    \begin{equation}\label{icd}
        (u, b, p, s)|_{t=0} = (u_0(x),  b_0(x),  p_0(x),  s_0(x)), \quad x \in \Omega,
    \end{equation}
    and the impermeable and perfectly conducting wall condition
    \begin{equation}\label{bcd}
    u \cdot \nu = b \cdot \nu = 0, \quad \forall t>0 \text{ and } x\in \partial\Omega\backslash\{q_1, \ldots, q_N\},
    \end{equation}
   where $\nu$ denotes the outward unit normal vector to the smooth part of the boundary $\partial\Omega$, $\{q_j\}_{j=1}^N$ are the corner points of $\partial\Omega$, $u=(u_1,u_2)^t$, $b=(b_1,b_2)^t$,  $p$ and $s$ represent the velocity,  magnetic field, pressure and entropy in the flow respectively, while $R(p,s)$ denotes the density given by the equation of state, and $Q(p,s) = \frac{R}{\partial_p R}$. 
We assume throughout this paper that the equation of state $R(p,s)$ is a smooth function on $\mathcal{V}\subset \R^+\times \R$ with $R>0$ and $\partial_p R >0$.  

    The study of the above problem is not only important and challenging  in mathematical theory, but also has many applications, such as the model characterizes the motion of a plasma confined in a rigid, perfectly conducting wall, the cross-sections of fusion devices may exhibit polygonal structures, and metallic containers or ducts are frequently modeled by polygonal domains.
    The initial boundary value problems (IBVPs) for compressible MHD equations in domains with smooth boundaries have been extensively studied in the literature. 
    The local well-posedness in standard Sobolev spaces was obtained in \cite{yanagisawa1987initial} under the condition $b\times \nu = 0$ on $\partial\Omega$. On the other hand, if the magnetic field is tangential to the boundary, namely, $b\cdot \nu = 0$, it was shown in \cite{ohno1998initial} that the linearized problem near a non-zero magnetic field is ill-posed in standard Sobolev spaces.
    To avoid this, Yanagisawa and Matsumura \cite{yanagisawa1991fixed} established the local well-posedness in anisotropic Sobolev spaces (cf. \cite{chen1980initialF, chen2007initial}) under the impermeable and perfectly conducting wall boundary conditions.

   The existing theoretical analysis  of partial differential equations in  domains with corners mainly restricts to problems of elliptic or parabolic equations,
   e.g.,  Grisvard \cite{grisvard1985elliptic}, Kozlov \cite{kozlov1997elliptic}, Solonnikov
   \cite{solonnikov1986solvability}, and references therein. In recent years, there is also a few progress on the study of Navier-Stokes equations in domains with corners, see \cite{MR4719953, tolksdorf2018p} for example.
    However, the analysis of IBVPs of hyperbolic systems in  domains with corners is still at a very early stage and remains largely open, with only very little progress achieved so far on the linear hyperbolic systems.
   In the 1970s, Osher \cite{osher1973initial, osher1974initial} obtained an a priori estimate of solutions to linear hyperbolic systems defined in domains with corners via a Kreiss symmetrizer, but the conditions imposed there are difficult to verify, and there exists a non-explicit loss of regularity that prevents the proof of strong well-posedness.
   Huang and Temam in \cite{huang2014linearized} obtained the well-posedness in $L^2$ for the linearized shallow water equation at a constant background state in a rectangle, and in (\cite{huang2014hyperbolic}) they got the well-posedness in $L^2$ for a linear hyperbolic system with a structural assumption and a special boundary condition in a rectangle by applying the Hille-Yosida theorem.    
    Under a more restrictive structural assumption, Benoit in \cite{benoit2015problemes} extended the study of \cite{huang2014hyperbolic} to the problem with dissipative boundary conditions, and obtained the estimate of solutions in high-order weighted Sobolev spaces in \cite{benoit2024persistence}.

    For nonlinear hyperbolic problems in non-smooth domains, to our knowledge, the only existing result was obtained by Godin in \cite{godin20212d}, in which he established the local well-posedness of classical solutions to the two-dimensional compressible Euler system in cornered domains with impermeable boundary conditions, provided that the measure of angle at each corner is properly small.
    He derived the a priori estimates of solutions in the usual Sobolev spaces by studying the divergence and curl of velocity from the special structure of the equations.
    
The goal of this work is to establish the well-posedness of the IBVP \eqref{MHD}-\eqref{bcd} defined in a curvilinear polygon. 
The corners on the boundary of the domain bring intrinsic difficulties for  analysis in contrast to problems in the smooth boundary settings. 
In particular, near a corner, there exists neither a smooth nonvanishing vector field tangent to the smooth pieces of $\partial\Omega$ nor a normal one to them. As a result, the classical argument for studying quasilinear hyperbolic systems in smooth domains can no longer be applied directly. The estimates of normal derivatives of solutions shall be obtained by combining the structure of ideal MHD equations with the Helmholtz-type decomposition for both velocity and magnetic fields. 
Moreover, in contrast to the study given Godin \cite{godin20212d} for the compressible Euler system, due to the strong coupling between velocity and magnetic fields in the system \eqref{MHD}, we will see that to estimate one order normal derivative of unknowns one needs two order tangential derivatives.
Furthermore, the lack of smooth tangential vectors prevents the direct use of the anisotropic Sobolev spaces introduced in Chen \cite{chen1980initialF}. To address this, we shall introduce new anisotropic Sobolev spaces that are adapted to the corner geometries.
   
    
    To be more precise, first we introduce the following  assumptions and notations for the domain:

    \begin{nota}\label{nota-1}
    (1)	Assume that $\Omega$ is a bounded simply connected domain of $\R^2$, with the boundary $\partial \Omega$ being a curvilinear polygon of class $C^\infty$ with finitely many convex angles. 
    
    (2) Let $\{q_1,q_2, \ldots, q_N\}$ be the set of corner points of $\partial\Omega$, and the measure of each angle at $q_n~(1\le n\le N)$ is denoted by $\omega_n~(0< \omega_n <\pi)$. 
    
    (3)  For any fixed $\tilde{x}\in \partial\Omega\backslash\{q_1, \ldots, q_N\}$, let $V(\tilde{x})$ be a small neighborhood of $\tilde{x}$ in $\R^2$, and
    $\Phi(x)$ a smooth extension in $V(\tilde{x})$ of the function $\dist(x,\partial\Omega)$ defined in $V(\tilde{x})\cap \overline{\Omega}$, such that $\Omega\cap V(\tilde{x}) = \{x\in V(\tilde{x}), \Phi(x)>0\}$ and that $\Phi$ vanishes on $\partial\Omega\cap V(\tilde{x})$. 
Set $\tilde{\nu} = -\nabla \Phi$, a smooth extension of the outward normal vector $\nu$ of  $\partial\Omega$ at $\tilde{x}$.

(4) For any corner point $\tilde{x} = q_n$ with $1\le n\le N$, denote by $\gamma_1^n$ and $\gamma_2^n$ two legs of the angle at $q_n$, let $\Phi_l$ ($l=1,2$) be the smooth extension of $\dist(x,\gamma_l^n)$ in a small open neighborhood $V(q_n)$ of $q_n$ in $\R^2$ such that
    $\Omega\cap V(q_n) = \{x\in V(q_n), \Phi_1(x)>0,\Phi_2(x)>0\}$, $\Phi_l$ vanishes on $\gamma_l^n \cap V(q_n)$ with $l=1,2$.
   Set $\tilde{\nu}^l = -\nabla \Phi_l$, then $\tilde{\nu}^l = \nu$, the outward normal vector on $\gamma_l^n\cap V(q_n)$ for each $l=1,2$. 
    
    \end{nota}

    For the problem  \eqref{MHD}-\eqref{bcd}, we will see that under the impermeable condition, $u\cdot\nu=0$ on $\partial\Omega \backslash \{q_1, \ldots, q_N\}$, both of the constrain $\nabla \cdot b = 0$ in $\Omega$ and the condition $b\cdot \nu = 0$ on $\partial\Omega \backslash \{q_1, \ldots, q_N\}$ hold in the existing time period of smooth solutions when they are true initially, so the problem \eqref{MHD}-\eqref{bcd} is not over-determined.

   To define the new anisotropic Sobolev spaces,   let us first introduce two smooth vector fields $w^1$ and $w^2$ tangential to $\partial\Omega\backslash\{q_1, \ldots,q_N\}$  as follows:
   
\begin{defn}[Tangential vector fields $w^1$ and $w^2$]\label{def-tan} The tangential vector fields $w^1(x)$ and $w^2(x)$ defined in the whole domain $\overline{\Omega}$ are constructed by using the partition of unity and their following values at each point of  $\overline{\Omega}$:
\begin{enumerate}
	\item[(1)] At each $x\in \Omega$ away from $\partial \Omega$, choose $w^1(x) = (1,0)^t$ and $w^2(x) = (0,1)^t$. 
	
	\item[(2)] For any fixed $\tilde{x}\in \partial\Omega\backslash\{q_1,\ldots,q_N\}$, let $\Phi(x)$ be the smooth function defined in the small neighborhood $V(\tilde{x})$ in Notation~\ref{nota-1}(3). In  
    $V(\tilde{x}) \cap \overline{\Omega}$,   as $\tilde{\nu}=-\nabla\Phi(x)\neq 0$, one has that either $\partial_1 \Phi(x)$ or $\partial_2 \Phi(x)$ is non-vanishing. 
	When $\partial_1 \Phi(x)\ne 0$, by setting
	$w^1(x) = \Phi(x)(1,0)^t$ and $w^2(x) = (\partial_2\Phi(x), -\partial_1 \Phi(x))^t$, one has $w^1(x)\cdot \tilde{\nu}(x) = -\Phi(x)\partial_1\Phi(x)$ vanishing on $V(\tilde{x})\cap\partial\Omega$,  $w^2(x) \cdot \tilde{\nu}(x) = 0$, and while $\partial_2 \Phi(x) \ne 0$, one may proceed analogously by interchanging the indices $1$ and $2$.
	
	\item[(3)] When $\tilde{x} = q_n$ for a fixed $1\le n\le N$, set $w^i = (-1)^{i}\Phi_i(x)\big(\partial_2\Phi_{i'}(x),-\partial_1 \Phi_{i'}(x)\big)^t$ for $i,i'=1,2$, $i\ne i'$, with $\Phi_i(x)$ being the smooth function defined in the small neighborhood $V(q_n)$ in Notation~\ref{nota-1}(4). Then, $w^i(x)\cdot \tilde{\nu}^i(x) = \Phi_i(x)\big( \tilde{\nu}^1(x)\times\tilde{\nu}^{2}(x)\big)$ vanishes  on $\gamma^n_i \cap V(q_n)$, and $w^i(x) \cdot \tilde{\nu}^{i'}(x) = 0$ for each $i,i'=1,2$ with $i\ne i'$.
	
\end{enumerate}

\end{defn}

 Now, we define an anisotropic Sobolev space as follows.
 
 \begin{defn}\label{def-sp} For any fixed integer $m\ge 0$, the space $H^m_*(\Omega)$ is defined  as
 \begin{equation}\label{aniso-space}
 H^m_*(\Omega) := \left\{f\in L^2(\Omega)|~  \partial_1^{\alpha_1}\partial_2^{\alpha_2}(\partial_{w^1})^{\alpha_3} (\partial_{w^2})^{\alpha_4} f\in L^2(\Omega), 2(\alpha_1 + \alpha_2)+\alpha_3 + \alpha_4 \le m\right\},
 \end{equation}
 equipped with the norm
 \begin{equation}\label{norm-sp}
 \|f\|^2_{H^m_*(\Omega)} := \sum_{\substack{2(\alpha_1+\alpha_2)+
 		\alpha_3+\alpha_4\le m}} \| \partial_1^{\alpha_1}\partial_2^{\alpha_2} 
 	(\partial_{w^1})^{\alpha_3} (\partial_{w^2})^{\alpha_4}
 	f\|^2_{L^2(\Omega)},
 \end{equation}
 	where  $\partial_{w^i} = w^i(x)\cdot \nabla$, with the smooth vector fields $w^i(x)~(i=1,2)$ being given in Definition~\ref{def-tan}. 
	
 \end{defn}
 
\begin{remark}\label{rmk-sp}
    For $\tilde{x}\in \overline{\Omega}$, we observe that
    \begin{enumerate}
        \item If $\tilde{x}\in \Omega$ away from $\partial\Omega$, and  $V(\tilde{x})\subset \Omega$ is a  neighborhood of $\tilde{x}$,  then $H^m_*(V(\tilde{x})) = H^m(V(\tilde{x}))$.
        \item If $\tilde{x}\in \partial\Omega\backslash\{q_1,\ldots, q_N\}$, let $\Phi(x)$ be the smooth function defined in the small neighborhood $V(\tilde{x})$ given in Notation~\ref{nota-1}(3), and without loss of generality, we assume that $\partial_1 \Phi(x)\ne 0$. Set the diffeomorphism $\Psi : (x_1,x_2) \mapsto (x_1^*,x_2^*) := (\Phi(x),x_2)$, then 
        $\Psi(V(\tilde{x})) =: \mathcal{U}\subset \{x_1^*>0, x_2^*\in \R\}$, and 
        \begin{equation*}
            \|f\|^2_{H^{m}_*(V(\tilde{x}))} = \sum_{\beta\in\mathbb{N}, 2\beta_1+\beta_2 + \beta_3\le m}\|\partial_{x_1^*}^{\beta_1}\partial_{x_2^*}^{\beta_2}(x_1^*\partial_{x_1^*})^{\beta_3}(f\circ \Psi^{-1})\|^2_{L^2(\mathcal{U})},
        \end{equation*}
        which means that $H^m_*$ reduces to the classical anisotropic Sobolev space introduced in \cite{chen1980initialF} near the smooth part of $\partial\Omega$. 
    \end{enumerate}
\end{remark}

 We will see in Remark~\ref{remark2.1} that the norms of the anisotropic Sobolev space $H^m_*(\Omega)$ are equivalent to each other if the order of derivatives involved  in  \eqref{norm-sp} is changed.

The main result of this paper can be stated as follows.

\begin{theorem}\label{main-thm}
For any fixed integer $m\ge 8$, assume that the angle at the corner $q_n$ satisfies 
\begin{equation}\label{angle-cdt}
    \omega_n \in \left(0, \frac{\pi}{[\frac{m}{2}]}\right) \backslash \left\{\frac{\pi}{[\frac{m}{2}]+1}, \frac{\pi}{[\frac{m}{2}]+2}, \ldots, \frac{\pi}{m}\right\},
\end{equation}
for each $1\le n \le N$ with the notation $[\cdot]$ denoting the integer part of a number, and the initial data given in \eqref{icd}, $(u_0, b_0, p_0, s_0) \in H^{m}(\Omega)$ satisfies the compatibility conditions of the problem \eqref{MHD}-\eqref{bcd} up to order $m-1$ on $ \partial\Omega\backslash\{q_1,\ldots, q_N\}$. Then, there exists $T>0$, such that the problem \eqref{MHD}-\eqref{bcd} has a unique solution $(u,b,p,s)\in \cap_{k=0}^m W^{k,\infty}(0,T; H_*^{m-k}(\Omega))$.

\end{theorem}

Now, we give several remarks on the above result.

\begin{remark}[Regularity of initial data]\label{rmk-1} (1) In the above theorem, the initial data are assumed to be in $H^m(\Omega)$ instead of merely in $H^m_*(\Omega)$, because it is required that they satisfy the compatibility conditions up to order $m-1$. Otherwise, if the initial data belongs to $H^m_*(\Omega)$, then from the equations one has only $(\partial_t^{k} u)(0,\cdot) \in H^{m-2k}_*(\Omega)$ with $k=0,\ldots, [\frac{m}{2}]$. Moreover, the trace of $(\partial_t^{[\frac{m}{2}]} u)(0,\cdot)$ on the boundary $\partial\Omega$ does not make sense in general. See \cite[Remark 2.2-2.3]{secchi1996initial} for discussion on this issue.

(2)
	We will see later  that the condition $m\ge 7$ is enough for the well-posedness of the linearized problem, as the a priori estimates require the bound for terms in the form of $\|\partial_x \partial_{w^1} u(t)\|_{L^\infty(\Omega)}$, and use the obvious embedding $H^4_*(\Omega)\hookrightarrow H^2(\Omega) \hookrightarrow L^\infty(\Omega)$ (see Lemma~\ref{lemma-Udot}). But for the nonlinear problem, a finer estimate is needed to have the convergence of the Picard iteration scheme, which requires $m\ge 8$.
\end{remark}

\begin{remark}[Angle assumptions]\label{rmk-2} 
(1) The smallness assumption of angles $\omega_n\in (0, \frac{\pi}{[\frac{m}{2}]})$ $(1\le n\le N)$ required in Theorem~\ref{main-thm} is necessary for the regularity persistence of the associated linearized problem of \eqref{MHD}-\eqref{bcd} in the anisotropic Sobolev spaces $H^m_*$. 
To explain this issue, 
without loss of generality,  assume that $x=0$ is a corner point on $\partial\Omega$, whose angle is $\omega \in (0,\pi)$, and $\Omega$ coincides with $\Omega_1 = \{x\in \R^2, 0< r< r_0, 0< \theta < \omega\}$ in the small neighborhood of the corner.

When $\omega \ge \frac{\pi}{[\frac{m}{2}]}$ with $\frac{\pi}{\omega} = n \in \mathbb{N}$, then $2n \le 2[\frac{m}{2}]\le m$.
As $n\ge 3$, set $f = r^{n}(\ln r \cos(n \theta) - \theta \sin (n \theta)) - \lambda x_2^n$, with $\lambda = \frac{\pi}{n(\sin \omega)^{n-1}\cos \omega}$, then $\Delta f = - n(n-1) \lambda x_2^{n-2} \in C^\infty([0,T]\times \overline{\Omega_1})$ for any fixed $T>0$, and $\partial_\theta f = 0$ on $\{\theta = 0, \omega\}$.
Let $\p(t,x) = t\cdot f(x)$, $u(t,x) = -\frac{1}{2}t^2 \nabla f(x)$, obviously they satisfy the following problem
\begin{equation}\label{1.7}
\begin{cases}
    \partial_t u + \nabla \p = 0, \quad \partial_t \p + \nabla \cdot u = f - \frac{1}{2}t^2 \Delta f,\\
   (u,\p)_{t=0} = 0, \quad u\cdot \nu = 0\quad  {\rm on}~ (\partial\Omega_1 \backslash\{x=0\}) \cap \{r<r_0\},
    \end{cases}
\end{equation}  
with $f-\frac{1}{2}t^2 \Delta f \in (H^{2n}_*\cap H^{2n+1}_*)((0,T)\times \Omega_1)$.
However, $u \not\in H^{n}((0,T)\times \Omega_1)$, which implies $u \not\in H^{2n}_*((0,T)\times \Omega_1)$.
As $n = 2$, let $f = r^{4}\left(\ln r \cos(4 \theta) - \theta \sin (4 \theta)\right) - 2\pi x_1 x_2^3$, and construct $u, \p$ in the same way as the case $n\ge 3$. Then, they satisfy the same problem as given in \eqref{1.7} with $f-\frac{1}{2}t^2 \Delta f \in (H^8_*\cap H^9_*)((0,T)\times \Omega_1)$. However, $u \not\in H^{4}([0,T]\times \Omega_1)$ and hence $u \not\in H^{8}_*((0,T)\times \Omega_1)$.

On the other hand, when $\omega \ge \frac{\pi}{[\frac{m}{2}]}$ with $\frac{\pi}{\omega}\not\in \mathbb{N}$, then $2([\frac{\pi}{\omega}]+1) \le 2[\frac{m}{2}] \le m$. 
Set $f = r^{\frac{\pi}{\omega}} \cos(\frac{\pi \theta}{\omega})\in (H^{2([\frac{\pi}{\omega}]+1)}_*\cap H^{2([\frac{\pi}{\omega}]+1)+1}_*)((0,T)\times \Omega_1)$ for any fixed $T>0$, then $\Delta f = 0$, and $\partial_\theta f = 0$ on $\{\theta = 0, \omega\}$. Let $\p(t,x) = t\cdot f(x)$ and $u(t,x) = -\frac{1}{2}t^2 \nabla f(x)$, then they also satisfy the problem \eqref{1.7} with $f-\frac{1}{2}t^2 \Delta f$ replaced by $f$,
however, it is easy to have $u \not\in H^{[\frac{\pi}{\omega}]+1}((0,T)\times \Omega_1)$, which certainly implies $u \not\in H^{2([\frac{\pi}{\omega}]+1)}_*((0,T)\times \Omega_1)$. 

(2) As noted in Remark~\ref{rmk-1}, we take the initial data in $H^m(\Omega)$ instead of in $H^m_*(\Omega)$ to simplify the formulation of compatibility conditions. For technical reasons, we shall additionally require that $\omega \ne \frac{\pi}{[\frac{m}{2}]+1}, \ldots, \frac{\pi}{m}$, which will be used in the proof of Proposition~\ref{prop-lift-cpbt} given in Appendix~A.

\end{remark}

Now, let us explain certain ideas to obtain the main result stated in Theorem~\ref{main-thm} before the end of this section.
    To get the estimate of normal derivatives of solutions, one needs to exploit the Div-Curl structure contained in the equations \eqref{MHD}. Inspired by the approach given by Wang and Zhang in \cite{wang2025incompressible}, first, from \eqref{MHD} one sees that $\nabla p$ and $\nabla\cdot u$ can be represented in terms of tangential derivatives of unknowns (namely, $\partial_t, u\cdot\nabla$, and $b\cdot \nabla$, as $u,b$ satisfy the boundary condition \eqref{bcd}), while $\nabla\cdot b$ satisfies a  transport equation which can be estimated by the initial data directly. As for $\nabla\times u$ and $\nabla\times b$, a careful use of the equations' structure allows some cancellation, yet remaining a non-eliminatable term, $\frac{1}{Q}b\times (R(\partial_t + u\cdot\nabla)^2 u)$ (see the last term of $G_3$ given in \eqref{normal-sys}), which arises as a source term in the equation of $\nabla\times b$.
    Consequently, from the perspective of energy estimates, the anisotropic spaces $H^m_*$ are particularly suitable, as controlling one order normal derivative actually needs two order tangential derivatives.

    To close the a priori estimates, one is left to control the pure tangential derivatives of $u,b$, and $p$. This step presents a technical challenge peculiar to the corner domains: the equations of high-order tangential derivatives introduce high-order commutators involving normal derivatives, preventing the energy estimates from being closed. Specifically, when the spatial derivatives act on $u$, they cannot be directly transformed into tangential derivatives of the unknowns via the equations, as in which only $\nabla \cdot u$ is represented by the tangential derivatives of the solution. Fortunately, this problematic term appears in the energy inequality as an inner product with the derivative of $p$. 
    By ``borrowing" a multiplier involved in the conormal derivatives of $p$, the normal derivative of $u$ can be indeed shifted onto $p$ (see Lemma \ref{lemma-J2} for the detail), and the estimates can be closed by using the equations.

The corners of domains bring another major challenge in establishing the existence of solutions. Unlike smooth domains where standard mollification and non-characteristic regularization can be applied, the presence of ``too many” normal directions in corner domains prevents these techniques. 
To address this, we first construct weak solutions in tangential spaces via the Lax–Phillips duality method in an abstract Hilbert space, rewriting the Green formula in terms of inner products despite the complicated high-order commutators. Then we carefully perform smoothing while preserving the traces in steps to establish the ``weak=strong" property. 
Finally, we upgrade the high-order normal regularity by coupling the equations for $\nabla\times u, \nabla\cdot b, \nabla\times b$ into a transport system, and using a Helmholtz-type decomposition to express the gradients in terms of divergences and curls.

    The remainder of this paper is organized as follows. In Section~2, we establish the a priori energy estimates for the linearized problem of \eqref{MHD}-\eqref{bcd} around a smooth background state. The well-posedness of this linearized problem was obtained in Section~3. Then, in Section~4, we extend the well-posedness result to the linearized problem with a non-smooth background state, and use the Picard iteration to prove the main result given in Theorem~\ref{main-thm} in Section~5. Finally, in the Appendices, certain supplementary technical results are given for completeness.

For convenience, we list the following notations, which shall be used throughout this paper.
\begin{itemize}
    \item Denote by $\partial_1 = \partial_{x_1}, \partial_2 = \partial_{x_2}$, $\nabla = \nabla_x = (\partial_1,\partial_2)^t$.
    \item For any two $2$-dimensional vector fields $u= (u_1,u_2)^t$ and $v = (v_1,v_2)^t$, denote by $u\times v = u^\perp\cdot v = u_2 v_1 - u_1 v_2$.
    \item By $\inp*{f}{g}_{\Omega}$ (resp. $\inp*{f}{g}_{\partial\Omega}$) or $\inp*{f}{g}_{L^2}$ (resp. $\inp*{f}{g}_{L^2(\partial\Omega)}$), we denote the inner product between $f$ and $g$ in $L^2(\Omega)$ (resp. $L^2(\partial\Omega)$). 
    \item For any Banach spaces $X, Y$, $\mathcal{L}(X)$ (resp. $\mathcal{L}(X\to Y)$) represents the space of bounded linear mappings from $X$ to itself (resp. $Y$), $\|\cdot\|_{\mathcal{L} (X)}$ (resp. $\|\cdot\|_{\mathcal{L} (X\to Y)}$) is the canonical operator norm.
\end{itemize}

    \section{A priori  estimates for the linearized MHD equations}\label{section-linear}
    \subsection{Preliminaries}
    \subsubsection{Properties of the anisotropic Sobolev spaces $H^m_*$}

Let  $H^m_*$ be the 
anisotropic spaces defined in Definition~\ref{def-sp}, and $\Omega$ be as introduced in Notation 1.1. Several properties of the function space $H^m_*$ will be stated in the following lemmas.
By the construction of $w^1(x),w^2(x)$ given in Definition~\ref{def-tan}, and  the definition of $H^m_*(\Omega)$ in \eqref{aniso-space}, we know as in Remark~\ref{rmk-sp}, that the space $H^m_*(\Omega)$ locally reduces to a standard anisotropic Sobolev space away from the corners, and near each corner, it reduces to the case $\Omega = \{x\in \R^2: x_1>0,x_2>0\}$, for which $\eqref{norm-sp}$ becomes
    \begin{equation*}
        \|f\|^2_{H^m_*(\Omega)} = \sum_{|\alpha|+2|\beta|\le m} \|x_1^{\alpha_1} x_2^{\alpha_2} \partial_1^{\alpha_1+\beta_1} \partial_2^{\alpha_2+\beta_2} f\|^2_{L^2(\Omega)}.
    \end{equation*}

    \begin{lemma}\label{lemma-imbed} The following continuous embeddings
$$H^m(\Omega)\hookrightarrow H^m_*(\Omega)\hookrightarrow H^{[\frac{m}{2}]}(\Omega), \quad 
        H^4_*(\Omega)\hookrightarrow L^\infty(\Omega),$$ 
        hold.
    \end{lemma}
    \begin{remark}
    The above continuous embeddings are obvious from the definition \eqref{aniso-space}. The last embedding is sharp, in the sense that $H^3_*(\Omega) \not\hookrightarrow L^\infty(\Omega)$. Indeed, in the case $\Omega = \{x \in \R^2: x_1>0, x_2>0\}$, and $H^3_*(\Omega) = \{f\in L^2(\Omega): x_1^{\alpha_1} x_2^{\alpha_2}\partial_1^{\alpha_1+\beta_1} \partial_2^{\alpha_2+\beta_2} f \in L^2, \forall |\alpha|+2|\beta|\le 3 \}$,
    letting $f(x) = \chi(|x|) \log (\log \frac{1}{|x|})$, with $\chi$ being a smooth cut-off function satisfying $\chi(|x|)\equiv 1 $ if $|x|\le \frac{1}{2}$ and $\chi(|x|) = 0$ if $|x|\ge 1$, one has $f\in H^3_*(\Omega)\backslash L^\infty(\Omega)$ obviously.
    \end{remark}

    Below, we need the following notations.
    \begin{nota}\label{nota-2.1}
        (1) For $m\in \mathbb{N}$, define the space
            $$X^m_*([0,T]\times \Omega):=\cap_{j=1}^m C^k([0,T]; H_*^{m-j}(\Omega)) $$
            with its canonical norm. 
            
        (2) Denote by
            \begin{equation*}
                \normmm{f(t)}_{m,*}^2 = \sum_{k=0}^m \|\partial_t^k f(t)\|^2_{H^{m-k}_*(\Omega)},
            \end{equation*}
            and $\|f(t)\|_m = \|f(t)\|_{H^m(\Omega)}$.
            
        (3) For any multi-index $\alpha = (\alpha_0, \ldots, \alpha_4)\in \mathbb{N}^5$, denote by
 	\begin{equation*}
 	\partial_*^\alpha := \partial_t^{\alpha_0}\partial_1^{\alpha_1}\partial_2^{\alpha_2}(\partial_{w_1})^{\alpha_3} (\partial_{w_2})^{\alpha_4}, \quad |\alpha|_* := \alpha_0 + 2(\alpha_1+\alpha_2) + \alpha_3 + \alpha_4, 
 	\end{equation*}
 and $\partial_*^m:=\sum_{|\alpha|_*= m}\partial_*^\alpha$, $D_*^m:=\sum_{|\alpha|_*\le m}\partial_*^\alpha$ for a fixed integer $m\ge 0$.
    \end{nota}

    \begin{lemma}\label{lemma-density}
    For any fixed integer $m\ge 0$, one has 
        \begin{enumerate}
            \item[(1)] $C^\infty(\overline{\Omega})$ is dense in $H^m_*(\Omega)$;
            \item[(2)] $C^\infty([0,T]\times \overline{\Omega})$ is dense in 
    $X^m_*([0,T]\times \Omega)$.  
        \end{enumerate}
    \end{lemma}
    \begin{proof}
        Without loss of generality, one only needs to consider the case $\Omega = \{x\in \R^2: x_1>0,x_2>0\}$ by using certain transformation.
        The proof of (1) follows from \cite[Lemma~B.1]{ohno1995initial}, regarding $H^m_*(\Omega)$ as a weighted Sobolev space with norm
        \begin{equation*}
            \|f\|^2_{H^m_*(\Omega)} = \sum_{|\alpha|\le m} \int_{\Omega} |\partial_1^{\alpha_1}\partial_2^{\alpha_2} f|^2 
            \sigma^2_\alpha(x)
            \d x,
        \end{equation*}
        where the weight function
        \begin{equation*}
            \sigma_\alpha(x) = \sum_{\substack{(2|\alpha|-m)_+ \le |\beta|,\, \beta \le \alpha}} x_1^{\beta_1}x_2^{\beta_2}
        \end{equation*}
        is a polynomial in the distances $x_1,x_2$ from $x$ to the boundaries $\{x_1= 0,x_2>0\}$ and $\{x_1>0,x_2=0\}$, respectively. 
        Observe that along the inward shift $v=(1,1)$, $\sigma_\alpha(x+\lambda v) > \sigma_\alpha(x)$ for $\lambda>0$, then the conclusion follows by adapting the argument in \cite[Theorem~7.2]{kufner1985weighted}.
    \end{proof}

    \begin{lemma}\label{lemma-moser}
        Let 
        $m \ge 0$
        be an integer and $r = \max\{m, 5\}$.
        There exists a constant $C>0$ such that the following holds:
        \begin{enumerate}
            \item[(1)] If $u \in H^m_*(\Omega)$ and $v\in H^r_*(\Omega)$, then $uv\in H^m_*(\Omega)$ and $\|uv\|_{H^m_*(\Omega)}\le C\|u\|_{H^{m}_*(\Omega)} \|v\|_{H^r_*(\Omega)}$,
            \item[(2)] If $u \in X^m_*([0,T]\times \Omega)$ and $v\in X^r_*([0,T];\Omega)$, then $uv\in X^m_*([0,T]\times \Omega)$ and for every $t\in [0,T]$, $\normmm{(uv)(t)}_{m,*}\le C\normmm{u(t)}_{m,*} \normmm{v(t)}_{r,*}$.
        \end{enumerate}
    \end{lemma}
    \begin{proof}
        The proof follows from a careful use of classical Sobolev embeddings, see \cite[Lemma~A.5]{ohno1995initial}.
    \end{proof}
    
\begin{coro}\label{coro-moser}
    Assume that $m\in \mathbb{N}$ is a fixed integer satisfying the conditions given in each point below. There is a constant $C>0$, such that 
    for any fixed $f_j,f,g \in X^m_*([0,T]\times \Omega)$ and $\alpha^j, \alpha\in \mathbb{N}^5$, with $j=1,\ldots,d$, the following inequalities hold: 
    \begin{enumerate}
        \item[(1)] $\|\prod_{1\le j\le d} \partial_*^{\alpha^j} f_j(t)\|_{0} \le C \prod_{1\le j\le d}\normmm{f_j(t)}_{m,*}$, as $\sum_{1\le j\le d} |\alpha^j|\le m$ and $m\ge 5$.
        \item[(2)] $\|\prod_{1\le j\le d} \partial_*^{\alpha^j} f_j(t)\|_0 \le C \sum_{1\le j\le d}(\normmm{f_j(t)}_{m,*}\prod_{1\le l\le d,l\ne j}\normmm{f_l(t)}_{m-1,*})$, \\ as $\sum_{1\le j\le d} |\alpha^j|\le m$ and $m\ge 6$.
        \item[(3)] $\|\partial_*^\alpha (f g)(t) - f\partial_*^\alpha g(t)\|_0\le C \normmm{\partial_* f(t)}_{m-1,*}\normmm{g(t)}_{m-1,*}$, as $|\alpha|\le m$ and $m\ge 6$.
        \item[(4)] $\|\partial_*^\alpha (f g)(t) - f\partial_*^\alpha g(t)\|_0\le C (\normmm{\partial_* f(t)}_{m-1,*}\normmm{g(t)}_{m-2,*}+\normmm{\partial_* f(t)}_{m-2,*}\normmm{g(t)}_{m-1,*})$, as $|\alpha|\le m$ and $m\ge 7$.
 \item[(5)]  For any smooth function $F(\cdot)$, it holds that $\normmm{F(f(t))}_{m,*} \le C_{F}(\|f(t)\|_{L^\infty}) (1+\normmm{f(t)}_{m,*}^m)$, as $m\ge 5$.  
    \end{enumerate}
\end{coro} 
    
    \subsubsection{Formulation of linearized problems}

    For a given smooth background state $Z = (U,B,P,S)$ in $(0,T)\times  \Omega$ with $ U = (U_1,U_2)^t, B= (B_1,B_2)^t$, satisfying 
    \begin{equation}\label{assum-Z}
    \begin{aligned}
         & U\cdot \nu = B \cdot \nu = 0 \text{ on } (0,T)\times (\partial\Omega\backslash \{q_1, \ldots, q_N\}),\\
         & (U,B,P,S) \text{ takes its values in a fixed compact subset }\mathcal{K} \text{ of }\R^4\times \mathcal{V},
    \end{aligned}
    \end{equation}
with $\nu$ being the unit outward normal vector on $\partial\Omega\backslash \{q_1, \ldots, q_N\}$, $\mathcal{V}$ being some open subset of $\R^+\times \R$, the linearized problem of \eqref{MHD}-\eqref{bcd} usually would take the form
    \begin{equation}\label{linear-formal}
    \begin{cases}
    R(P,S)(\partial_t u + U \cdot \nabla u) + \nabla \p - B\cdot \nabla b = F_1, \\
    \partial_t b + U \cdot \nabla b - B\cdot \nabla u + B(\nabla \cdot u) = F_2,\\
    \frac{1}{Q(P,S)} (\partial_t \p + U \cdot \nabla \p) - \frac{1}{Q(P,S)} B \cdot (\partial_t b + U \cdot \nabla b) +  \nabla \cdot u = F_3 ,\\
    \partial_t s + U \cdot \nabla s = F_4,
    \end{cases}
    \end{equation}
where
$\p = p + \frac{1}{2}|b|^2$
is the total pressure, and $F_1, F_2$ are $2$-dimensional vectors.

But one hopes to have 
\begin{equation}\label{b-bcd}
b\cdot \nu = 0,\quad \forall t>0, ~x\in \partial\Omega\backslash \{q_1, \ldots, q_N\}
\end{equation}
for the solution of the linear problem \eqref{linear-formal}, if it is true initially, provided that the source term satisfies
$$ F_2 \cdot \nu = 0, \quad  \forall t>0, ~x\in \partial\Omega\backslash \{q_1, \ldots, q_N\}.$$
This is also crucial to satisfy the perfectly conducting wall condition, $b\cdot\nu=0$ on the smooth part of $\partial\Omega$ for the solution of the nonlinear problem \eqref{MHD}-\eqref{bcd}, which will be constructed as the limit of approximate solution sequences via a Picard iteration of \eqref{MHD} in Section~5. However, by a direct calculation, one can see that the ansatz \eqref{b-bcd} does  not hold in general for the solution of the equations \eqref{linear-formal}, so one needs to modify this linearized system. 
    
    Recall that in the case where $\partial\Omega$ is smooth (see \cite{yanagisawa1991fixed} and \cite{secchi1995well}), the linearized evolution equation of $b$ given in \eqref{linear-formal} is replaced by the following one via adding a lower order term,
    \begin{equation*}
        \partial_t b + U\cdot \nabla b - B\cdot \nabla u + B(\nabla\cdot u) + \big[(U\cdot \nabla)\tilde{\nu}(x) \cdot b - (B\cdot \nabla )\tilde{\nu}(x)\cdot u\big] \tilde{\nu}(x) = 0,
    \end{equation*}
    where $\tilde{\nu}(x)$ is a smooth extension of $-\nabla\dist(x,\partial\Omega)$ from $x$ in a thin neighborhood of $\partial\Omega$ to $x\in \overline{\Omega}$. However, in our case, especially when $x$ is in a small neighborhood of a corner, $q_n$ for some $n$, the distance function $\dist (x,\partial\Omega)$ is non-smooth.  To resolve this, first we construct two matrices in the following lemma.
    
    \begin{lemma}\label{lemma-h-def}
    There are two $2\times 2$ matrices $h^i\in C^\infty(\R^2)$, $i=1,2$, such that for any vector $v = (v_1,v_2)^t$, it holds 
        \begin{equation}\label{h}
        (U\cdot \nabla \nu)\cdot v = (\sum_{i=1,2} U_i h^i v)\cdot \nu,\quad {\rm on} \quad \partial\Omega\backslash\{q_1,\ldots, q_N\} 	
        \end{equation}
    \end{lemma}

\begin{proof}
  The construction of $h^i$ is inspired by \cite{godin20212d}. It is enough to construct them locally in $\R^2$ near each point $\Tilde{x}\in \partial\Omega$,  then use a $C^\infty$ partition of unity and an extension to $\R^2$. 
  
 For any  fixed $\tilde{x}\in \partial\Omega\backslash\{q_1, \ldots, q_N\}$, let $\Phi(x)$ and $\tilde{\nu} = -\nabla \Phi$ be the function and extension of the outward normal vector at $\partial\Omega$ in a small neighborhood of $\tilde{x}$, $V(\tilde{x})$, defined in Definition~\ref{def-tan}(3), and when   $\tilde{x}$ is one of the corner points, $q_n$ for some $1\le n\le N$, let $\Phi_l(x)~(l=1,2)$ and $\tilde{\nu}^l = -\nabla \Phi_l$ be the functions and extensions of the outward normal vectors to the legs $\gamma_l^n$ of the corner $q_n$ in a small neighborhood of $q_n$, $V(q_n)$, defined in Definition~\ref{def-tan}(4). 
  
To have \eqref{h}, it is enough to find matrices $h^i$ so that 
\begin{equation}\label{h-1}
\begin{cases}
(h^i)^t \tilde{\nu} = \partial_i \tilde{\nu} \quad {\rm on}~ \partial\Omega\cap V(\tilde{x}), \quad {\rm as}
~ \tilde{x}\in \partial\Omega\backslash\{q_1,\ldots,q_N\},\\
(h^i)^t \tilde{\nu}^l = \partial_i \tilde{\nu}^l ~(l=1,2)\quad
{\rm  on}~ \partial\Omega\cap V(q_n), \; {\rm as} \; \tilde{x} = q_n~{\rm for}~1\le n\le N.
\end{cases}
\end{equation}

Obviously, from \eqref{h-1} we know that $h^i ~(i=1,2)$ are uniquely determined in each neighborhood of $V(q_n)$ for all $1\le n\le N$, and they exist, but not uniquely when $\tilde{x}\in \partial\Omega\backslash\{q_1,\ldots,q_N\}$, for example one can choose $(h^i)_{1j} = (\partial_1 \Phi)^{-1} \partial_i\partial_j \Phi(x)$ and $(h^i)_{2j} = 0$ as $l=1,2$, when $\partial_1\Phi(\tilde{x})\neq 0$.
\end{proof}

Note that the above construction of $h^i~(i=1,2)$ is independent of $U$ and the vector $v$ given in \eqref{h}. Let 
\begin{equation}\label{2.6}
h(x,U) = \sum_{i=1,2}U_ih^i, \quad h(x,B) = \sum_{i=1,2}B_i h^i.
\end{equation}

Observe that $h(x,u)b = h(x,b)u$ because $\tilde{\nu},\tilde{\nu}^1,\tilde{\nu}^2$ are gradients of smooth functions.
	
Now, we revise the linearized system \eqref{linear-formal} as the following one for $z = (u,b,\p,s)^t$,
    \begin{equation}\label{linear}
        \begin{cases}
            R(P,S)(\partial_t u + U \cdot \nabla u) + \nabla \p - B\cdot \nabla b = F_1, \\
            \partial_t b + U \cdot \nabla b - B\cdot \nabla u + B(\nabla \cdot u) + \left(h(x,U)b - h(x,B)u\right)= F_2,\\
            \frac{1}{Q(P,S)}(\partial_t \p + U \cdot \nabla \p) - \frac{1}{Q(P,S)} B \cdot (\partial_t b + U \cdot \nabla b) + \nabla \cdot u = F_3 ,\\
            \partial_t s + U \cdot \nabla s = F_4,
        \end{cases}
    \end{equation}
    with boundary conditions
    \begin{equation}\label{bcd-linear}
        u \cdot \nu = 0, \quad \forall t>0, ~x\in \partial\Omega \backslash\{q_1, \ldots, q_N\},
    \end{equation}
    and initial data
    \begin{equation}\label{icd-linear}
    \begin{aligned}
        z|_{t=0} = z_0 = (u_0,b_0,p_0,s_0)^t,
    \end{aligned}
    \end{equation}
    with
    \begin{equation}\label{b-icd}
        b_0\cdot \nu = 0 \text{ on } \partial\Omega\backslash\{q_1,\ldots, q_N\}.
    \end{equation}
    Here the source term $F=(F_1,F_2,F_3,F_4)^t$ satisfies
    \begin{equation}\label{F-cdt}
        F_2 \cdot \nu = 0, \quad\forall t>0,~ x\in \partial\Omega\backslash\{q_1,\ldots,q_N\}.
    \end{equation}

It will be convenient to denote the boundary conditions
    \eqref{bcd-linear} and \eqref{b-bcd}
    by
\begin{equation}\label{2.11}
{\cal B} z=0, \quad {\rm on}\quad 
\partial\Omega\backslash\{q_1,\ldots,q_N\}
    \end{equation}
 and   \begin{equation}\label{2.12}
{\cal B}' z=0, \quad {\rm on}\quad 
\partial\Omega\backslash\{q_1,\ldots,q_N\}
    \end{equation}
respectively.

   For the above problem \eqref{linear}-\eqref{icd-linear}, we assume that for some $m \in \mathbb{N}\backslash\{0\}$,
    \begin{equation}\label{z0F-cdt}
    \begin{aligned}
        &z_0\in H^{m}(\Omega), F\in H^{m}((0,T)\times \Omega),\\ & z_0,F\text{ satisfy the compatiblity}~ \text{conditions up to order }m-1,
    \end{aligned}
    \end{equation}
    where the compatiblity condition of order $j~(0\le j\le m-1)$ is naturally defined as: 
    $$u_{(j)}\cdot \nu = 0, \quad {\rm on}~\partial\Omega \backslash\{q_1,\ldots,q_N\}$$ 
    with $u_{(0)}(x) = u(0,x)$, and  for $j\ge 1$, $u_{(j)}(x) = \partial_t^j u(0,x)$
    being computed from \eqref{linear} and \eqref{icd-linear} in terms of components of $\partial_x^\alpha z_0(x)$ ($|\alpha|\le j$), and $\partial^\beta (Q(P,S))(0,x)$, $\partial^\beta (R(P,S))(0,x)$, $\partial^\beta U(0,x)$, $\partial^\beta B(0,x)$, $\partial^\beta F(0,x)$ ($|\beta|\le j-1$).
  
  With the above modified linear equations \eqref{linear}, the ansatz \eqref{b-bcd} is true as given in the following result.
  
  \begin{lemma} Assume that the problem \eqref{linear}-\eqref{icd-linear} has a classical solution $z=(u, b, p, s)^t$ for $0<t<T$, then under the
 assumptions \eqref{b-icd}-\eqref{F-cdt}, it holds     
 \begin{equation}\label{b-bcd-1}
 b \cdot \nu = 0, \quad\forall 0<t<T,~ x\in \partial\Omega\backslash\{q_1,\ldots,q_N\}.
 \end{equation}
\end{lemma}

\begin{proof}
	By taking the restriction of the second equation of \eqref{linear} on $\partial\Omega\backslash\{q_1,\ldots,q_N\}$ and the inner product with $\nu(x)$ it gives
    \begin{equation}\label{bcd-b-eq}
        \partial_t (b\cdot \nu) + U\cdot \nabla (b\cdot \nu) = 0 \quad \text{on } \partial\Omega\backslash\{q_1,\ldots,q_N\}.
    \end{equation}
  From that $U\cdot \nabla$ is a tangential derivative of $\partial\Omega\backslash\{q_1,\ldots,q_N\}$ by using $U\cdot \nu = 0$, it implies immediately $b\cdot \nu = 0$ as $0<t<T$, $x\in \partial\Omega\backslash\{q_1,\ldots,q_N\}$ when it is true at $t=0$..
\end{proof}

From the assumption \eqref{assum-Z} of the background state, we know that both of $U\cdot\nabla, B\cdot\nabla$ are tangential derivatives along $\partial\Omega\setminus\{q_1,\ldots,q_N\}$, now we state a result relating them with the directional derivative  $\partial_{w^i}=w^i(x)\cdot \nabla$ given in the Definition~\ref{def-sp} of the space $H^m_*$.

    \begin{lemma}\label{lemma-tangential}
        Under the assumption \eqref{assum-Z}, there exists a smooth vector field $V = (V_1, V_2)^t$ in $[0,T]\times \overline{\Omega}$, such that $U\cdot \nabla = V_1 \partial_{w^1} + V_2 \partial_{w^2}$, and for any nonnegative integer $k$, $\normmm{V(t)}_{k,*} \le C\normmm{U(t)}_{k+2,*}$, where $C$ is a positive constant independent of $U$. 
        The same result holds for the vector field $B\cdot \nabla$, when we write $B\cdot \nabla = \bar{B}_1 \partial_{w^1} + \bar{B}_2 \partial_{w^2}$.
    \end{lemma}
    \begin{proof}
        It suffices to verify this result locally near each point $\tilde{x}\in \partial\Omega$. 
        
        (1) For any fixed $\tilde{x} \in \partial\Omega \backslash\{q_1,\ldots, q_N\}$, let $\Phi(x)$ be the extension of ${\rm dist}(x, \partial\Omega)$ in a small neighborhood of $V(\tilde{x})$ as given in Notation~\ref{nota-1}(3). As $\nabla\Phi(x)\neq 0$, without loss of generality, suppose $\partial_1\Phi(x) \ne 0$ in $V(\tilde{x})\cap \Omega$. Introduce the invertible transform
        $$x^*_1 = \Phi(x), \quad x^*_2 = x_2,$$
        then, obviously, one has
        \begin{equation}\label{U-1}
         U(t,x)\cdot \nabla = U^*(t,x^*)\cdot \nabla^*,
        \end{equation}       
        where $\nabla^* = (\partial_{x^*_1}, \partial_{x^*_2})$, and 
         \begin{equation}\label{U-2}
        U^*_1(t,x^*) = \partial_1\Phi(x) U_1(t,x) + \partial_2\Phi(x) U_2(t,x), \quad U^*_2(t,x^*) = U_2(t,x).
    \end{equation}

         On the other hand, from the definition of the vector fields $w^1(x)$ and $w^2(x)$ given in Definition~\ref{def-tan}, it is easy to have
         \begin{equation}\label{U-3}
       (\partial_1 \Phi(x))^{-1}w^1(x)\cdot\nabla = x_1^* \partial_1^*,\quad   (\partial_1 \Phi(x))^{-1} w^2(x) \cdot\nabla = - \partial_2^*.
       \end{equation}    
         
         Noting that the assumption $U\cdot \nu=0$ implies $U_1^* = 0$ when $x_1^* = 0$, thus there exists a smooth $\overline{U}_1^*(t,x^*)$, such that $U^*_1(t,x^*) = x_1^* \overline{U}_1^*(t,x^*)$. Plugging \eqref{U-2}, \eqref{U-3} into \eqref{U-1}, it gives immediately the conclusion,
         $$U\cdot \nabla = V_1 \partial_{w^1} + V_2 \partial_{w^2}, $$
         with $V_1(t,x) = (\partial_1 \Phi(x))^{-1} \overline{U}_1^*(t,\Phi(x),x_2)$ and $V_2(t,x) = -(\partial_1 \Phi(x))^{-1} U^*_2(t,\Phi(x),x_2)$ satisfying the estimate claimed in this lemma.
        
(2) When $\tilde{x} = q_n$ for some $1\le n\le N$, let $\Phi_l(x)~(l=1,2)$ be the two extensions of ${\rm dist}(x, \gamma_l^n)$ in a small neighborhood of $V(\tilde{x})$ as given in Notation~\ref{nota-1}(4). Introduce the transform
$$x_1^* = \Phi_1(x), \quad x_2^* = \Phi_2(x)$$ 
and 
$$U_l^*(t,x^*) = \partial_1 \Phi_l(x) U_1(t,x) + \partial_2 \Phi_l(x) U_2(t,x), \quad  l=1,2.$$ 

By a direct calculation, one has
\begin{equation}\label{U-4}
U(t,x)\cdot \nabla = U^*(t,x^*)\cdot \nabla^*,
\end{equation}
with $\nabla^* = (\partial_{x^*_1}, \partial_{x_2^*})$, and
\begin{equation}\label{U-5}
w^l\cdot\nabla=(\nabla\Phi_1\times \nabla\Phi_2) x_l^*\partial_{x_l^*}, \quad l=1,2.
\end{equation}

On the other hand, by using $U\cdot \nu=0$ on $(\gamma_1^n\cup\gamma_2^n)\setminus\{q^n\}$, we have
$$U_i^* = 0, \quad {\rm as ~} x^*_i= 0~(i=1,2),$$ 
so there exist  smooth $\overline{U}_l^*(t,x^*)~(l=1,2)$, such that $U^*_l(t,x^*) = x_l^* \overline{U}_l^*(t,x^*)$. Therefore, from \eqref{U-4} and \eqref{U-5}, we conclude that
$$U(t,x)\cdot\nabla=(\nabla\Phi_1\times\nabla\Phi_2)^{-1}\left(\overline{U}_1^*(t, \Phi_1(x), \Phi_2(x))\partial_{w^1}+\overline{U}_2^*(t, \Phi_1(x),\Phi_2(x))\partial_{w^2}\right)$$
by noting $\nabla\Phi_1\times \nabla\Phi_2\neq 0$ in a small neighborhood of $q^n$ from $\omega_n \in (0,\pi)$, and satisfying the estimate for $V_l(t,x)=(\nabla\Phi_1\times\nabla\Phi_2)^{-1}\overline{U}_l^*(t, \Phi_1(x), \Phi_2(x))
~(l=1,2)$, claimed in the lemma.
    \end{proof}

Now, let us give a remark on the equivalence of the norms of the anisotropic Sobolev space $H^m_*(\Omega)$ if one changes the order of derivatives in \eqref{norm-sp}, by using the above lemma.

 \begin{remark}\label{remark2.1}
 Although $\partial_{w^1}, \partial_{w^2}$ and $\partial_1, \partial_2$ do not commute, nor do $\partial_{w^1}, \partial_{w^2}$ commute themselves, but their commutators only involve derivatives of strictly lower order, by noting the identity
 \begin{equation}\label{tan-cmtt-1}
 [\partial_{w}^k, \partial_x] = \sum_{i=1,2}\sum_{0\le l\le k-1} a^i_{k,l}(x) \partial_i \partial_{w}^l = \sum_{i=1,2}\sum_{0\le l\le k-1} b^i_{k,l}(x)  \partial_{w}^l\partial_i,
 \end{equation}
 which can be derived by induction on the order $k$,
 where $\partial_w = \partial_{w^1}$ or $\partial_{w^2}$, $\partial_{x} = \partial_{x_1}$ or $\partial_{x_2}$, and $a^i_{k,l},b^i_{k,l}$ are smooth functions. On the other hand, it is straightforward to have $[\partial_{w^1},\partial_{w^2}] = v \cdot \nabla$, where $v = w^1\cdot \nabla w^2 - w^2 \cdot \nabla w^1$. From the construction of $w^1,w^2$ given in Definition~\ref{def-tan}, we have  
 $v(x)\cdot \nu(x) = 0$ as $x\in \partial\Omega \backslash\{q_1,\ldots,q_N\}$. Thus, by using Lemma~\ref{lemma-tangential} and induction, we get
 \begin{equation}\label{tan-cmtt-2}
 \begin{aligned}
 [\partial_{w^1}^k, \partial_{w^2}^l] = \sum_{0\le (k',l')< (k,l)} c_{k,k',l,l'}(x) \partial_{w^1}^{k'}\partial_{w^2}^{l'},
 \end{aligned}
 \end{equation}
 where $c^i_{k,k',l,l'}(x)$ are smooth functions. Therefore, the norms of the anisotropic Sobolev space $H^m_*(\Omega)$ defined in \eqref{norm-sp} are equivalent to each other if one changes the order of derivatives contained in \eqref{norm-sp}.

 \end{remark}

 \subsection{A priori estimates of solutions to the linearized problem}

 The main goal of this subsection is to establish the a priori estimates for the linearized problem \eqref{linear}-\eqref{icd-linear} under the assumptions \eqref{F-cdt} and \eqref{z0F-cdt}. Before that, let us give some notations first:
 
 \begin{nota}\label{nota-2.2}
For the problem \eqref{linear}-\eqref{icd-linear}, denote by
 $$M_m(Z,t)=\sup_{0\le \tau\le t}\normmm{Z(\tau)}_{m,*},$$ 
 and 
 $$\Lambda_m(z,F,t) = \normmm{z(0)}_{m,*}+\normmm{F(0)}_{m-1,*}+\int_0^t\normmm{F(\tau)}_{m,*}\d \tau$$
with the notation $\normmm{f(t)}_{m,*}$ being defined in Notation \ref{nota-2.1}, where $Z=(U, B, P, S)^t$ and $z=(u, b, p, s)^t$ represent the background state and solution respectively, and $F=(F_1, \ldots, F_4)^t$ is the source term.
\end{nota}

For $m \in \mathbb{N}$, we say that  the compatibility conditions of the problem \eqref{linear}-\eqref{icd-linear} up to order $m$  mean, at $t = 0$ and $x\in \partial\Omega \backslash\{q_1, \ldots, q_N\}$,
\begin{equation}\label{cptb}
b\cdot \nu = 0, \quad \partial_t^j u\cdot \nu = 0, \quad 0\le j\le m.
\end{equation}
It should be noted that $\partial_t^j b \cdot \nu = 0$ on $\{t=0\}\times (\partial\Omega \backslash\{q_1, \ldots, q_N\})$ also holds for $1\le j\le m$ when \eqref{cptb} is true.

  The main result of this subsection is the following one.
    
    \begin{theorem}\label{thm-apriori}
       For a fixed integer $m\ge 7$, assume that the measure of angle at each corner $q_n\in \partial\Omega$ satisfies $\omega_n \in (0,\frac{\pi}{[\frac{m}{2}]})$ for $1\le n\le N$, the smooth background state $Z=(U, B, P, S)^t$ satisfies the hypothesis \eqref{assum-Z}, and the compatibility conditions up to order $m-1$ are satisfied for the problem \eqref{linear}-\eqref{icd-linear}. Then, there is a continuous increasing function $\phi:\overline{\R^+}\to \overline{\R^+}$ independent of $T$ and may depending on $\mathcal{K}$, such that the estimate
        \begin{equation}\label{apriori}
            \normmm{z(t)}_{m,*} \le \phi(M_m(Z,T))\Lambda_m(z,F,t)
        \end{equation}
      holds for $0\le t\le T$, and the solution $z=(u, b, p, s)^t\in X^{m+1}_* ([0,T];\Omega)$ satisfying $\nabla \p, \nabla\cdot u \in X^{m}_* ([0,T]; \Omega)$ of \eqref{bcd-linear}-\eqref{icd-linear}.
        Moreover, when $m\ge 8$, there are two  continuous increasing functions  $\phi_1,\phi_2:\overline{\R^+}\to \overline {\R^+}$ independent of $T$, such that for any $0\le t\le T$, the following one holds
        \begin{equation}\label{apriori-fine}
        \begin{aligned}
            \normmm{z(t)}_{m,*} \le & \phi_1(M_{m-1}(Z,T))\Big(\Lambda_m(z,F,t)\\&+
            \phi_2(M_m(Z,t))\int_0^t(\normmm{F(\tau)}_{m-1,*}+\Lambda_{m-1}(z,F,\tau))\d \tau\Big).
        \end{aligned}
        \end{equation}
       
    \end{theorem}

The strategy to prove Theorem~\ref{thm-apriori} is first to estimate the tangential derivatives of $z$ by the energy method, then, by exploring the special structure hidden in the system, we derive a divergence-curl system for $(u,b)$ to deduce that the terms involving a proper normal derivative can be controlled by the tangential derivatives. Finally, we close the proof by using Gronwall's inequality.

\subsubsection{Estimates of tangential derivatives}

Define tangential derivatives
\begin{equation*}
\mathcal{T}_0 = \partial_t,\, \mathcal{T}_1 = \partial_{w^1},\, \mathcal{T}_2 = \partial_{w^2},\, {\mathcal{T}}^\alpha = {\mathcal{T}_0}^{\alpha_0} {\mathcal{T}_1}^{\alpha_1} {\mathcal{T}_2}^{\alpha_2},
\end{equation*}
and
\begin{equation*}
\mathscr{T}_0 = \partial_t,\, \mathscr{T}_1 = \partial_{w^1} + h(x,w^1),\, \mathscr{T}_2 = \partial_{w^2} + h(x,w^2),\, {\mathscr{T}}^\alpha = {\mathscr{T}_0}^{\alpha_0} {\mathscr{T}_1}^{\alpha_1} {\mathscr{T}_2}^{\alpha_2},
\end{equation*}
where $h(x,w^j)=\sum_{i=1,2}w_i^jh^i$ ($j=1,2$) with $h^i$ being $2\times 2$ matrices given in Lemma \ref{lemma-h-def}. 
Obviously, $\mathscr{T}$ is a modified version of $\mathcal{T}$, and their difference $\mathscr{T}-\mathcal{T}$ only involves zero-th order terms.
The construction of $h(x,\cdot)$ shows that if $v = (v_1,v_2)\in H^2((0,T)\times \Omega)$ satisfies $v\cdot \nu = 0$ on $(0,T)\times \big(\partial\Omega\backslash\{q_1,\ldots,q_N\}\big)$, then $\mathscr{T} v \cdot \nu =0$ on $(0,T)\times \big(\partial\Omega\backslash\{q_1,\ldots,q_N\}\big)$.
Thus, fixing $m\ge 7$, we are driven to control (the equivalent norm)
\begin{equation}\label{fctnl}
\begin{aligned}
\normmm{z(t)}_{m,*} = \sum_{0\le k\le [\frac{m}{2}]}\sum_{\substack{\beta\in \mathbb{N}^3\\ |\beta|\le m-2k}}
\|\mathbb{T}^\beta z(t)\|_k,
\end{aligned}
\end{equation}
where by $\mathbb{T}^\beta$ we denote the differential operator
\begin{equation*}
    \mathbb{T}^\beta = \begin{pmatrix}
        \mathscr{T}^\beta {\bf I}_{4\times 4} &  \\
        & \mathcal{T}^\beta {\bf I}_{2\times 2}  
    \end{pmatrix},
\end{equation*}
and $\mathbb{T}^\beta z = (\mathscr{T}^\beta u, \mathscr{T}^\beta b, \mathcal{T}^\beta \p, \mathcal{T}^\beta s)^t$.

The system \eqref{linear} is symmetrizable. Indeed, if the system \eqref{linear} is written in the matrix form: 
\begin{equation}\label{linear-matrix}
    (L+\mathbb{B})(Z) z = F,
\end{equation}
with $L(Z) = A_0(Z)\partial_t + A_1(Z)\partial_1 + A_2(Z)\partial_2$, and $ \mathbb{B}(Z)z = \left(0,0,(h(x,U)b-h(x,B)u)^t,0,0\right)^t,$ then there exists a symmetrizer
\begin{equation*}
    S_0(Z) = \begin{pmatrix}
        {\bf I}_{2\times 2} & & \\
        & \tilde{S} &  \\
          & & 1
    \end{pmatrix},\quad {\rm with~}
\tilde{S} = \begin{pmatrix}
         1 & & -B_1  \\
         & 1  & -B_2  \\
        & & 1
 \end{pmatrix}   
\end{equation*}
and the system \eqref{linear} is equivalent to
\begin{equation}\label{symm}
    (\tilde{L}+\mathbb{B})(Z)z := (S_0 A_0) (Z) \partial_t z + (S_0 A_1) (Z) \partial_1 z + (S_0 A_2) (Z) \partial_2 z + \mathbb{B}(Z) z = \overline{F},
\end{equation}
where
    \begin{equation*}
        S_0 A_0(Z) = \begin{pmatrix}
            R{\bf I}_{2\times 2}  & & & & & \\[2mm]
              
               & 1 + \frac{1}{Q} B_1^2 & \frac{1}{Q} B_1 B_2 & -\frac{1}{Q} B_1 & \\[2mm]
              & \frac{1}{Q} B_1 B_2 & 1 + \frac{1}{Q} B_2^2 & -\frac{1}{Q} B_2 & \\[2mm]
               & -\frac{1}{Q} B_1 & - \frac{1}{Q} B_2 & \frac{1}{Q} & \\[2mm]
                & & & & 1
        \end{pmatrix},
    \end{equation*}
    is positively definite, and
    \begin{equation*}
        S_0 A_1(Z) = \begin{pmatrix}
            R U_1 & & -B_1 & & 1 & \\[2mm]
             & R U_1 & & -B_1 & & \\[2mm]
            -B_1 & & (1+\frac{1}{Q} B_1^2) U_1 & \frac{1}{Q} B_1 B_2 U_1 & -\frac{1}{Q} B_1 U_1 & \\[2mm]
             & -B_1 & \frac{1}{Q} B_1 B_2 U_1 & (1+ \frac{1}{Q} B_2^2)U_1 & -\frac{1}{Q} B_2 U_1 & \\[2mm]
            1 & & -\frac{1}{Q} B_1 U_1 & -\frac{1}{Q} B_2 U_1 & \frac{1}{Q} U_1 & \\
            & & & & & & U_1
        \end{pmatrix},
    \end{equation*}
    \begin{equation*}
        S_0 A_2(Z) = \begin{pmatrix}
            R U_2 & & -B_2 & & & \\
             & R U_2 & & -B_2 & 1 & \\[2mm]
            -B_2 & & (1+\frac{1}{Q} B_1^2) U_2 & \frac{1}{Q} B_1 B_2 U_2 & -\frac{1}{Q} B_1 U_2 & \\[2mm]
             & -B_2 & \frac{1}{Q} B_1 B_2 U_2 & (1+ \frac{1}{Q} B_2^2)U_2 & -\frac{1}{Q} B_2 U_2 & \\[2mm]
             & 1 & -\frac{1}{Q} B_1 U_2 & -\frac{1}{Q} B_2 U_2 & \frac{1}{Q} U_2 & \\[2mm]
            & & & & & & U_2
        \end{pmatrix},
    \end{equation*}
    are two symmetric matrices, 
    \begin{equation*}
        \overline{F} = (F_1,F_2- F_3 B, F_3, F_4)^t.
    \end{equation*}

    
The next lemma can be obtained by using a standard energy approach for the system \eqref{symm}.
\begin{lemma}\label{lemma-L2} Under the assumption of Theorem~\ref{thm-apriori}, there exists a polynomial $P(\cdot)$, such that the following estimate holds:
\begin{equation}\label{L2-est}
    \frac{\d}{\d t} \int_\Omega R|u|^2 + |b|^2 + \frac{1}{Q} |\p - b\cdot B|^2 + |s|^2\,\d x \le P(\normmm{Z(t)}_{6,*})\|{z}(t)\|_0(\|{z}(t)\|_0 + \|{F}(t)\|_0)
\end{equation}
for the solution $z=(u,b,\p,s)^t$ of the problem \eqref{linear}-\eqref{icd-linear}.
\end{lemma}

\begin{proof}
From the system \eqref{symm}, it is easily seen that \begin{equation*}
    \inp*{z}{S_0 A_0 z} = R|u|^2 + |b|^2 + \frac{1}{Q}|\p - B\cdot b|^2 \d x + |s|^2.
\end{equation*}
On the boundary, by using $U\cdot \nu = B\cdot \nu = 0$ on $\partial\Omega \backslash\{q_1,\ldots, q_N\}$, it gives
\begin{equation*}
    A_\nu : = S_0 A_1 \nu_1 + S_0 A_2 \nu_2 =\begin{pmatrix}
         0 & & & & \nu_1 & \\
         & 0 & & & \nu_2 & \\
         & & 0 & & & \\
         & & & 0 & & \\
         \nu_1 & \nu_2 & & & 0 & \\
         & & & & & & 0
    \end{pmatrix},
\end{equation*}
 which implies that $\inp*{z}{A_\nu z}_{\partial\Omega} = 2\inp*{\p}{u\cdot\nu}_{\partial\Omega} = 0$ by using the boundary condition $u\cdot \nu =0$ on $\partial\Omega \backslash\{q_1,\ldots, q_N\}$.

By taking the inner product $\inp*{z}{(\tilde{L}+\mathbb{B})(Z) z}_\Omega$ and integrating by parts, we have
\begin{equation}
\begin{aligned}
    \frac{\d}{2 \d t} \inp*{z}{(S_0 A_0)(Z) z}_{\Omega} & =  \inp*{z}{\overline{F}}_{\Omega} - \frac{1}{2}\inp*{z}{A_\nu z}_{\partial\Omega} - \inp*{z}{\mathbb{B} z}_{\Omega}\\
    & +\frac{1}{2}\inp*{z}{\left(\partial_t (S_0 A_0) + \partial_1(S_0 A_1) + \partial_2 (S_0 A_2)\right)z}_{\Omega}\\
    & \le \inp*{z}{\overline{F}}_{\Omega} + P(\|\partial_{t,x} Z\|_{L^\infty(\Omega)}) \|z\|_{0}^2.
\end{aligned}
\end{equation}
Here and in the sequel, $P(\cdot)$ represents some generic polynomial, and may differ from line to line.
Then, one deduces
\begin{equation}\label{L2-temp}
\begin{aligned}
    \frac{\d}{\d t} \int_\Omega R|u|^2 + |b|^2 + \frac{1}{Q} |\p - B\cdot b|^2 + |s|^2 \,\d x \le 
    P(\normmm{Z(t)}_{6,*})\|{z}(t)\|_0^2
    + \int_\Omega z \cdot \overline{F} \d x,
\end{aligned}
\end{equation}
for the solution $z=(u,b,\p,s)^t$ of the problem \eqref{linear}-\eqref{icd-linear}, which implies \eqref{L2-est} immediately.
\end{proof}

Now we estimate the higher-order tangential derivatives of solutions to the problem \eqref{linear}-\eqref{icd-linear}.

First, decompose $S_0 A_1, S_0 A_2$ into $S_0 A_1 = A_1^I + A_2^{II}$ and $S_0 A_2 = A_2^I + A_2^{II}$, respectively, such that
\begin{equation}\label{A-I}
(A_1^{I} \partial_1 + A_2^{I} \partial_2) : z\longrightarrow \begin{pmatrix}
    RU\cdot\nabla u - B\cdot\nabla b\\
    (1+\frac{1}{Q}B\otimes B)(U\cdot\nabla b) - \frac{1}{Q}B(U\cdot \nabla \p) - B\cdot\nabla u\\ \frac{1}{Q}U\cdot\nabla \p- \frac{1}{Q}B\cdot(U\cdot\nabla b)\\
    U\cdot\nabla s
\end{pmatrix},
\end{equation}
and
\begin{equation}\label{A-II}
    (A_1^{II} \partial_1 + A_2^{II} \partial_2) : z\longrightarrow (\nabla \p, 0,0,\nabla\cdot u, 0)^t,
\end{equation}
with $z=(u,b,\p,s)^t$.
Namely, $A_1^I\partial_1 + A_2^I \partial_2$ and $A_1^{II}\partial_1 + A_2^{II} \partial_2$ are the tangential and normal derivative parts of $(S_0 A_1)\partial_1 + (S_0 A_2)\partial_2$, respectively.

For a fixed $\alpha \in \mathbb{N}^3$ with $|\alpha| \le m$, 
by applying $\mathbb{T}^\alpha$ on the system \eqref{symm}, it follows
\begin{equation*}
\begin{aligned}
    (\tilde{L}+\mathbb{B})(Z) \mathbb{T}^\alpha z & = \mathbb{T}^\alpha \overline{F} + [(S_0 A_0)(Z) \partial_t, \mathbb{T}^\alpha] z + [A_1^I(Z) \partial_1 + A_2^I (Z)\partial_2, \mathbb{T}^\alpha]z\\
    & +[\mathbb{B}(Z), \mathbb{T}^\alpha] z + [A_1^{II}\partial_1 + A_2^{II}\partial_2, \mathbb{T}^\alpha] z.
\end{aligned}
\end{equation*}
By employing  the estimate \eqref{L2-temp} for the IBVP of the above equation, it follows
\begin{equation}\label{tan''}
    \begin{aligned}
    &\quad\frac{\d}{\d t} \int_\Omega R|\mathscr{T}^\alpha u|^2 + |\mathscr{T}^\alpha b|^2 + \frac{1}{Q} |\mathcal{T}^\alpha \p - B\cdot (\mathscr{T}^\alpha b)|^2 + |\mathcal{T}^\alpha s|^2\,\d x \\
    &\le P(\normmm{Z(t)}_{6,*})\|\mathbb{T}^\alpha z(t)\|_0^2 + \int_\Omega \mathbb{T}^\alpha z \cdot ([A_1^{II}\partial_1 + A_2^{II}\partial_2,\mathbb{T}^\alpha] z) \d x\\
    &+ \int_\Omega \mathbb{T}^\alpha z \cdot \left( \mathbb{T}^\alpha \overline{F} + [(S_0 A_0)(Z) \partial_t, \mathbb{T}^\alpha] z + [A_1^I(Z) \partial_1 + A_2^I (Z)\partial_2, \mathbb{T}^\alpha] z + [\mathbb{B}(Z), \mathbb{T}^\alpha] z\right) \d x.
\end{aligned}
\end{equation}
By using Corollary~\ref{coro-moser}, one has
\begin{equation}\label{tan''-1}
    \begin{aligned}
        \left\|[(S_0 A_0)(Z) \partial_t+\mathbb{B}(Z), \mathbb{T}^\alpha] z\right\|_0 & \le \left\|[(S_0 A_0) (Z), \mathbb{T}^\alpha] \partial_t z(t)\right\|_0 + \left\|[\mathbb{B}(Z), \mathbb{T}^\alpha] z\right\|_0\\
        & \le \eta_m(z,Z,t),
    \end{aligned}    
\end{equation}
where
\begin{equation}\label{eta-def}
        \eta_m(z,Z,t) = \begin{cases}
            P(M_m(Z,t))\normmm{z(t)}_{m,*}, & \text{if } m= 7,\\
            \sum\limits_{i=0,1} P(M_{m-i}(Z,t))\normmm{z(t)}_{m-1+i,*}, & \text{if } m\ge 8,
        \end{cases}
\end{equation}
with $M_k(Z,t)$ being defined in Notation~\ref{nota-2.2}.

Now, let us consider 
the term $[A_1^I(Z)\partial_1 + A_2^I(Z)\partial_2, \mathbb{T}^\alpha] z$ given on the right hand side of \eqref{tan''}.
Roughly speaking, $A_1^I(Z)\partial_1 + A_2^I(Z)\partial_2$ is a first order tangential operator, thus $[A_1^I(Z)\partial_1 + A_2^I(Z)\partial_2, \mathbb{T}^\alpha] z$ contains at most $|\alpha|$-th order tangential derivatives of $z$ according to \eqref{tan-cmtt-2}. However, one cannot use Lemma~\ref{lemma-tangential} to estimate this term directly. For example, when dealing with $[U\cdot \nabla, \mathbb{T}^\alpha]z$ with $|\alpha|=m$, if one identifies $U\cdot \nabla$ with $V_1\partial_{w^1} + V_2\partial_{w^2}$ as in Lemma~\ref{lemma-tangential}, there would be $m$-th order tangential derivatives falling on $V_i$, which  requires $\normmm{U(t)}_{m+2,*}$ to control, two order higher than we wish to have given in  \eqref{apriori}.

To fill in this gap, we introduce the following lemma.
\begin{lemma}\label{lemma-Udot}
        For a fixed integer 
        $s\ge 5$, there exists a positive constant $C>0$, such that for any given $U= (U_1,U_2)^t\in C^\infty([0,T]\times \overline{\Omega})$ satisfying $U\cdot \nu = 0$ on $ \partial\Omega\backslash\{q_1,\ldots,q_N\}$, and $f\in X^s_*([0,T]\times \Omega)$, the following inequalities hold for any $0\le t\le T$.
        \begin{enumerate}
            \item[(1)] When $1\le j\le s$ is an integer,
            \begin{equation*}
            \|\partial_*^{s-j}[\partial_*^j, U\cdot \nabla]f(t)\|_0 \le \begin{cases}
                C\normmm{U(t)}_{7,*}\normmm{f(t)}_{s,*},&\text{as }  5\le s\le 7,\\
                C\sum\limits_{i=0,1} \normmm{U(t)}_{s-i,*}\normmm{f(t)}_{s-1+i,*},&\text{as }  s \ge 8.
            \end{cases}
            \end{equation*}
            \item[(2)] When $1\le i\le j\le 2$, and  $s\ge 7$,
            $$\normmm{(U\cdot \nabla)^i f(t)}_{s-j,*} \le C \normmm{U(t)}^i_{s,*}\normmm{f(t)}_{s,*};$$
            and as $s\ge 8$, then
            $$\normmm{(U\cdot \nabla)^i f(t)}_{s-j,*} \le C\normmm{U(t)}_{s-1,*}^{i-1}\sum_{l=0,1} \normmm{U(t)}_{s-l,*}\normmm{f(t)}_{s-1+l,*}.$$
        \end{enumerate}
    \end{lemma}
    \begin{proof} Note that
    \begin{equation*}
        \begin{aligned}
            \|\partial_*^{s-j}[\partial_*^j, U\cdot \nabla]f(t)\|_0 \lesssim &\|\sum\limits_{i=1}^2\partial_* U_i\partial_*^{s-1}\partial_i f(t)\|_0 + \sum_{k=2}^s \|\sum\limits_{i=1}^2\partial_*^k U_i\partial_*^{s-k}\partial_i f(t)\|_0\\
            \lesssim & \|(\mathscr{T}^1 U) \cdot \nabla \partial_*^{s-1}f(t)\|_0 + \|(\mathcal{T}^1-\mathscr{T}^1)U\cdot \nabla \partial_*^{s-1}f(t)\|_0 \\
            &+\|\sum\limits_{i=1}^2\partial_* U_i [\partial_*^{s-1},\partial_i] f\|_0 + \sum\limits_{i=1}^2\normmm{\partial_*^2 U_i \partial_i f(t)}_{s-2,*}\\
            \lesssim & \Big(\|\partial_x(\mathscr{T}^1 U)(t)\|_{L^\infty(\Omega)} + \|\partial_x((\mathcal{T}^1-\mathscr{T}^1)U)(t)\|_{L^\infty(\Omega)} \\
            &+\|\partial_* U(t)\|_{L^\infty(\Omega)}\Big) \normmm{f(t)}_{s,*} + \sum\limits_{i=1}^2\normmm{\partial_*^2 U_i \partial_i f(t)}_{s-2,*},
        \end{aligned}
    \end{equation*}
    where in the last inequality we invoke Lemma~\ref{lemma-tangential} and \eqref{tan-cmtt-1}. Therefore, we get the inequalities given in the first part of this lemma by using Lemma~\ref{lemma-moser}, Corollary~\ref{coro-moser} and Lemma~\ref{lemma-imbed}.
    
    The proof of the inequalities given in the second part of this lemma is also a straightforward application of Corollary~\ref{coro-moser} by observing that $\partial_*(U\cdot \nabla f)=\sum\limits_{i=1}^2(\partial_* U_i \partial_i f + V_i\partial_{w^i}\partial_* f + U_i [\partial_*,\partial_i]) f$, where $U_i\partial_i = V_i\partial_{w^i}$ as in Lemma~\ref{lemma-tangential}.
    \end{proof}

Using Lemma~\ref{lemma-Udot}, Lemma~\ref{lemma-tangential} and Corollary~\ref{coro-moser}, we conclude that
\begin{equation}\label{tan''-2}
    \|[A_1^I(Z)\partial_1+A_2^I(Z)\partial_2, \mathbb{T}^\alpha]z\|_0 \le \eta_m(z,Z,t),
\end{equation}
with $\eta_m(z,Z,t)$ being defined in \eqref{eta-def}.

Now we are left to estimate the term $\int_\Omega \mathbb{T}^\alpha z\cdot ([A_1^{II}\partial_1+ A_2^{II}\partial_2, \mathbb{T}^\alpha] z) \d x$ appeared on the right hand side of \eqref{tan''}.
By a direct computation, we have
\begin{equation}\label{J1-J2}
\begin{aligned}
    &\quad\inp*{\mathbb{T}^\alpha z}{[A_1^{II}\partial_1+ A_2^{II}\partial_2, \mathbb{T}^\alpha] z}_{\Omega} \\&= \inp*{\mathscr{T}^\alpha u}{\nabla \mathcal{T}^\alpha \p - \mathscr{T}^\alpha \nabla \p}_{\Omega} + \inp*{\mathcal{T}^\alpha \p}{\nabla\cdot \mathscr{T}^\alpha u - \mathcal{T}^\alpha \nabla\cdot u}_{\Omega}\\
    & =: \mathcal{J}_1 + \mathcal{J}_2.
\end{aligned}
\end{equation}
According to \eqref{tan-cmtt-1}, $\mathcal{J}_1$ ($\mathcal{J}_2$ resp.) is a sum of terms of the form $\int_\Omega\Theta(x)\mathcal{T}^\beta u_l \cdot \mathcal{T}^\gamma \partial_i p \d x$ ($\int_\Omega \Theta(x)\mathcal{T}^\beta \p \cdot \mathcal{T}^\gamma \partial_i u_l \d x$ resp.), with $|\beta|\le |\alpha|, |\gamma|\le |\alpha|-1, i,l = 1,2$ and $\Theta \in C^\infty(\overline{\Omega})$.
Then there are at most $(|\alpha|+1)$-th $*$-order derivatives falling on $z$ in both $\mathcal{J}_1$ and $\mathcal{J}_2$.
From the momentum equation given in \eqref{linear}, we know that $\partial_i \p$ can be replaced by $(F_1)_i - R(\partial_t+U\cdot\nabla) u_i + B\cdot \nabla b_i$.
Therefore, by applying Lemma~\ref{lemma-Udot} for the term $\mathcal{J}_1$ given in \eqref{J1-J2}, one has
\begin{equation}\label{cmtt-J1}
    |\mathcal{J}_1| \le \sum_{|\beta|\le |\alpha|}\|\mathscr{T}^\beta u(t)\|_0\cdot \begin{cases}
P(\normmm{Z(t)}_{m,*}) (\normmm{{F}(t)}_{m-1,*} + \normmm{{z}(t)}_{m,*}), & \text{if } m\ge 7,\\[2mm]
\begin{array}{r}
\sum_{i=0,1} P(\normmm{Z(t)}_{m-i,*}) \big(\normmm{F(t)}_{m-2+i,*}\quad \\
+\normmm{z(t)}_{m-1+i,*}\big), 
\end{array} 
& \text{if }m\ge 8.
\end{cases}
\end{equation}

The estimate of $\mathcal{J}_2$ is more subtle, relying on a sharper observation. 
First, we define the operators $\mathscr{H}_i$, $i=1,2$ by
\begin{equation*}
    \mathscr{H}_i u = \nabla\cdot \big(h(x,w^i) u\big) + [\nabla\cdot, \mathcal{T}_i] u,
\end{equation*}
mapping the vector $u$ into a scalar function, where $\mathcal{T}_i = \partial_{w^i}$ as introduced before.
Then, we have
\begin{lemma}\label{lemma-observ}
    (1) For any fixed  $\tilde{x} \in \partial\Omega\backslash\{q_1,\ldots,q_N\}$, let $w^i$, $h(x,w^i), i=1,2$ and $\Phi(x)$ be as introduced in Definition~\ref{def-tan}, \eqref{2.6} and Notation~\ref{nota-1}(3) respectively, and assume that $\partial_1\Phi(x)\ne 0$ in $V(\tilde{x})\cap \Omega$ without loss of generality. 
    Then, for any smooth function $f$ defined in $V(\tilde{x})\cap \Omega$ and vanishing on $V(\tilde{x})\cap \partial\Omega$, there exist smooth functions $\Theta^{\alpha}_{i,l}(x)$ with $i,l=1,2$, such that 
    \begin{equation}\label{cmtt-observ1}
        \begin{aligned}
	&f(x) \mathscr{H}_1 u = \sum_{\substack{\alpha \in \mathbb{N}^2, |\alpha| \le 1,\\ l=1,2}} \Theta^\alpha_{1,l}(x) \mathcal{T}_1^{\alpha_1} \mathcal{T}_2^{\alpha_2} u_l,\\
	& \mathscr{H}_2 u = \sum_{\substack{\alpha \in \mathbb{N}^2, |\alpha| \le 1,\\ l=1,2}} \Theta^\alpha_{2,l}(x) \mathcal{T}_1^{\alpha_1} \mathcal{T}_2^{\alpha_2} u_l.
	\end{aligned}
	\end{equation}
    (2) For a fixed corner $\tilde{x} = q_n$ for some $1\le n\le N$, let $w^i$, $h(x,w^i)$ be as above, and $\gamma_i^n, \Phi_i(x), i=1,2$ be as introduced in Notation~\ref{nota-1}(4). Then, for any $f_i$ defined in $V(\tilde{x}) \cap \Omega$ and vanishing on $\gamma_i^n$, there exist smooth functions $\Theta^\alpha_{i,l}(x)$ with $i,l=1,2$, such that 
	\begin{equation}\label{cmtt-observ2}
	f_i(x) \mathscr{H}_i u = \sum_{\substack{\alpha \in \mathbb{N}^2, |\alpha| \le 1,\\ l=1,2}} \Theta^\alpha_{i,l}(x) \mathcal{T}_1^{\alpha_1} \mathcal{T}_2^{\alpha_2} u_l.
	\end{equation}
    
\end{lemma}

\begin{proof}
	We first prove the above second point.
	Consider the case $i=1$.
	It suffices to show that the leading order operator on the left-hand side of \eqref{cmtt-observ2} is tangential. Indeed, the left-hand side of \eqref{cmtt-observ2} equals to 
    $$f_1(x) \left(\sum\limits_{l,j=1}^2\partial_l w^1_j\partial_j u_l + \sum\limits_{i,l,j=1}^2w^1_j h^j_{il}\partial_i u_l\right)+{\rm zero-th~order~terms~of~}u.$$
     
     Since $f_1(x) = 0$ on $\gamma_1^n$, we already see that the first term of the above formula is a tangential derivative of $u$ on $V(\tilde{x})\cap\gamma^n_1$ using Lemma~\ref{lemma-tangential}. To verify this property on $V(\tilde{x})\cap\gamma^n_2$, it suffices to have
     $$\sum\limits_{j=1}^2\partial_l w^1_j\partial_j \Phi_2(x) + \sum\limits_{i,j=1}^2w^1_j h^j_{il}\partial_i \Phi_2(x) = 0, \quad {\rm for}\quad  l=1,2.$$ 
     This is true from the explicit expressions of $h^1, h^2$ given in Lemma~\ref{lemma-h}.
	
	Next, we consider the first point given in this lemma. The first claim of \eqref{cmtt-observ1} is trivial since $f(x) = 0$ on $V(\tilde{x})\cap \partial\Omega$. In a way similar to the above proof of the second point, to verify the second assertion given in  \eqref{cmtt-observ1}, 
	it suffices to have 
    $$\sum\limits_{j=1}^2\partial_l w^2_j\partial_j \Phi(x) +\sum\limits_{i, j=1}^2 w^2_j h^j_{il}\partial_i \Phi(x) = 0, \quad  {\rm for}\quad  l=1,2.$$ 
    This can also be readily obtained by a direct computation.
\end{proof}

From Lemma~\ref{lemma-observ}, one can easily see the following properties.
\begin{coro}\label{coro-tan-struc}
    For a smooth enough function $f$ defined in $\Omega$, one has that
    for any $\alpha, \beta \in\mathbb{N}^3$, $(\mathcal{T}^\alpha (w^i_j(x) f)) \cdot (\mathcal{T}^\beta \mathscr{H}_i u)$ ($i,j = 1,2$) is a sum of terms in the form
        $\Theta(x) \mathcal{T}^{\alpha'} f(x) \cdot \left(\mathcal{T}^{\beta'} u_l\right)$, with $|\alpha'| \le |\alpha|, |\beta'| \le |\beta|+1, l=1,2$ and $\Theta\in C^\infty(\overline{\Omega})$.
\end{coro}
\begin{proof}
    Using a partition of unity, it suffices to prove the property given in Corollary~\ref{coro-tan-struc} locally at any point $\tilde{x}\in \overline{\Omega}$. 
    
    When $\tilde{x} = q_n$ for some $1\le n\le N$, let $w^i, h(x,w^i), \gamma_i^n, \Phi_i(x), i=1,2$ be as in Lemma~\ref{lemma-observ}(2). Recall that $w^i(x) = (-1)^{i}\Phi_{i}(x)(\partial_2 \Phi_{i'}(x), - \partial_1\Phi_{i'}(x))^t$ with $i\neq i'$, it follows $w^i_j(x) = 0$ on $\gamma^n_i$ for $j=1,2$. 
    
    For any $\alpha, \beta \in \mathbb{N}^3$, one has
    \begin{equation*}
        \begin{aligned}
            \mathcal{T}^\alpha(w^i_j(x) f)\cdot \mathcal{T}^\beta \mathscr{H}_i u 
            = \sum_{\alpha' \le \alpha} \binom{\alpha}{\alpha'} (\mathcal{T}^{\alpha - \alpha'}f) \left(\mathcal{T}^{\alpha'} w^i_j(x)\cdot \mathcal{T}^\beta \mathscr{H}_i u\right),
        \end{aligned}
    \end{equation*}
    where $\binom{\alpha}{\alpha'} = \prod_{i=0,1,2}\binom{\alpha_i}{\alpha_i'}$.
    Furthermore, from the formula
    \begin{equation*}
        \mathcal{T}^{\alpha'} w^i_j(x)\cdot \mathcal{T}^\beta \mathscr{H}_i u = \sum_{\beta'\le \beta} (-1)^{|\beta-\beta'|} \binom{\beta}{\beta'}\mathcal{T}^{\beta'}\left(\mathcal{T}^{\beta-\beta'+\alpha'}w^i_j(x)\cdot\mathscr{H}_i u\right),
    \end{equation*}
   and applying Lemma~\ref{lemma-observ}(2) with $f_i = \mathcal{T}^{\beta-\beta'+\alpha'} w^i_j$, by noting that it vanishes on $\gamma_i^n$, the property given in Corollary~\ref{coro-tan-struc} in a neighborhood of $\tilde{x}$ is proved. 

    For $\tilde{x} \in \partial\Omega \backslash\{q_1,\ldots, q_N\}$, one can get the conclusion similarly. 
    
    For any interior point $\tilde{x}$ in $\Omega$, the conclusion holds trivially in a small neighborhood of $\tilde{x}$ by using Definition~\ref{def-tan}(1).
\end{proof}
The next lemma shows the hidden structure of $\mathcal{J}_2$.
\begin{lemma}\label{lemma-J2}
    For any fixed $\alpha\in \mathbb{N}^3$ with $1\le |\alpha|\le m$,  the term $\mathcal{J}_2$ given in \eqref{J1-J2}
    is a sum of terms of the form 
    \begin{equation}\label{form}
        \int_\Omega \Theta(x) \mathcal{T}^\beta \partial_{x_i} \p \cdot \mathcal{T}^\sigma u_l ~\d x, \quad \text{with } |\beta|\le |\alpha|-1, |\sigma|\le |\alpha|.
    \end{equation}
   with $i,l = 1,2$, and $\Theta\in C^\infty(\overline{\Omega})$. 
\end{lemma}
\begin{proof}
    Since $\partial_t$ commutates with all $\partial_{w^i}, h(x,w^i)$, we assume without loss of generality that $\alpha_0 = 0$, so $|\alpha| = \alpha_1 + \alpha_2$, $\mathcal{T}^\alpha = \mathcal{T}_1^{\alpha_1}\mathcal{T}^{\alpha_2}_2$.

    First, one can inductively obtain
    \begin{equation*}
        \begin{aligned}
            &\nabla\cdot \mathscr{T}^\alpha u - \mathcal{T}^\alpha \nabla\cdot u \\
        = & \sum_{0\le \beta_2 \le \alpha_2 -1} \mathcal{T}^{\alpha_1}_1 \mathcal{T}^{\alpha_2-1-\beta_2}_2 \mathscr{H}_2 (\mathscr{T}^{\beta_2}_2 u) + \sum_{0\le \beta_1 \le \alpha_1-1} \mathcal{T}^{\alpha_1 - 1- \beta_1}_1 \mathscr{H}_1 (\mathscr{T}^{\beta_1}_1 \mathscr{T}^{\alpha_2}_2 u),
        \end{aligned}
    \end{equation*}
    where, by convention, any summation is understood to be zero if its range is empty.
    
    Then, $\mathcal{J}_2$ can be rewritten as
    \begin{equation*}
        \begin{aligned}
            \mathcal{J}_2 & = \sum_{0\le \beta_2 \le \alpha_2 -1}\sum_{i=1,2} 
            \inp*{\mathcal{T}_1^{\alpha_1}\left(w^2_i(x) \partial_{x_i}(\mathcal{T}_2^{\alpha_2-1} \p)\right)}{\mathcal{T}^{\alpha_1}_1 \mathcal{T}^{\alpha_2-1-\beta_2}_2 \mathscr{H}_2 (\mathscr{T}^{\beta_2}_2 u)}_\Omega\\
            & + \sum_{0\le \beta_1 \le \alpha_1-1}\sum_{i=1,2}\inp*{w^1_i(x)\partial_{x_i}\left(\mathcal{T}_1^{\alpha_1-1}\mathcal{T}_2^{\alpha_2} \p\right)}{\mathcal{T}^{\alpha_1 - 1- \beta_1}_1 \mathscr{H}_1 (\mathscr{T}^{\beta_1}_1 \mathscr{T}^{\alpha_2}_2 u)}_\Omega,
        \end{aligned}
    \end{equation*}
   which implies that $\mathcal{J}_2$ is a sum of terms in the form given in \eqref{form}, by using Corollary~\ref{coro-tan-struc}, \eqref{tan-cmtt-1} and \eqref{tan-cmtt-2}.
\end{proof}

Now, one can estimate the terms of the form given in \eqref{form} in the same way as for $\mathcal{J}_1$. 
Therefore, we conclude that $\mathcal{J}_2$ is also controlled by the right-hand side of \eqref{cmtt-J1}.

Plugging the estimates of $\mathcal{J}_1, \mathcal{J}_2$, and
\eqref{tan''-1}, \eqref{tan''-2} into \eqref{tan''}, it follows that for any $\alpha\in \mathbb{N}^3$ with $|\alpha|\le m$,

\begin{equation}\label{est-tan}
\begin{aligned}
&\frac{\d}{\d t} \int_{\Omega} R|\mathscr{T}^\alpha u|^2 + |\mathscr{T}^\alpha b|^2 + \frac{1}{Q}|\mathcal{T}^\alpha \p - B\cdot \mathscr{T}^\alpha b|^2 + |\mathcal{T}^\beta s|\d x\\ &
\le \sum_{|\beta|\le |\alpha|}\|\mathbb{T}^\beta z\|_0\cdot \big( \eta_m({F},Z,t) + \eta_m({z},Z,t) \big)
\end{aligned}
\end{equation}
Now we take the sum of all $\alpha\in \mathbb{N}^3, |\alpha|\le m$, and use the Gronwall inequality for \eqref{est-tan} to conclude

\begin{prop}
For a fixed integer $m\ge 7$, assume that the smooth background state $Z=(U, B, P, S)^t$ satisfies the hypothesis \eqref{assum-Z}, and the compatibility conditions up to order $m-1$ are satisfied for the problem \eqref{linear}-\eqref{icd-linear}. Then, the solution $z=(u,b,\p,s)$ of \eqref{linear}-\eqref{icd-linear} satisfies the following estimate
\begin{equation}\label{est-tan''}
\sum_{|\alpha|\le m} \|\mathbb{T}^\alpha z(t)\|_0 \lesssim \sum_{|\alpha|\le m} \|\mathbb{T}^\alpha z(0)\|_0  + \int_0^t \big( \eta_m({F},Z,\tau) + \eta_m({z},Z,\tau) \big) \d \tau,
\end{equation}
for all $0\le t\le T$.

\end{prop}

\subsubsection{Estimates of normal derivatives}

The proposal of this subsection is to estimate the normal derivatives of solutions to the linearized problem
\eqref{linear}-\eqref{icd-linear}.
First, we recall a result from \cite{godin20212d} and \cite{grisvard1985elliptic}.
    \begin{lemma}\label{lemma-Hodge}
        Let $\Omega$ be a bounded subdomain of $\R^2$ with piecewisely smooth boundary and finitely many corners $\{q_n\}_{n=1}^N$ on $\partial\Omega$ with inner angles being $\{\omega_n\}_{n=1}^N$ correspondingly. If for a fixed $s\ge 1$,  $0<\omega_n<\frac{\pi}{s}$ holds for all $1\le n\le N$, then there is a constant $C>0$, such that the following estimate
        \begin{equation}\label{Hodge-est}
            \|X\|_{s} \le C(\|X\|_{0} + \|\nabla \cdot X\|_{s-1} + \|\nabla\times X\|_{s-1})
        \end{equation}
       holds for any field $X \in H^s(\Omega)$ satisfying $X\cdot \nu = 0$ on $\partial \Omega\backslash\{q_1, \ldots,q_N\}$. If $\Omega$ is simply connected, the term $\|X\|_0$ may be omitted on the right-hand side of \eqref{Hodge-est}.
    \end{lemma}

    \begin{remark}
        Note that the assumption $\omega_n < \frac{\pi}{s}$ in the lemma is necessary; otherwise, some counterexamples were constructed in  \cite{grisvard1985elliptic}.
    \end{remark}
    Now, the main aim is to establish the following result.

    \begin{prop}\label{2.2}
    Under the same assumptions as given in Theorem \ref{thm-apriori}, the solution 
$z=(u,b,\p,s)$ of the problem \eqref{linear}-\eqref{icd-linear} satisfies the following estimates
    \begin{equation}\label{est-normal''}
        \begin{aligned}
            &\sum_{1\le k\le [\frac{m}{2}]}\sum_{|\beta|\le m-2k} \left(\|\mathcal{T}^\beta \p(t)\|_{k} + \|\mathscr{T}^\beta(u,b)(t)\|_{k}\right)
            \\
            & \quad \lesssim 
            \begin{cases}
            \phi(M_{m}(Z,t))\big(\Lambda_{m}({z},{F},t) +
            \int_0^t \normmm{{z}(\tau)}_{m,*}\d \tau \big), 
            &\text{if }m\ge 7,\\
            \begin{aligned}
                &\phi_1(M_{m-1}(Z,t))\left(\Lambda_m({z},{F},t) + \int_0^t \normmm{{z}(\tau)}_{m,*} \d \tau\right)\\
                & \quad + \phi_2(M_{m}(Z,t))\int_0^t\left(\normmm{{z}(\tau)}_{m-1,*} + \normmm{{F}(\tau)}_{m-1,*}\right)\d\tau,
            \end{aligned}
            & \text{if }m\ge 8,
        \end{cases}
        \end{aligned}
    \end{equation}
for all $0\le t\le T$, with the same notations as in Theorem \ref{thm-apriori}.
        
    \end{prop}

\begin{proof}
(1) \underline{\it Estimate on $\p$}. First, by using the momentum equation given in \eqref{linear}, we have
    \begin{equation}\label{2.47}
    \begin{aligned}
        \|\mathcal{T}^\beta \p(t)\|_{k} & \lesssim 
        \|\mathcal{T}^\beta \p(t)\|_{k-1} + \|[\nabla,\mathcal{T}^\beta]\p(t)\|_{k-1} + \|\mathcal{T}^\beta \nabla \p(t)\|_{k-1}\\
        & \lesssim \normmm{\p(t)}_{m-1,*} + \|\mathcal{T}^\beta\big(F_1 + B\cdot \nabla b - R D_t u\big)(t)\|_{k-1}
        \\
        & \lesssim \normmm{\p(t)}_{m-1,*} + \normmm{{F}(t)}_{m-2,*}
        + \normmm{B\cdot \nabla b(t)}_{m-2,*} + \normmm{R D_t u(t)}_{m-2,*}
    \end{aligned}
    \end{equation}
where the notation
$$D_t = \partial_t + U\cdot \nabla$$
shall be used in the following calculation frequently for simplicity.

     If $m\ge 8$, setting $s=m-1\ge 7$, $i=j=1$ in Lemma~\ref{lemma-Udot}, then $\normmm{B\cdot \nabla b(t)}_{m-2,*}\lesssim \normmm{B(t)}_{m-1,*} \normmm{b(t)}_{m-1,*}$, while $m=7$, writing $B\cdot\nabla = \bar{B}_1 \partial_{w^1} + \bar{B}_2 \partial_{w^2}$ as in Lemma~\ref{lemma-tangential}, then by Lemma~\ref{lemma-moser}, $\normmm{B\cdot \nabla b(t)}_{5,*} \lesssim \normmm{B(t)}_{7,*}\normmm{b(t)}_{6,*}$.
    The analysis of $\normmm{R D_t u(t)}_{m-2,*}$ is similar. Thus, we conclude
\begin{equation}\label{grad-p}
    \begin{aligned}
        \|\mathcal{T}^\beta \p(t)\|_{k} \lesssim \normmm{\p(t)}_{m-1,*} + \normmm{{F}(t)}_{m-2,*}
        +
        P(\normmm{Z(t)}_{\max{(m,8)}-1,*}) \normmm{{z}(t)}_{m-1,*}.
    \end{aligned}
\end{equation}

(2) \underline{\it Estimate on the divergence of $u$}.
From  the continuity equation given in \eqref{linear}, we see that 
the divergence of $u$ is represented by the tangential derivatives of $b, \p$, so one has
    \begin{equation}\label{div-u}
    \begin{aligned}
        &\|\nabla\cdot \mathscr{T}^\beta  u(t)\|_{k-1} \le 
        \|[\nabla\cdot, \mathscr{T}^\beta] u(t)\|_{k-1} + \|\mathscr{T}^\beta (\nabla\cdot u)(t)\|_{k-1}
        \\
        &\lesssim \normmm{u(t)}_{m-1,*} + \|\mathscr{T}^\beta\big(F_3 + \frac{1}{Q} B\cdot D_t b - \frac{1}{Q} D_t \p\big)(t)\|_{k-1}\\
        &\le \normmm{u(t)}_{m-1,*} + \normmm{F_3(t)}_{m-2,*} + \normmm{\frac{1}{Q} B\cdot D_t b(t)}_{m-2,*} + \normmm{\frac{1}{Q}D_t\p(t)}_{m-2,*}\\
        &\lesssim \normmm{u(t)}_{m-1,*} + \normmm{F_3(t)}_{m-2,*} +
        P(\normmm{Z(t)}_{\max{(m,8)}-1,*}) \normmm{{z}(t)}_{m-1,*}.
    \end{aligned}
    \end{equation}
    where in the last inequality we have used the same idea as above in estimating the last two terms of the right hand side of \eqref{2.47}.

(3) \underline{\it Estimate on the divergence of $b$}.

   Taking the divergence operator  $\nabla\cdot$ on both sides of the second equation given  in \eqref{linear}, it follows
    \begin{equation}\label{eq-div-b}
    \begin{aligned}
         D_t(\nabla\cdot b) = &\sum\limits_{i,j=1}^2(\partial_i B_j\partial_j u_i - \partial_i U_j \partial_j b_i)- (\nabla\cdot B)(\nabla\cdot u) \\
         & + \nabla\cdot  F_2
         + \nabla\cdot \big( h(x,B)u-h(x,U)b \big),
    \end{aligned}
    \end{equation}
    in which the right-hand side contains at most one order  derivatives of $u$ and $b$.
    By applying $\partial_x^{k-1}\mathscr{T}^{\beta}$ 
     with $|\beta|+2k\le m$,
    on both sides of \eqref{eq-div-b}, it gives
    \begin{equation}\label{tan-div-b}
        D_t(\nabla\cdot\partial_x^{k-1}\mathscr{T}^\beta b) = \partial_x^{k-1}\mathscr{T}^\beta (\text{RHS of \eqref{eq-div-b}}) - D_t(\partial_x^{k-1}[\mathscr{T}^\beta,\nabla\cdot]b) + [D_t,\partial_x^{k-1}\mathscr{T}^\beta](\nabla\cdot b),
    \end{equation}
   from which we deduce \begin{equation}\label{dt-div-b}
        \begin{aligned}
            \frac{\d}{2\d t} \int_\Omega |\nabla\cdot \partial_x^{k-1}\mathscr{T}^\beta b(t)|^2 \d x \le & \frac{1}{2} \int_\Omega (\nabla\cdot U)|\nabla\cdot\partial_x^{k-1}\mathscr{T}^\beta b(t)|^2 \d x\\
            & + \|\nabla\cdot \partial_x^{k-1}\mathscr{T}^\beta b(t)\|_0\Big(
            \|\partial_x^{k-1}\mathscr{T}^\beta (\text{RHS of \eqref{eq-div-b}})(t)\|_0\\
            & + \|D_t(\partial_x^{k-1}[\mathscr{T}^\beta,\nabla\cdot]b)(t)\|_0 + \|[D_t,\partial_x^{k-1}\mathscr{T}^\beta](\nabla\cdot b)(t)\|_0\Big),
        \end{aligned}
    \end{equation}
     by using the standard $L^2$ energy method.

    By applying Lemma~\ref{lemma-tangential}, with the notation $U_i\partial_i = V_i \partial_{w^i}$, one has
    \begin{equation*}
         \begin{aligned}
             \|(\partial_t + U\cdot \nabla)(\partial_x^{k-1}[\mathscr{T}^\beta,\nabla\cdot]b)(t)\|_0 & \le \|\partial_t \partial_*^{m-1} b(t)\|_0 + \sum_{i=1}^2\|V_i(t)\|_{L^\infty(\Omega)}\|\partial_{w^i} \partial_*^{m-1}b(t)\|_0\\
             &\le P(\normmm{U(t)}_{6,*}) \normmm{b(t)}_{m,*}.
         \end{aligned}
    \end{equation*}
    By using  Lemma~\ref{lemma-Udot}(1) with $s=j=m-2$, we have
    \begin{equation*}
        \|[\partial_t + U\cdot\nabla,\partial_x^{k-1}\mathscr{T}^\beta](\nabla\cdot b)(t)\|_0 \lesssim 
            \begin{cases}
                \normmm{U(t)}_{7,*} \normmm{\nabla\cdot b (t)}_{m-2,*} & \text{if }m\ge 7,\\
                \sum_{i=0,1} \normmm{U(t)}_{m-2-i,*}\normmm{\nabla\cdot b(t)}_{m-3+i,*} & \text{if }m\ge 8.
            \end{cases}
    \end{equation*}
    Use Lemma~\ref{lemma-moser} again to estimate $\int_\Omega (\nabla\cdot U)|\nabla\cdot\partial_x^{k-1}\mathscr{T}^\beta b(t)|^2 \d x$ and $\|\partial_x^{k-1}\mathscr{T}^\beta (\text{RHS of \eqref{eq-div-b}})(t)\|_0$. 
    Thus, from \eqref{dt-div-b} we conclude
    \begin{equation}\label{div-b}
        \frac{\d}{\d t} \int_\Omega |\nabla\cdot \partial_x^{k-1}\mathscr{T}^\beta b(t)|^2\d x \lesssim \|\nabla\cdot \partial_x^{k-1}\mathscr{T}^\beta b(t)\|_0 \big(\normmm{{F}(t)}_{m,*} + \eta_m({z},Z,t)\big),
    \end{equation}
    with $\eta_m({z},Z,t)$ being given in \eqref{eta-def}. From \eqref{div-b}, we immediately obtain
\begin{equation}\label{div-b''}
    \normmm{(\nabla\cdot b)(t)}_{m-2,*} \lesssim \normmm{(\nabla\cdot b)(0)}_{m-2,*} + \int_0^t \big(\normmm{{F}(\tau)}_{m,*} + \eta_m({z},Z,\tau)\big) \d \tau.
    \end{equation}

(4) \underline{\it Estimate on curls of $u$ and $b$}.

On the other hand, from Lemma~\ref{lemma-Hodge} one has
    \begin{equation}\label{dcsys}
        \|\mathscr{T}^\beta (u,b)(t)\|_{k} \lesssim \|\mathscr{T}^\beta (u,b)(t)\|_0 + \|\nabla\cdot \mathscr{T}^\beta (u,b)(t)\|_{k-1} + \|\nabla\times \mathscr{T}^\beta (u,b)(t)\|_{k-1},
    \end{equation}
    where we have used that $\mathscr{T}^\beta u\cdot \nu = \mathscr{T}^\beta b\cdot \nu = 0$ on $\partial\Omega\backslash\{q_1,\ldots,q_N\}$ from \eqref{bcd-linear} and \eqref{b-bcd-1}.

Thus, to study $\|\mathscr{T}^\beta (u,b)(t)\|_{k}$ from \eqref{dcsys}, by using the above estimates of the divergence of  $u$ and $b$, it suffices to estimate the curls of $u$ and $b$. 
   For that, inspired by the approach given in \cite{wang2025incompressible},
    we shall use the special structure contained in the linearized system \eqref{linear}  to study these curls.
    
    Denote by $E(u), E(b), E(\p)$ the first three equations of \eqref{linear} respectively. By computing $\nabla\times E(u)$ and $\nabla \times E(b) - \nabla\times (B E(\p)) - \frac{1}{Q} B\times D_t E(u)$, it follows that $(\nabla\times u, \nabla\times b)^t$ satisfies the following symmetric hyperbolic system:
    \begin{equation}\label{vort}
        \begin{cases}
            R D_t (\nabla\times u) - B\cdot\nabla (\nabla \times b) = \tilde{F}_1,\\
            (1+\frac{1}{Q}|B|^2) D_t(\nabla \times b) - B\cdot \nabla (\nabla\times u) = \tilde{F}_2,
        \end{cases}
    \end{equation}
    where
    \begin{equation}\label{F1,2}
        \begin{aligned}
        \tilde{F}_1 = & \nabla\times F_1 + R[D_t,\nabla\times]u - [B\cdot\nabla, \nabla\times] b - (\nabla R\times D_t u),\\
            \tilde{F}_2  = & \nabla\times F_2 - \nabla\times (B F_3) - \frac{1}{Q} B\times D_t F_1 
            -[\nabla\times, D_t]b + [\nabla\times, B\cdot \nabla] u \\
            & + \nabla\times \left(h(x,B)u-h(x,U)b\right)
            + \frac{1}{Q}(\nabla\times B) (D_t \p - B\cdot D_t b) - B\times [\nabla, \frac{1}{Q} D_t] \p\\
            & + B\times [\nabla, \frac{1}{Q} B\cdot] D_t b
            + \sum_{j=1,2} \frac{1}{Q} B\times ([B_j \partial_1, D_t] b_j, [B_j \partial_2, D_t] b_j)^t \\
            & - \frac{1}{Q} (B\cdot D_t B)(\nabla\times b) + \frac{1}{Q} B\times [D_t, R] D_t u + \frac{1}{Q} B\times (R D_t^2 u).
        \end{aligned}
    \end{equation}
   Note that all terms on the right hand side of \eqref{F1,2} contain the derivatives of unknowns at most order one except the last term $\frac{1}{Q}B\times (RD_t^2 u)$, which yields that to control one normal derivative of unknowns  requires two tangential derivatives of $u$.
   
   Applying $\partial_x^{k-1} \mathcal{T}^\beta$  with $|\beta|+2k\le m$ to \eqref{vort}, it follows
    \begin{equation}\label{vort-drv}
        \begin{cases}
            RD_t(\nabla\times \partial_x^{k-1}\mathscr{T}^\beta u) - B\cdot \nabla(\nabla\times \partial_x^{k-1}\mathscr{T}^\beta b) = \mathcal{R}_1,\\
            (1+\frac{1}{Q}|B|^2) D_t (\nabla\times \partial_x^{k-1}\mathscr{T}^\beta b) - B\cdot \nabla (\nabla\times \partial_x^{k-1}\mathscr{T}^\beta b) = \mathcal{R}_2,
        \end{cases}
    \end{equation}
    where
    \begin{equation*}
    \begin{aligned}
    &
    \begin{aligned}
        \mathcal{R}_1
        =& \partial_x^{k-1}\mathcal{T}^\beta \tilde{F}_1 
        - [\partial_x^{k-1}\mathcal{T}^\beta, RD_t(\nabla\times\cdot)] u + [\partial_x^{k-1}\mathcal{T}^\beta, B\cdot \nabla (\nabla\times\cdot)] b\\
        & - RD_t\big(\nabla\times\partial_x^{k-1}(\mathcal{T}^\beta-\mathscr{T}^\beta) u\big) + B\cdot\nabla \big(\nabla\times\partial_x^{k-1}(\mathcal{T}^\beta-\mathscr{T}^\beta)b \big),
    \end{aligned}
    \\
    &
    \begin{aligned}
        \mathcal{R}_2
        = & \partial_x^{k-1}\mathcal{T}^\beta \tilde{F}_2 - [\partial_x^{k-1}\mathcal{T}^\beta, (1+\frac{1}{Q}|B|^2) D_t \nabla\times] b + [\partial_x^{k-1}\mathcal{T}^\beta, B\cdot \nabla (\nabla\times\cdot)] u\\
        & - (1+\frac{1}{Q}|B|^2)D_t\big(\nabla\times \partial_x^{k-1}(\mathcal{T}^\beta-\mathscr{T}^\beta) b\big) + B\cdot\nabla \big(\nabla\times\partial_x^{k-1}(\mathcal{T}^\beta-\mathscr{T}^\beta)u \big).
    \end{aligned}
    \end{aligned}
    \end{equation*}
    Taking the inner products between $\nabla\times \partial_x^{k-1}\mathscr{T}^\beta u$ ($\nabla\times \partial_x^{k-1}\mathscr{T}^\beta b$ resp.) and the first (second resp.) equation of \eqref{vort-drv}, it follows
    \begin{equation}\label{curl-1}
    \begin{aligned}
        &\frac{\d}{2\d t} \int_\Omega 
        R|\nabla\times \partial_x^{k-1}\mathscr{T}^\beta u|^2 
        + (1+\frac{1}{Q}|B|^2) |\nabla\times \partial_x^{k-1}\mathscr{T}^\beta b|^2 \,\d x
        \\
         = & - \int_\Omega (\nabla\cdot B) (\nabla\times \partial_x^{k-1}\mathscr{T}^\beta b)\cdot (\nabla\times \partial_x^{k-1}\mathscr{T}^\beta u)\d x \\
        & + \frac{1}{2}\int_\Omega\left(\partial_t R+ \nabla\cdot(RU)\right) |\nabla\times \partial_x^{k-1}\mathscr{T}^\beta u|^2 \d x \\
        & + \frac{1}{2}\int_\Omega\left(\partial_t (\frac{1}{Q}|B|^2)+ \nabla\cdot\left((1+\frac{1}{Q}|B|^2)U\right)\right) |\nabla\times \partial_x^{k-1}\mathscr{T}^\beta b|^2 \d x\\
        & + \int_\Omega \left((\nabla\times \partial_x^{k-1}\mathscr{T}^\beta u)\cdot \mathcal{R}_1 + (\nabla\times \partial_x^{k-1}\mathscr{T}^\beta b)\cdot \mathcal{R}_2 \right)\d x.
    \end{aligned}
    \end{equation}
By applying  Lemmas~\ref{lemma-tangential}, \ref{lemma-Udot}, \ref{lemma-moser} and Corollary~\ref{coro-moser} for the right hand side of \eqref{curl-1}, one obtains
    \begin{equation}\label{curl-ub}
        \begin{aligned}
            &\quad\frac{\d}{2\d t}\int_\Omega R|\nabla\times \partial_x^{k-1}\mathscr{T}^\beta u(t)|^2 + (1+\frac{1}{Q}|B|^2)|\nabla\times \partial_x^{k-1}\mathscr{T}^\beta b(t)|^2 \d x\\
            & \le \big(\|\nabla\times\partial_x^{k-1}\mathscr{T}^\beta u(t)\|_0 + \|\nabla\times\partial_x^{k-1}\mathscr{T}^\beta b(t)\|_0\big)\cdot\big(\eta_m({F},Z,t)+\eta_m({z},Z,t)\big),
        \end{aligned}
    \end{equation}
    with the notation $\eta_m({z},Z,t)$ being given in \eqref{eta-def}. From \eqref{curl-ub}, one immediately gets
    \begin{equation}\label{curl-ub''}
    \begin{aligned}
        &\quad \|\nabla\times\partial_x^{k-1}\mathscr{T}^\beta u(t)\|_0 + \|\nabla\times\partial_x^{k-1}\mathscr{T}^\beta b(t)\|_0\\
        &\lesssim
        \normmm{{z}(0)}_{m,*}
        +\int_0^t \big(\eta_m({F},Z,\tau)+\eta_m({z},Z,\tau)\big)\d \tau
    \end{aligned}
    \end{equation}
   for all $0\le t\le T$. 
    
(5)    Combining estimates given in \eqref{grad-p}-\eqref{div-u}, \eqref{div-b''} and \eqref{curl-ub''}, and taking sum over all $1\le k\le [\frac{m}{2}], |\beta|\le m-2k$, it follows the estimate \eqref{est-normal''} given in Proposition \ref{2.2} immediately by using the obvious fact
    \begin{equation*}
    \begin{aligned}
        &\normmm{{f}(t)}_{s-1,*} \le \normmm{{f}(0)}_{s-1,*} + \int_0^t \normmm{{f}(\tau)}_{s,*} \d \tau
    \end{aligned}
    \end{equation*}
    for $F$ and $z$.

\end{proof}

Now, we can finish the proof of the a priori estimate given in Theorem~\ref{thm-apriori}.

\underline{\bf Proof of Theorem~\ref{thm-apriori}:}
First, for the entropy, by acting $\sum_{|\alpha|_*\le m}\partial_*^{\alpha}$ on the last equation of \eqref{linear}, we get immediately that  for any $\alpha\in \mathbb{N}^5, |\alpha|_*\le m$,
\begin{equation}\label{est-s}
\frac{\d}{\d t} \int_\Omega |\partial_*^{\alpha} s|^2 \d x \le \normmm{s(t)}_{m,*}\cdot
\big( \eta_m({F},Z,t) + \eta_m({z},Z,t) \big),
\end{equation}
which implies
\begin{equation}\label{est-s''}
\normmm{s(t)}_{m,*} \le \normmm{s(0)}_{m,*} + \int_0^t \big( \eta_m({F},Z,\tau) + \eta_m({z},Z,\tau) \big) \d \tau.
\end{equation}

Combining estimates given in \eqref{est-normal''}, \eqref{est-tan''} and \eqref{est-s''}, we obtain
\begin{equation}\label{est''}
\begin{aligned}
\normmm{z(t)}_{m,*} \lesssim 
\begin{cases}
\phi(M_{m}(Z,t))\big(\Lambda_{m}({z},{F},t) +
\int_0^t \normmm{{z}(\tau)}_{m,*}\d \tau \big), 
&\text{if }m\ge 7,\\
\begin{aligned}
&\phi_1(M_{m-1}(Z,t))\left(\Lambda_m(z,F,t) + \int_0^t \normmm{z(\tau)}_{m,*} \d \tau\right)\\
& \quad + \phi_2(M_{m}(Z,t))\int_0^t\left(\normmm{z(\tau)}_{m-1,*} + \normmm{F(\tau)}_{m-1,*}\right)\d\tau,
\end{aligned}
& \text{if }m\ge 8,
\end{cases}
\end{aligned}
\end{equation}
which implies the first claim \eqref{apriori} given in  Theorem~\ref{thm-apriori} directly by using Gronwall's inequality. 

When $m\ge 8$, applying Gronwall's inequality in the second case of \eqref{est''}, it gives
\begin{equation*}
\begin{aligned}
\normmm{z(t)}_{m,*} &\le \phi(M_{m-1}(Z,T)) \Big(
\Lambda_m (z, F,t) \\
&+ \phi_2(M_m(Z,t)) \int_0^t \big(\normmm{z(\tau)}_{m-1,*} + \normmm{F(\tau)}_{m-1,*}\big)\d\tau
\Big).
\end{aligned}
\end{equation*}
One can invoke \eqref{apriori} for $\normmm{z(\tau)}_{m-1,*}$ on the right-hand side of the above inequality and complete the proof of the second estimate \eqref{apriori-fine} given in  Theorem~\ref{thm-apriori}.

\section{Linearized problem with a smooth background state}\label{Section-wellposed}

The main proposal of this section is to establish the following well-posedness result of the linearized problem  \eqref{linear}-\eqref{icd-linear}. For that, first we shall obtain the existence of a weak solution by using the a priori estimates given in the previous section and the Riesz representation theory, then prove that this weak solution is indeed a strong one by several smoothing procedures.

\begin{theorem}\label{thm-LIBVP}
    For a fixed integer $m\ge 7$, assume that each angle on the boundary of the domain $\Omega$ obeys \eqref{angle-cdt} for all $1\le n\le N$, and the background state, the initial data, and source terms satisfy the conditions \eqref{assum-Z}, \eqref{z0F-cdt}, \eqref{b-icd}, and \eqref{F-cdt} respectively. Then, the linearized problem \eqref{linear}-\eqref{icd-linear} has a unique solution $z\in X^{m}_*([0,T_0]\times \Omega)$ for some $T_0\in (0,T]$ depending on $\Omega$ and $\mathcal{K}$, and the solution satisfies the estimate \eqref{apriori} for $0\le t\le T_0$; moreover, as $m\ge 8$, the estimate \eqref{apriori-fine} holds as well.
\end{theorem}

The proof of this theorem will be given in the following three subsections. 

\subsection{The associated boundary value problem}

The first step is to show the well-posedness of the boundary value problem \eqref{linear}-\eqref{bcd-linear}, posed for $t\in (-\infty, T)$. For the next proposition, we introduce the following slight modification of \eqref{assum-Z}.
\begin{equation}\label{assum-Z'}
    \begin{aligned}
         & Z \in C^\infty((-\infty,T]\times \Omega) \text{ and } 
         U\cdot \nu = B \cdot \nu = 0 \text{ on } (-\infty,T)\times (\partial\Omega\backslash \{q_1, \ldots, q_N\});\\
         & (U,B,P,S) \text{ takes its values in a fixed compact subset }\tilde{\mathcal{K}} \text{ of }\R^4\times \mathcal{V};\\
         & Z \text{ is constant, and } U=0, \text{ when }|t| \text{ is large}.
    \end{aligned}
\end{equation}
In fact, \eqref{assum-Z'} is easily achieved by means of a suitable extension of $Z$ satisfying \eqref{assum-Z}. 

\begin{prop}\label{prop-regularity}
    For a given integer $m\ge 2$, assume that $\omega_n \in (0,\frac{\pi}{[\frac{m}{2}]})$ for all $n=1,\ldots, N$, \eqref{assum-Z'} holds, and $F\in H^{m}_*((-\infty,T)\times \Omega)$ satisfies \eqref{F-cdt} and vanishes when $t<0$. Then, the problem \eqref{linear}-\eqref{bcd-linear} has a unique solution $z\in \big(X^{m}_*\cap H^{m}_*\big)((-\infty,T]\times\Omega)$ satisfying ${\cal B}' z = 0$ on $\partial\Omega \backslash\{q_1,\ldots,q_N\}$ and $z = 0 $ as $t<0$.
\end{prop}

The proof is divided into two main steps: first show that the problem is tangentially well-posed in a strong sense by using the Riesz representation theory, then use the divergence-curl system from the equations to obtain the normal regularity of the solution.

\subsubsection{Well-posedness in tangential spaces}

Consider the following boundary value problem for the symmetry hyperbolic system in $\R\times \Omega$,
\begin{equation}\label{gamma-prob}
\begin{cases}
    (\tilde{L} + \mathbb{B} + \gamma S_0 A_0) z = F, \quad \R\times \Omega,\\
    {\cal B} z = 0, \quad \R\times (\partial\Omega\backslash\{q_1,\ldots, q_N\}),
\end{cases}
\end{equation}
where the parameter $\gamma > 0$ will be chosen later, and $\tilde{L} = S_0 L$ is given in \eqref{symm}. 

The strong solution of the above problem is defined as

\begin{defn}[Strong solution]
    Let $m\ge 0$ be a fixed integer, for given $F\in H^m_{\tg}(\R\times \Omega)$, we say that $z\in H^m_{\tg}(\R \times \Omega)$ is a strong solution of the boundary value problem \eqref{gamma-prob} if there exists sequences $\{z_\varepsilon\}_{\varepsilon>0}$ such that 
    \begin{equation}\label{strong}
        \begin{aligned}
            & z_\varepsilon \in \cap_{s\in \mathbb{N}} H^s(\R\times \Omega),\, {\cal B} z_\varepsilon = 0 \text{ on }\partial\Omega\backslash\{q_1,\cdots, q_N\},\\&  \text{and } z_\varepsilon\to z, \quad  (\tilde{L} + \mathbb{B} + \gamma S_0 A_0)z_\varepsilon \to F \text{ in } H^m_{\tg}(\R\times \Omega).
        \end{aligned}
    \end{equation}
\end{defn}

For fixed $m\in \mathbb{N}$ and parameter $\gamma>0$, the tangential space  $H^{m,\gamma}_\tg(I\times \Omega)$ is the one containing elements with the norm
\begin{equation*}
\begin{aligned}
    \|z\|_{H^{m,\gamma}_\tg(I\times \Omega)} = \sum_{|\alpha|\le m} \gamma^{m-|\alpha|} \|\mathbb{T}^\alpha z\|_{L^2(I\times \Omega)}
\end{aligned}
\end{equation*}
being finite, and its inner product is defined as
\begin{equation*}
    \inp*{z}{w}_{H^{m,\gamma}_\tg(I\times \Omega)} = \sum_{|\alpha|\le m} \gamma^{2(m-|\alpha|)} \inp*{\mathbb{T}^\alpha z}{\mathbb{T}^\alpha w}_{L^2(I\times \Omega)}, \quad
    \forall z,w\in H^{m,\gamma}_\tg(I\times \Omega).
\end{equation*}

\begin{prop}\label{prop-tan-well}
    For a fixed $m\in \mathbb{N}$, assume that \eqref{assum-Z'} holds, then there exists $\gamma_0>0$, depending on $Z$, such that 
    for any given $F\in H^m_{\tg}(\R\times \Omega)$, the boundary value problem \eqref{gamma-prob} has a unique strong solution $z \in CH^m_{\tg}\cap H^m_{\tg} (\R\times \Omega)$, satisfying
    \begin{equation}\label{z-tan-est}
    \sum_{|\alpha|\le m}\gamma^{2(m-|\alpha|)}\|\mathbb{T}^\alpha z(t)\|^2_{L^2(\Omega)}
    +\gamma \|z\|^2_{H^{m,\gamma}_{\tg}(\R\times \Omega)}\le C\frac{1}{\gamma}\|F\|^2_{H^{m,\gamma}_{\tg}(\R\times\Omega)}, \quad \forall t\in \R,
    \end{equation}
    for any $\gamma \ge \gamma_0$.
    Moreover, if $F$ vanishes as $t\le T_1$ for any fixed $T_1\in \R$, then the solution $z$ also vanishes for $t\le T_1$.
\end{prop}

Since in domains with corners, the boundary value problem for linear hyperbolic systems -- even in the non-characteristic case -- has no general result, we do not adapt the commonly used non-characteristic regularization approach as did in \cite{secchi1996initial} and \cite{yanagisawa1991fixed}. Instead, we shall employ a direct duality argument within high-order Sobolev spaces to get the existence of solutions. The proof of Proposition~\ref{prop-tan-well} contains three main steps: First, use the Lax-Phillips duality method to get the existence of weak solutions to the problem \eqref{gamma-prob}. Then, show that the weak solutions are, indeed, strong solutions, through a series of mollifications, and thus the a priori estimates are valid.
Finally, we finish the proof by applying the causality principles.

First, we prove the following a priori estimates in $H^{m,\gamma}_{\tg}(\R\times \Omega)$.
\begin{prop}\label{prop-apriori-tan}
    Assume that $Z$ satisfies \eqref{assum-Z'}, then for the problem \eqref{gamma-prob}, one can find constants $\gamma_0, C >0$, depending on $Z$, such that for any $\gamma \ge \gamma_0$, if $z\in H^{m,\gamma}_{\tg}(\R\times \Omega)$ and $\partial_{t,x} z\in H^{m,\gamma}_{\tg}(\R\times \Omega)$, then one has
    \begin{equation}\label{3.5}
        \sum_{|\alpha|\le m}\gamma^{2(m-|\alpha|)}\|\mathbb{T}^\alpha z(t)\|^2_{L^2(\Omega)} + \gamma \|z\|^2_{H^{m,\gamma}_{\tg}(\R\times \Omega)}\le C\frac{1}{\gamma}\|(\tilde{L}+\mathbb{B}+\gamma S_0 A_0)z\|^2_{H^{m,\gamma}_{\tg}(\R\times\Omega)},
`    \end{equation}
    for all $t\in \R$.
\end{prop}

\begin{proof}
    Let $F = (\tilde{L}+\mathbb{B} + \gamma S_0 A_0) z$, obviously one has
    \begin{equation*}
\begin{aligned}
    (\tilde{L}+\mathbb{B} + \gamma S_0 A_0)(Z) \mathbb{T}^\alpha z = &
    \mathbb{T}^\alpha F + [(S_0 A_0)(Z) \partial_t, \mathbb{T}^\alpha] z + \gamma[(S_0 A_0)(Z), \mathbb{T}^\alpha] z \\&+ [A_1^I(Z) \partial_1 + A_2^I (Z)\partial_2 + \mathbb{B}(Z), \mathbb{T}^\alpha]
    + [A_1^{II}\partial_1 + A_2^{II}\partial_2, \mathbb{T}^\alpha] z,
\end{aligned}
\end{equation*}
with $A_k^I, A_k^{II}$ ($k=1,2$) being defined in \eqref{A-I}-\eqref{A-II}.
    Similar to the estimate of tangential derivatives in the proof of Theorem~\ref{thm-apriori}, by taking the inner product of the above equation with $\mathbb{T}^\alpha z(t)$, we get
    \begin{equation*}
        \begin{aligned}
            &\quad\frac{\d}{2 \d t} \inp*{\mathbb{T}^\alpha z(t)}{(S_0 A_0)(Z) \mathbb{T}^\alpha z(t)}_{\Omega} + \gamma \inp*{\mathbb{T}^\alpha z(t)}{(S_0 A_0)(Z) \mathbb{T}^\alpha z(t)}_{\Omega}\\
            &= \inp*{\mathbb{T}^\alpha z(t)}{\mathbb{T}^\alpha {F}(t) + [(S_0 A_0)(Z) \partial_t + A_1^{I}(Z)\partial_1 + A_2^{I}(Z) \partial_2, \mathbb{T}^\alpha] z(t)}_{\Omega}\\
            &+\inp*{\mathbb{T}^\alpha z(t)}{ [ \mathbb{B}(Z)+\gamma(S_0 A_0)(Z), \mathbb{T}^\alpha]z(t)}_\Omega\\
            &+ \inp*{\mathbb{T}^\alpha z(t)}{
             (\frac{1}{2}(\partial_t(S_0 A_0) + \partial_1(S_0 A_1) + \partial_2 (S_0 A_2))-\mathbb{B})(Z)\mathbb{T}^\alpha z(t)}_{\Omega}\\
            &+ \sum_{\substack{|\beta|\le |\alpha|-1\\|\sigma|\le |\alpha|}}\sum_{i,l=1,2}\inp*{\Theta^{\beta,\sigma}(x) \mathbb{T}^\beta \left((F_1)_i - R(\partial_t + U\cdot \nabla)u_i + B\cdot \nabla b_i\right)(t)}{\mathbb{T}^\sigma u_l(t)}_\Omega,
        \end{aligned}
    \end{equation*}
    where $\Theta^{\beta,\sigma}\in C^\infty(\overline{\Omega})$.  By identifying $U\cdot \nabla$ with $\sum_{i=1,2} V_i \partial_{w^i}$, $B\cdot \nabla$ with $\sum_{i=1,2} \overline{B}_i \partial_{w^i}$ as in Lemma~\ref{lemma-tangential}, and using \eqref{tan-cmtt-2} and the Cauchy inequality, it follows there exists $C>0$, depending on $Z$, such that for any $K>0$,
    \begin{equation}\label{tan-gamma'}
        \begin{aligned}
           & \quad \|\mathbb{T}^\alpha z(t)\|^2_{L^2(\Omega)} + \gamma \|\mathbb{T}^\alpha z\|^2_{L^2(\R\times\Omega)}
            \le \frac{\gamma}{2 K} \sum_{|\beta|\le |\alpha|}\|\mathbb{T}^\beta z\|^2_{L^2(\R\times\Omega)}\\
           &+ C \|\mathbb{T}^\alpha z\|_{L^2(\R\times \Omega)}\sum_{|\beta|\le |\alpha|-1} \gamma \|\mathbb{T}^\beta z\|_{L^2(\R\times \Omega)}
           + K \sum_{|\beta|\le |\alpha|}
            \left( \frac{1}{\gamma}\|\mathbb{T}^\beta F\|^2_{L^2(\R\times\Omega)} + C \|\mathbb{T}^\beta z\|^2_{L^2(\R\times\Omega)}\right) 
        \end{aligned}
    \end{equation}
    holds. 
    By multiplying \eqref{tan-gamma'} with $\gamma^{m-|\alpha|}$, and summing up for all $\alpha\in \mathbb{N}^3, |\alpha| \le m$, the estimate \eqref{3.5} follows easily by choosing $K,\gamma$ sufficiently large in order.
\end{proof}

To obtain the existence of weak solutions to \eqref{gamma-prob} in $H^{m,\gamma}_{\tg}(\R\times \Omega)$, we shall first study the adjoint boundary-value problem of \eqref{gamma-prob} in $H^{m,\gamma}_{\tg}(\R\times \Omega)$. 

Denote by $\tilde{L}^* = - \partial_t\big( (S_0 A_0) \cdot\big) - \partial_1\big((S_0 A_1) \cdot\big) - \partial_2\big((S_0 A_2)\cdot \big)$ the adjoint operator of $\tilde{L}$ in $L^2(\R\times \Omega)$.
For any fixed $s\in \mathbb{N}$, introduce the spaces 
\begin{equation*}
     H^1H^{s,\gamma}_{\tg}(\R\times\Omega) : = \{z\in H^{s,\gamma}_{\tg}(\R\times\Omega): \partial_t z, \partial_x z \in H^{s,\gamma}_{\tg}(\R\times\Omega)\},
\end{equation*}
equipped with the norm $\|z\|_{H^{s,\gamma}_{\tg}(\R\times\Omega)} + \|\partial_{t,x} z\|_{H^{s,\gamma}_{\tg}(\R\times\Omega)}$.

Next aim is to establish a Green identity for the operator $\tilde{L}$ in the space $H^{m,\gamma}_{\tg}(\R\times \Omega)$, where, inspired by \cite{benoit2025regular}, we shall use the Riesz representation theorem to represent the commutators.

\begin{prop}[Green's identity]\label{prop-tan-Green}
    For  a fixed integer $m\in \mathbb{N}$, there exist a constant $C>0$ independent of  $\gamma$, and three linear bounded opertors $\phi_1, \phi_2,\varphi \in \mathcal{L}(H^{m,\gamma}_{\tg}(\R\times \Omega))$, satisfying $\|\phi_1\|_{\mathcal{L}(H^{m,\gamma}_{\tg})}+ \|\phi_2\|_{\mathcal{L}(H^{m,\gamma}_{\tg})} \le \frac{1}{\gamma} C$, $\|\varphi\|_{\mathcal{L}(H^{m,\gamma}_{\tg})}\le C$, such that for any $z = (u_z,b_z,\p_z,s_z)^t, w = (u_w,b_w,\p_w,s_w)^t \in H^1 H^{m,\gamma}_{\tg}(\R\times\Omega)$, one has the following identity,
    \begin{equation}\label{tan-grn}
    \begin{aligned}
        &\inp*{(\tilde{L}+\mathbb{B}+\gamma S_0A_0)z}{(1-\phi_1)w}_{H^{m,\gamma}_\tg}  = 
        \inp*{z}{\big((1+\phi_2)\tilde{L}^* +\mathbb{B}^t+\gamma S_0 A_0 + \varphi\big) w}_{H^{m,\gamma}_\tg} \\
        &\qquad
        +\sum_{|\alpha|\le m} \gamma^{2(m-|\alpha|)}
        \Big(\inp*{(\mathbb{T}^\alpha u_z) \cdot \nu}{(\mathbb{T}^\alpha u_w) \cdot \nu + \mathbb{T}^\alpha \p_w}_{L^2(\partial\Omega)} \\
        &\qquad
        + \inp*{- (\mathbb{T}^\alpha u_z) \cdot \nu + \mathbb{T}^\alpha \p_z}{(\mathbb{T}^\alpha u_w)\cdot \nu}_{L^2(\partial\Omega)} \Big).
    \end{aligned}
\end{equation}
\end{prop}
\begin{proof}
By integration by parts, one has
\begin{equation*}
    \begin{aligned}
        &\inp*{(\tilde{L}+\mathbb{B}+\gamma S_0A_0)z}{w}_{H^{m,\gamma}_\tg}  = \sum_{|\alpha|\le m} \gamma^{2(m-|\alpha|)} \inp*{\mathbb{T}^\alpha (\tilde{L}+\mathbb{B}+\gamma S_0 A_0) z}{\mathbb{T}^\alpha w}_{L^2}\\
        &\quad =\sum_{|\alpha|\le m} \gamma^{2(m-|\alpha|)} \Big(
        \inp*{\mathbb{T}^\alpha z}{\mathbb{T}^\alpha (\tilde{L}^*+\mathbb{B}^t+\gamma S_0A_0) w}_{L^2} + \inp*{A_\nu \mathbb{T}^\alpha z}{\mathbb{T}^\alpha w}_{L^2(\partial\Omega)} \\
        &\quad + \underbrace{\inp*{[\mathbb{T}^\alpha, \tilde{L}+\mathbb{B}+\gamma S_0A_0]z}{\mathbb{T}^\alpha w}_{L^2}}_{\mathcal{I}_1} + \underbrace{\inp*{\mathbb{T}^\alpha z}{[\tilde{L}^*+\mathbb{B}^t+\gamma S_0 A_0, \mathbb{T}^\alpha]w}_{L^2}}_{\mathcal{I}_2}\Big).
    \end{aligned}
\end{equation*}
By a direct computation, it holds
\begin{equation*}
    \begin{aligned}
        \inp*{A_\nu \mathbb{T}^\alpha z}{\mathbb{T}^\alpha w}_{L^2(\partial\Omega)} & = 
        \inp*{(\mathbb{T}^\alpha u_z) \cdot \nu}{(\mathbb{T}^\alpha u_w) \cdot \nu + \mathbb{T}^\alpha \p_w}_{L^2(\partial\Omega)} \\
        & + \inp*{- (\mathbb{T}^\alpha u_z) \cdot \nu + \mathbb{T}^\alpha \p_z}{(\mathbb{T}^\alpha u_w)\cdot \nu}_{L^2(\partial\Omega)},
    \end{aligned}
\end{equation*}
and
\begin{equation}\label{3.8}
    \begin{aligned}
        \mathcal{I}_1 + \mathcal{I}_2 &= \sum_{|\alpha|\le m}\gamma^{2(m-|\alpha|)} \Bigg(
        \underbrace{\inp*{[\mathbb{T}^\alpha, S_0A_0 \partial_t + \sum_{i=1,2} A_i^{I}\partial_i+\mathbb{B}+\gamma S_0A_0]z}{\mathbb{T}^\alpha w}_{L^2}}_{\mathcal{J}_1} \\
        &+ \underbrace{\inp*{\mathbb{T}^\alpha z}{[-(S_0 A_0 \partial_t + \sum_{i=1,2} A_i^{I}\partial_i)+\mathbb{B}^t+\gamma S_0 A_0, \mathbb{T}^\alpha]w}_{L^2}}_{\mathcal{J}_2}\\
        &+ \underbrace{\inp*{[\mathbb{T}^\alpha, \sum_{i=1,2} A_i^{II}\partial_i]z}{\mathbb{T}^\alpha w}_{L^2}}_{\mathcal{J}_3} 
        - \underbrace{\inp*{\mathbb{T}^\alpha z}{[\sum_{i=1,2} A_i^{II}\partial_i,\mathbb{T}^\alpha]w}_{L^2}}_{\mathcal{J}_4}\Bigg),
    \end{aligned}
\end{equation}
where $A^I_1, A^I_2$ and $A^{II}_1, A^{II}_2$ are defined in \eqref{A-I} and \eqref{A-II} respectively.
Since by \eqref{tan-cmtt-2}, $A_1^{I}\partial_1 + A_2^{I}\partial_2$ contains only  tangential derivatives , one has that $\mathcal{J}_1, \mathcal{J}_2$ are bilinear continuous on $H^{m,\gamma}_\tg \times H^{m,\gamma}_\tg$ with respect to $z$ and $w$, and in particular by using the Cauchy-Schwarz inequality, there exists a constant $C$ independent of $z,w$ and $\gamma$, such that
\begin{equation*}
    |\mathcal{J}_1|+ |\mathcal{J}_2| \le C \|z\|_{H^{m,\gamma}_\tg} \|w\|_{H^{m,\gamma}_\tg}.
\end{equation*}
Consequently, using the Riesz representation theorem, there exist two continuous linear operators $\varphi_1$ and $\varphi_2 : H^{m,\gamma}_\tg\to H^{m,\gamma}_\tg$ such that
\begin{equation*}
    \mathcal{J}_1 = \inp*{z}{\varphi_1(w)}_{H^{m,\gamma}_\tg}, \quad \mathcal{J}_2 = \inp*{z}{\varphi_2(w)}_{H^{m,\gamma}_\tg},
\end{equation*}
with the operator norms $\|\varphi_1\|_{\mathcal{L}(H^{m,\gamma}_\tg)}, \|\varphi_2\|_{\mathcal{L}(H^{m,\gamma}_\tg)} \le C$.

As for $\mathcal{J}_3$ and $\mathcal{J}_4$, we can explicitly compute them from the expression of $A_1^{II}\partial_1 + A_2^{II}\partial_2$ given in \eqref{A-II}. For example,
\begin{equation*}
\begin{aligned}
    \mathcal{J}_3 = \sum_{|\alpha|\le m} \gamma^{2(m-|\alpha|)} \Big(&\inp*{\mathscr{T}^\alpha \nabla \p_z - \nabla \mathcal{T}^\alpha \p_z}{\mathscr{T}^\alpha u_w}_{L^2}\\
    & +\inp*{\mathcal{T}^\alpha (\nabla\cdot u_z) - \nabla\cdot \mathscr{T}^\alpha u_z}{\mathcal{T}^\alpha \p_w}_{L^2}
    \Big).
\end{aligned}
\end{equation*}
Then by Lemma~\ref{lemma-J2}, one has
\begin{equation*}
\begin{aligned}
    \mathcal{J}_3 = \sum_{|\alpha|\le m} \gamma^{2(m-|\alpha|)}
    \sum_{\substack{|\beta|\le |\alpha|-1 \\ |\sigma|\le |\alpha|}} \sum_{i,l=1,2}
        \Bigg(&\inp*{\Theta_1^{\beta,\sigma}(x)\mathcal{T}^\beta \partial_i \p_z}{\mathcal{T}^\sigma (u_l)_w}_{L^2}
        \\
    +&\inp*{\Theta_2^{\beta,\sigma}(x)\mathcal{T}^\beta \partial_i \p_w}{\mathcal{T}^\sigma (u_l)_z}_{L^2}\Bigg),
    \end{aligned}
\end{equation*}
with some $\Theta_1^{\beta,\sigma}, \Theta_2^{\beta,\sigma}\in C^\infty(\overline{\Omega})$.
Invoking
\begin{equation*}
    \begin{aligned}
        &\partial_i \p_z = (\tilde{L}z)_i - R(\partial_t + U\cdot\nabla) (u_i)_z + B\cdot \nabla (b_i)_z,\\
        &\partial_i \p_w = -(\tilde{L}^* w)_i - R(\partial_t + U\cdot\nabla) (u_i)_w + B\cdot \nabla (b_i)_w,
    \end{aligned}
\end{equation*}
we can write $\mathcal{J}_3 = \mathcal{J}_3^\flat + \mathcal{J}_3^{\sharp,1} + \mathcal{J}_3^{\sharp,2}$, where
\begin{equation*}
\begin{aligned}
    &\mathcal{J}_3^\flat = \sum_{|\alpha|\le m}\gamma^{2(m-|\alpha|)} \sum_{|\beta|,|\sigma|\le |\alpha|} \inp*{\Theta^{\beta,\sigma}_{(1)} (x,Z)\mathcal{T}^\beta z}{\mathcal{T}^\sigma w}_{L^2},
    \\
    &\mathcal{J}_3^{\sharp,1} = \sum_{|\alpha|\le m}\gamma^{2(m-|\alpha|)} \sum_{\substack{|\beta|\le |\alpha|-1 \\|\sigma|\le |\alpha|}} \inp*{\Theta^{\beta,\sigma}_{(2)}(x) \mathcal{T}^\beta (\tilde{L} z)}{\mathcal{T}^\sigma w}_{L^2}, \\ 
    &\mathcal{J}_3^{\sharp,2} = \sum_{|\alpha|\le m}\gamma^{2(m-|\alpha|)} \sum_{\substack{|\beta|\le |\alpha|\\|\sigma|\le |\alpha|-1}} \inp*{\Theta^{\beta,\sigma}_{(3)}(x) \mathcal{T}^\beta z}{\mathcal{T}^\sigma (\tilde{L}^* w)}_{L^2},
\end{aligned}
\end{equation*}
for some $\Theta^{\beta,\sigma}_{(1)}\in C^\infty(\overline{\Omega}\times \mathcal{K})$, $\Theta^{\beta,\sigma}_{(2)}, \Theta^{\beta,\sigma}_{(3)} \in C^\infty(\overline{\Omega})$.
Then one can write $\mathcal{J}^\flat_3 = \inp*{z}{\varphi_3(w)}_{H^{m,\gamma}_\tg}$ for a $\varphi_3 \in \mathcal{L}(H^{m,\gamma}_\tg)$ with $\|\varphi_3\|_{\mathcal{L}(H^{m,\gamma}_\tg)}\le C$ for some $C>0$ independent of $(z,w,\gamma)$. On the other hand, by using the Cauchy-Schwarz inequality,
\begin{equation*}
    |\mathcal{J}_3^{\sharp,1}| \le C \|\tilde{L} z\|_{H^{m-1,\gamma}_\tg}\|w\|_{H^{m,\gamma}_\tg}\le \frac{1}{\gamma}C \|\tilde{L} z\|_{H^{m,\gamma}_\tg}\|w\|_{H^{m,\gamma}_\tg}.
\end{equation*}
Thus, $(\tilde{L} z,w)\mapsto \mathcal{J}_3^{\sharp,1}$ is a bilinear continuous form on $H^{m,\gamma}_\tg\times H^{m,\gamma}_\tg$. Again, by using the Riesz representation theorem, there exists $\phi_1 \in \mathcal{L}(H^{m,\gamma}_\tg)$ with $\|\phi_1\|_{\mathcal{L}(H^{m,\gamma}_\tg)}\le \frac{1}{\gamma}C$ for some $C>0$ independent of $(z,w,\gamma)$, such that
\begin{equation*}
     \mathcal{J}_3^{\sharp,1} = \inp*{\tilde{L} z}{\phi_1(w)}_{H^{m,\gamma}_\tg}.
\end{equation*}
Similarly, one can have
\begin{equation*}
    \mathcal{J}_3^{\sharp,2} = \inp*{z}{\phi_2(\tilde{L}^* (w))}_{H^{m,\gamma}_\tg}
\end{equation*}
for some $\phi_2 \in \mathcal{L}(H^{m,\gamma}_\tg)$ with $\|\phi_2\|_{\mathcal{L}(H^{m,\gamma}_\tg)}\le \frac{1}{\gamma}C$.

The analysis for $\mathcal{J}_4$ given in \eqref{3.8} is completely similar as the above for $\mathcal{J}_4$. 

Summarizing the above calcluations, one can conclude the Green identity given in \eqref{tan-grn}, with some new $\phi_1,\phi_2,\varphi\in \mathcal{L}(H^{m,\gamma}_\tg)$, satisfying $\|\phi_1\|_{\mathcal{L}(H^{m,\gamma}_\tg)},\|\phi_2\|_{\mathcal{L}(H^{m,\gamma}_\tg)} \le \frac{1}{\gamma} C$ and $\|\varphi\|_{\mathcal{L}(H^{m,\gamma}_\tg)}\le C$.
\end{proof}

From $\|\phi_1\|_{\mathcal{L}(H^{m,\gamma}_\tg)}\le \frac{1}{\gamma} C$, it immediately implies that when $\gamma$ is properly large, $(1-\phi_1)$ is a bijection from $H^{m,\gamma}_{\tg}(\R\times \Omega)$ to itself. So, from \eqref{tan-grn}, we know that 
\begin{equation}\label{gamma-prob-1}
\begin{cases}
    \left((1+\phi_2)\tilde{L}^* + \mathbb{B}^t + \gamma S_0 A_0 + \varphi\right)
    (1-\phi_1)^{-1} w = G, \quad \R\times \Omega,\\
    {\cal B}(1-\phi_1)^{-1}w = 0, \quad \R\times (\partial\Omega\backslash\{q_1,\ldots, q_N\}),
\end{cases}
\end{equation}
is the adjoint boundary-value problem of \eqref{gamma-prob}
in $H^{m,\gamma}_{\tg}(\R\times \Omega)$.

In particular, when $m=0$, we have $\phi_1 = \phi_2 = \varphi = 0$, then for any $z,w\in H^1(\R\times \Omega)$
\begin{equation}\label{L2-grn}
    \inp*{\tilde{L} z}{w}_{L^2(\R\times \Omega)} = \inp*{z}{\tilde{L}^* w}_{L^2(\R\times \Omega)} + \inp*{A_\nu z}{ w}_{L^2(\R\times\partial\Omega)}.
\end{equation}
Let $\mathcal{D}L^2(\R\times \Omega): = \{z\in L^2(\R\times \Omega): \tilde{L} z\in L^2(\R\times \Omega)\}$, equipped with the norm $\|z\|_{L^2(\R\times \Omega)} + \|\tilde{L}z\|_{L^2(\R\times \Omega)}$. We extend the $L^2$ Green's identity \eqref{L2-grn} to be as follows.
\begin{lemma}
    The map
    \begin{equation*}
        H^1(\R\times \Omega) \ni z\mapsto A_\nu z_{|\partial\Omega\backslash\{q_1,\ldots, q_N\}} 
    \end{equation*}
    extends uniquely to a continuous linear map from $\mathcal{D} L^2(\R\times\Omega)$ to $H^{-\frac{1}{2}} (\R\times \partial\Omega)$, and the Green identity in $L^2(\R\times \Omega)$ given in \eqref{L2-grn} holds for $z\in \mathcal{D}L^2(\R\times \Omega)$, $w\in H^1 (\R\times\Omega)$.
\end{lemma}
\begin{proof}
    According to \cite[Theorem~1.5.1.3]{grisvard1985elliptic}, for any $\varphi \in H^{1/2}(\R\times \partial\Omega)$, there exists $\Phi\in H^1(\R\times \Omega)$, such that $\Phi|_{\partial\Omega} = \varphi$, and $\|\Phi\|_{H^1}\lesssim \|\varphi\|_{H^{1/2}}$.
    Then for any $z\in H^1(\R\times \Omega)$,
    \begin{equation}\label{Anu-trace}
    \inp*{A_\nu z}{\varphi}_{L^2(\R\times \partial\Omega)} = \inp*{\tilde{L}z}{\Phi}_{L^2(\R\times \Omega)} - \inp*{z}{\tilde{L}^* \Phi}_{L^2(\R\times \Omega)}
    \end{equation}
    shows that $\|A_\nu z_{|\partial\Omega\backslash\{q_1,\ldots, q_N\}} \|_{H^{-\frac{1}{2}}(\R\times\partial\Omega)} \lesssim \|z\|_{\mathcal{D}L^2(\R\times \Omega)}$, which proves the existence of a continuous extension claimed in the lemma. Also, it is easy to see that $H^1(\R\times \Omega)$ is dense in $\mathcal{D}L^2(\R\times \Omega)$, so the extension is unique.
\end{proof}

The Green formula given in \eqref{tan-grn} motivates the definition of weak solutions of the problem \eqref{gamma-prob} in $H^{m,\gamma}_{\tg}(\R\times \Omega)$.
\begin{defn}[Weak solution]\label{defn-weak}
    For a given $F\in H^{m,\gamma}_{\tg}(\R\times \Omega)$ with a fixed integer $m\in {\mathbb N}$, we say that $z\in H^{m,\gamma}_{\tg}(\R\times \Omega)$ is a weak solution of \eqref{gamma-prob} if for all $w \in C^\infty_0(\R\times \overline{\Omega})$ such that ${\cal B} w = 0$ on $\R\times (\partial\Omega\backslash\{q_1, \ldots, q_N\})$, one has
    \begin{equation}\label{weak-sol} 
        \inp*{z}{\big((1+\phi_2)\tilde{L}^* +\mathbb{B}^t+\gamma S_0 A_0 + \varphi\big) w}_{H^{m,\gamma}_\tg} 
        = \inp*{F}{(1-\phi_1)w}_{H^{m,\gamma}_\tg}.
    \end{equation}
\end{defn}
Now we introduce the spaces
\begin{equation*}
    \mathcal{D} H^{m,\gamma}_{\tg}(\R\times \Omega) := \{z\in H^{m,\gamma}_{\tg}(\R\times \Omega): \tilde{L}z \in H^{m,\gamma}_{\tg}(\R\times\Omega)\}
\end{equation*}

The following properties of the weak solutions are obtained immediately.
\begin{lemma}
   For a given $F\in H^{m,\gamma}_{\tg}(\R\times \Omega)$, assume that $z\in H^{m,\gamma}_{\tg}(\R\times \Omega)$ is a weak solution of \eqref{gamma-prob}, then
    \begin{enumerate}
        \item[(1)] $z \in \mathcal{D} H^{m,\gamma}_{\tg}$, and $\tilde{L} z = F-\mathbb{B}z - \gamma S_0 A_0 z$ in the sense of distributions on $\R\times \Omega$,
        \item[(2)] ${\cal B} z = 0$ on $\R\times \partial\Omega\backslash\{q_1,\ldots, q_N\}$. 
    \end{enumerate}
\end{lemma}
\begin{proof}
The first assertion given in (1) is straightforward from \eqref{weak-sol}. 
Moreover, from \eqref{weak-sol}, for any $\Phi = (\Phi_1,\ldots,\Phi_6)^t \in C^\infty_c(\R\times \Omega)$ such that ${\cal B}\Phi = (\Phi_1,\Phi_2)^t\cdot\nu = 0$ on $\partial\Omega\backslash\{q_1,\ldots, q_N\}$, one has
    \begin{equation*}
    \begin{aligned}
        0 = \inp*{A_\nu z}{\Phi}_{(H^{-\frac{1}{2}}\times H^{\frac{1}{2}})(\R\times\partial\Omega)} = \inp*{u \cdot \nu}{\Phi_5},
    \end{aligned}
    \end{equation*}
which implies $u\cdot \nu = 0$ on $\partial\Omega \backslash\{q_1,\ldots, q_N\}$.
\end{proof}

Now we prove the existence of weak solutions of \eqref{gamma-prob} in $H^{m,\gamma}_{\tg}(\R\times \Omega)$ given in the following proposition.
\begin{prop}\label{prop-tan-weak}
    There exists $\gamma_0> 0$, depending on $Z$, such that for any fixed $\gamma\ge \gamma_0$, and given $F\in H^{m,\gamma}_{\tg}(\R\times \Omega)$, the problem \eqref{gamma-prob} has a weak solution $z$ in $H^{m,\gamma}_{\tg}(\R\times \Omega)$.
\end{prop}
\begin{proof}
    Denote by the space 
    \begin{equation*}
        \mathcal{H} : = \left\{w\in H^1 H^{m,\gamma}_\tg(\R\times \Omega): {\cal B} w ={\cal B'} w = 0 \text{ on }\partial\Omega\backslash\{q_1,\ldots,q_N\}\right\}.
    \end{equation*}
    For any $w\in \mathcal{H}$ and $\alpha\in \mathbb{N}^3, |\alpha|\le m$, one has ${\cal B}\mathbb{T}^\alpha w = 0$ on $\partial\Omega\backslash\{q_1,\ldots,q_N\}$.

    Let the operator
    \begin{equation*}
        \underline{L}^*[\phi_2] := (1+\phi_2)\tilde{L}^* + \mathbb{B}^t + \gamma S_0 A_0 + \varphi,
    \end{equation*}
with $\phi_2\in {\cal L}(H^{m,\gamma}_{tan}(\R\times\Omega))$ being given in Proposition \ref{prop-tan-Green}, and
    $\mathcal{H}_1 = \underline{L}^*[\phi_2]\mathcal{H}$, a subspace of $H^{m,\gamma}_\tg(\R\times \Omega)$.

    Since the leading terms of $\underline{L}^*[0]$ coincide with those of $\tilde{L}$, Proposition~\ref{prop-apriori-tan} shows that for large $\gamma$, it holds that for any $w\in \mathcal{H}$, 
    \begin{equation*}
        \|w\|_{H^{m,\gamma}_\tg} \lesssim \frac{1}{\gamma} \|\underline{L}^*[0] w\|_{H^{m,\gamma}_\tg}.
    \end{equation*}
    Since $\phi_2 \in \mathcal{L}(H^{m,\gamma}_\tg)$ and $\|\phi_2\|_{\mathcal{L}(H^{m,\gamma}_\tg)} \le \frac{1}{\gamma}C$, so for sufficiently large $\gamma$, we still have
    \begin{equation*}
        \|w\|_{H^{m,\gamma}_\tg} \lesssim \frac{1}{\gamma} \|\underline{L}^*[\phi_2]w\|_{H^{m,\gamma}_\tg},
    \end{equation*}
  which shows that the mapping $\underline{L}^*[\phi_2]: \mathcal{H} \to \mathcal{H}_1$ is one to one, and the reciprocal mapping $\mathcal{F}$ satisfies
    \begin{equation*}
        \|\mathcal{F} \omega\|_{H^{m,\gamma}_\tg} \lesssim \frac{1}{\gamma}\|\omega\|_{H^{m,\gamma}_\tg}, \quad  \forall \omega \in \mathcal{H}_1.
    \end{equation*}
    Thus, the linear functional
    \begin{equation*}
        \omega \mapsto \ell (\omega) : = \inp*{F}{(1-\phi_1)\mathcal{F}\omega}_{H^{m,\gamma}_\tg}
    \end{equation*}
    is continuous on $\mathcal{H}_1$ equipped with the norm $\|\cdot\|_{H^{m,\gamma}_\tg}$. Therefore, by using the Hahn-Banach theorem it can be extended to be a continuous linear functional on $H^{m,\gamma}_\tg$, and by the Riesz representation theorem, there is $z\in H^{m,\gamma}_\tg$ such that $\ell (\omega) = \inp*{z}{\omega}_{H^{m,\gamma}_\tg}$. Especially, by restricting the functional $\ell$ on $\mathcal{H}_1$ one deduces that \eqref{weak-sol} holds for all $w\in \mathcal{H}$, so $z$ is a weak solution of \eqref{gamma-prob}.
\end{proof}

The next important goal is to verify that the above weak solution of \eqref{gamma-prob} is indeed a strong one, which yields that it is unique and satisfies the corresponding energy estimates. 

\begin{prop}[Weak = strong]\label{prop-weak=strong}
    Under the same assumptions as given in Proposition~\ref{prop-tan-weak}, any weak solution of \eqref{gamma-prob} defined by Definition~\ref{defn-weak} is a strong solution.
\end{prop}
By using a standard partition of unity and coordinate transformation, we study this problem for three cases. In any small neighborhood of an interior point of $\Omega$, that any weak solution being a strong one can be easily obtained by adapting classical mollification procedure, as no any boundary condition is involved. 
In a small neighborhood of any fixed point on smooth portions of $\partial\Omega$, as discussed in Remark~\ref{rmk-sp}, the space $H^m_{\tg}(\R\times \Omega)$ coincides with the classical anisotropic Sobolev space, and the small neighborhood can be transformed to $\{x_1>0, x_2\in \R\}$ with the boundary condition $u_1 = 0$ on $x_1 = 0$. 
The proof in this case follows from \cite[Theorem~4]{rauch1985symmetric}, which proceeds in two main steps: first, one performs a tangential mollification (in $\partial_t, x_1\partial_1, \partial_2$) to obtain a sequence $\{z_\varepsilon\}_{\varepsilon >0}$ that is tangentially smooth, satisfies $(u_1)_\varepsilon = 0$ on $x_1 = 0$, and converges as
$z_\varepsilon \to z, \tilde{L} z_\varepsilon \to \tilde{L} z$ in $H^m_{\tg}$ as $\varepsilon\to 0$. The structure of the equations then ensures that $\partial_1 \p_\varepsilon, \partial_1 (u_1)_\varepsilon$ also belong to $H^m_{\tg}$. In particular, the $H^1$-regularity of $u_1$ in the normal variable allows a further
refinement to a sequence $z_{\varepsilon'}$ that is smooth in both tangential and normal variables and satisfies all the requirements given in \eqref{strong}.
The remaining task and also the most difficult case to consider this problem in a small neighborhood of a corner point.

For  a fixed corner $\tilde{x}\in \{q_1,\ldots,q_N\}$, let $\Omega\cap V(\tilde{x}) = \{0<\Phi_1(x),\Phi_2(x)<\delta\}$ be the local coordinate patch in the small neighborhood of $\tilde{x}$ as introduced in Notation~\ref{nota-1}(3), with $\Phi_1(\tilde{x}) = \Phi_2(\tilde{x}) = 0$. Set $x_i^* = \Phi_i(x)$ for $i=1,2$, $x^* = (x_1^*, x_2^*)$, and denote by $\Phi = (\Phi_1,\Phi_2)$ and $\Phi^{-1} = (\Psi_1, \Psi_2)$. Define $z^* = (u^*,b^*,\p^*,s^*)^t$ by
\begin{equation*}
    \begin{aligned}
    &u_i^*(t,x^*) =\partial_1 \Phi_i(\Psi(x^*)) u_1(t,x)+\partial_2 \Phi_i(\Psi(x^*))u_2(t,x),\quad i=1,2\\
    &b_i^*(t,x^*) =\partial_1 \Phi_i(\Psi(x^*)) b_1(t,x)+\partial_2 \Phi_i(\Psi(x^*))b_2(t,x), \quad i =1, 2\\
    &\p^*(t,x^*) = \p(t,x),\quad  s^*(t,x^*) = s(t,x).
    \end{aligned}
\end{equation*}
Obviously, it holds that
\begin{equation*}
    \nabla \cdot u = \nabla^* \cdot u^* + e\cdot u^*,
\end{equation*}
where $\nabla^* = (\partial_1^*,\partial_2^*)^t$ with $\partial_j^* = \frac{\partial}{\partial x^*_j}$, 
$e=(e_1,e_2)^t$ with $e_j = \sum_{1\le i,k\le 2} (g^{-1})_{ki} \partial_k^* g_{ij}$, and $g$ denoting the $2\times 2$ matrix satisfying $(g^{-1})_{ij} = \frac{\partial x^*_i}{\partial{x_j}}$.

Define $Z^*(t,x^*)$ (resp. $F^*(t,x^*)$) in the same way as $z^*(t,x^*)$, but with $z$ replaced by $Z$ (resp. $F$).  It follows from \eqref{gamma-prob} that $z^* = (u^*,b^*,\p^*,s^*)^t$ satisfies the following system, for $t\in (0,T)$ and $x^*\in \mathcal{U}:= \{x^*\in \R^2, 0< x_1^*, x_2^* < \delta\}$,
\begin{equation}\label{qtr-eq}
    \begin{cases}
        L'_u(u^*,b^*) + \nabla^* \p^* = GF_1^*,\\
        L'_b(u^*,b^*,\p^*) = F_2^*,\\
        L'_p(u^*,b^*,\p^*) + \nabla^* \cdot u^* = F_3^*,\\
        L'_s(s^*) = F_4^*,
    \end{cases}
\end{equation}
where the notations are
\begin{equation*}
    \begin{cases}
        \begin{aligned}
            L'_u (u^*,b^*)  = & G R(P^*,S^*)\left(\partial_t u^* + U^* \cdot \nabla^* u^* + \gamma u^*\right) - G \left(B^* \cdot \nabla^* b^*\right) \\
            & - G E(B^*) b^* + G R(P^*,S^*) E(U^*) u^*,
        \end{aligned}\\
        \begin{aligned}
            L'_b (u^*,b^*,\p^*)  = &\left(1+\frac{1}{Q(P^*,S^*)} B^* \otimes (G B^*) \right) \left(\partial_t b^* + U^* \cdot\nabla^* b^* + \gamma b^*\right) \\
            & - \frac{1}{Q(P^*, S^*)} B^* \left(\partial_t \p^* + U^* \cdot \nabla^* \p^* + \gamma \p^*\right) - B^* \cdot \nabla^* u^*  \\
            & + \frac{1}{Q(P^*,S^*)}\left((GB^*)\cdot E(U^*) b^*\right) B^*,
        \end{aligned}\\
        \begin{aligned}
            L'_p (u^*,b^*,\p^*) = & \frac{1}{Q(P^*,S^*)} \Big(\partial_t\p^* + U^* \cdot \nabla^* \p^* + \gamma \p^* - (G B^*)\cdot (E(U^*) b^*)\\
            & - (G B^*) \cdot\left(\partial_t b^* + U^* \cdot\nabla^* b^* + \gamma b^*\right) \Big) + e\cdot u^*,
        \end{aligned}\\
        \begin{aligned}
            L'_s(s^*) = \partial_t s^* + U^* \cdot \nabla^* s^* + \gamma s^*,
        \end{aligned}
    \end{cases}
\end{equation*}
consisting of tangential or zero-order terms only, with
$G = g^t g$, 
$E(U^*) = g^{-1} \sum_{i=1,2}U^*_i \partial^*_i g$ and $E(B^*) = g^{-1} \sum_{i=1,2}B^*_i \partial^*_i g$. Note that $h(x, U)b-h(x, B)u$ given in \eqref{linear} disappears in the above system as we straighten the boundary. 
Denote the system \eqref{qtr-eq} as 
\begin{equation*}
    \tilde{L}' z^* = F',
\end{equation*}
with $F' = (GF_1^*, F_2^*, F_3^*, F_4^*)^t$, and below we shall drop the asterisks of notations for simplicity.

To verify that a weak solution in $H^{m,\gamma}_{tan}$ of \eqref{qtr-eq} in $\{x^*\in\R, 0<x_1^*, x_2^*<\delta\}$ is a strong one, we shall first perform a tangential mollification in the directions $\partial_t, x_1\partial_1, x_2\partial_2$, to obtain a tangentially smooth sequence $\{z_\varepsilon\}_{\varepsilon>0}$ such that $(u_i)_\varepsilon = 0$ on $x_i = 0$ for $i=1,2$, and 
$$z_\varepsilon \to z, \quad \tilde{L}' z_\varepsilon \to \tilde{L}' z \text{ in } H^m_{\tg} \text{ as }\varepsilon\to 0. $$ 
However, the equations only ensure that $\nabla \p_\varepsilon$, and $\nabla\cdot u_\varepsilon$ are tangentially smooth, and the $H^1$-regularity in the normal directions (including both $\partial_1$ and $\partial_2$) is not guaranteed for $u_\varepsilon$.
Thus, in the second step, we avoid normal translations for $u_\varepsilon$ and instead employ a suitable extension and mollification to construct a sequence $\{z_{\varepsilon'}\}$ satisfying \eqref{strong}.

    \underline{Step~1. Tangential approximation.} Let us first construct $\{z_\varepsilon=(u_\varepsilon,b_\varepsilon,\p_\varepsilon,s_\varepsilon)^t\}_{\varepsilon>0}$ such that: 
    \begin{equation}\label{semi-strong}
        \begin{aligned}
            &z_\varepsilon \in \cap_{s\in\mathbb{N}}H_\tg^s(\R\times \Omega), (\nabla \p_\varepsilon, \nabla\cdot u_\varepsilon) \in H^m_{\tg}(\R\times \Omega),\\& u_\varepsilon \cdot \nu = 0 \text{ on }\partial\Omega\backslash\{q_1,\cdots, q_N\}, \text{ and } z_\varepsilon\to z, \tilde{L}' z_\varepsilon \to \tilde{L}' z \text{ in } H^m_{\tg}(\R\times \Omega).
        \end{aligned}
    \end{equation}
    
    Inspired by \cite{rauch1985symmetric}, we introduce a set of smoothing operators as variants of Friedrich's mollifiers. Let $j\in C^\infty_0(\R)$ be even and non-negative, $\supp \, j\subset [-1,1]$ and $\int_\R j \d x = 1$. For any $\varepsilon > 0$, define 
    \begin{equation*}
        J^\tg_\varepsilon f := \int_{\R^3} f(t+\varepsilon \tau, x_1 e^{\varepsilon y_1}, x_2 e^{\varepsilon y_2}) j(\tau) j(y_1) j(y_2) \d \tau \d y,
    \end{equation*}
    and, for $i=1,2$,
    \begin{equation*}
        J^{\tg,i}_{\varepsilon,\pm} f := \int_{\R\times\Omega} e^{\pm \varepsilon y_i}f(t+\varepsilon \tau, x_1 e^{\varepsilon y_1}, x_2 e^{\varepsilon y_2}) j(\tau) j(y_1) j(y_2) \d \tau \d y,
    \end{equation*}
    which serves as modified versions of $J^{\tg}_\varepsilon$. It is easy to know that
    $$\partial_{x_i} J^{\tg}_\varepsilon = J^{\tg, i}_{\varepsilon,+} \partial_{x_i}, \quad \nabla\cdot \left(\diag(J^{\tg,1}_{\varepsilon,-}, J^{\tg,2}_{\varepsilon,-})(\cdot)\right) = J^{\tg}_\varepsilon (\nabla\cdot),$$ 
    moreover, the properties stated below can also be obtained as in \cite[Lemma, p.~174]{rauch1985symmetric}. 
    \begin{lemma}\label{lemma-tan-molif}
    Denote by $\mathcal{J}_\varepsilon$ the set $\{J^{\tg}_\varepsilon, J^{\tg,1}_{\varepsilon,\pm}, J^{\tg,2}_{\varepsilon,\pm}\}$, then for any $J_\varepsilon, J^1_\varepsilon, J^2_\varepsilon\in \mathcal{J}_\varepsilon$, and $\mathcal{T} = \partial_t, x_1\partial_1, x_2\partial_2$, the following hold for any given $z\in H^m_{\tg}(\R\times \Omega)$ with $m\in \mathbb{N}$, 
        \begin{enumerate}
            \item[(1)] $J_\varepsilon z \in \cap_{s\in \mathbb{N}} H^s_{\tg}(\R\times \Omega)$;
            \item[(2)] $J_\varepsilon z \to z$ in $H^{m}_{\tg}(\R\times \Omega)$, as $\varepsilon\to 0$;
            \item[(3)] $\mathcal{T} (J^1_\varepsilon z) - J^2_\varepsilon (\mathcal{T}z) \to 0$ in $H^m_{\tg}(\R\times \Omega)$ as $\varepsilon\to 0$.
        \end{enumerate}
    \end{lemma}

    Define
    \begin{equation*}
        (u_i)_\varepsilon = J^{\tg,i}_{\varepsilon,-}u_i\quad  (i=1,2),\quad
        (b_\varepsilon,\p_\varepsilon,s_\varepsilon)^t = J^\tg_\varepsilon (b,\p,s).
    \end{equation*}
    From Lemma \ref{lemma-tan-molif}, one has $z_\varepsilon \in \cap_{s\in \mathbb{N}} H^s_{\tg}(\R\times \Omega)$ and $z_\varepsilon \to z$ in $H^m_{\tg}(\R\times \Omega)$.
    
    By applying $J^{\tg,1}_{\varepsilon,+}$ and $J^{\tg,2}_{\varepsilon, +}$ to the first and second components of the momentum equation in \eqref{qtr-eq} respectively, it yields for each $i=1,2$, 
    \begin{equation*}
        \begin{aligned}
            \partial_{x_i} \p_\varepsilon = J^{\tg, i}_{\varepsilon,+}(\partial_{x_i}\p) =&  J^{\tg, i}_{\varepsilon,+} (F'_1)_i - (L'_u(u_\varepsilon,b_\varepsilon))_i\\
            & - \left(J^{\tg, i}_{\varepsilon,+} (L'_u(u,0))_i - \big(L'_u((J^{\tg, 1}_{\varepsilon,+}u_1, J^{\tg, 2}_{\varepsilon,+}u_2)^t, 0 )\big)_i
            \right)\\
            & - \left(J^{\tg, i}_{\varepsilon,+} (L'_u(0,b))_i - \big(L'_u(0, J^{\tg}_\varepsilon b)\big)_i
            \right).
        \end{aligned}
    \end{equation*}
    By Lemma~\ref{lemma-tan-molif}, it follows that $\partial_{x_i} \p_\varepsilon \in H^{m}_{\tg}(\R\times \Omega)$, and hence
    \begin{equation*}
    \begin{aligned}
        L'_u(u_\varepsilon,b_\varepsilon) + \nabla \p_\varepsilon \to F_1' \quad\text{in } H^m_{\tg}(\R\times \Omega) \text{ as }\varepsilon \to 0.
    \end{aligned}
    \end{equation*}
    Similarly, by acting $J^\tg_\varepsilon$ on the continuity equation it gives that
    \begin{equation*}
    \begin{aligned}
        \nabla\cdot u_\varepsilon = J^{\tg}_{\varepsilon} (\nabla\cdot u)  = & J^{\tg}_{\varepsilon} F'_3 - L'_p(u_\varepsilon, b_\varepsilon, \p_\varepsilon)\\
        & - \left(J^{\tg}_{\varepsilon} L'_p(u,0,0) - L'_p\big((J^{\tg, 1}_{\varepsilon,+}u_1, J^{\tg, 2}_{\varepsilon,+}u_2)^t,0,0\big)
        \right)\\
        & - \left(J^{\tg}_{\varepsilon} L'_p(0,b,\p) - L'_p\big(0,J^{\tg}_{\varepsilon} b, J^{\tg}_{\varepsilon} \p\big)
        \right),
    \end{aligned}
    \end{equation*}
    which implies $\nabla\cdot u_\varepsilon \in H^{m}_{\tg}(\R\times \Omega)$ and
    \begin{equation*}
    \begin{aligned}
        & L'_p(u_\varepsilon,b_\varepsilon,\p_\varepsilon) + \nabla\cdot u_\varepsilon \longrightarrow F'_3 \quad\text{in } H^m_{\tg}(\R\times \Omega) \text{ as }\varepsilon \to 0.
    \end{aligned}
    \end{equation*}
    Applying $J^\tg_\varepsilon$ in the same manner to the evolution equations of $b$ and $s$, and summarizing the above calculations, we conclude that 
    $$\tilde{L}' z_\varepsilon \longrightarrow F' = \tilde{L}' z \text{ in } H^{m}_{\tg}(\R\times \Omega), \text{ as }\varepsilon \to 0$$ 
    thus verify the assertion \eqref{semi-strong}.

    \underline{Step~2. Normal approximation.} The next task is to mollify $z_\epsilon$ obtained in \eqref{semi-strong} to get better regularity in the normal variables. For simplicity of notations, rewrite $z_\epsilon$ as $z$, and assume now that all of $z$, $\nabla\p$ and $ \nabla\cdot u$ belong to  $\cap_{s\in \mathbb{N}} H^s_{\tg}(\R\times \Omega)$.
    Let us smoothing $u$, and $(b,\p,s)^t$ by different approaches respectively. 

    Define two standard Friedrich-type mollifiers $J_\varepsilon$ and $\tilde{J}_\varepsilon$ as follows: for $f$ defined on $\R\times \Omega$ and $g$ defined on $\R^3$ respectively, let
    \begin{equation*}
        (J_\varepsilon f)(t,x) =\int_{\R^2} \frac{1}{\varepsilon^2} j(\frac{x_1- y_1 +2 \varepsilon}{\varepsilon}) j(\frac{x_2-y_2 + 2\varepsilon}{\varepsilon}) f(t,y)\d y,
    \end{equation*}
    and
    \begin{equation*}
        (\tilde{J}_\varepsilon g)(t,x) =\int_{\R^2} \frac{1}{\varepsilon^2} j(\frac{x_1- y_1}{\varepsilon}) j(\frac{x_2-y_2}{\varepsilon}) g(t,y)\d y,
    \end{equation*}
    with $j\in C_0^\infty(\R)$ being the function given below \eqref{semi-strong}.

    In analog to those given in Lemma~\ref{lemma-tan-molif}, these two operators have the following properties.
    
    \begin{lemma}\label{lemma-molif} 
    For given $z\in H^m_{\tg}(\R\times \Omega), \tilde{z} \in H^m_{\tg}(\R^3)$ with some $m\in \mathbb{N}$, one has
        \begin{enumerate}
            \item[(1)] $J_\varepsilon z \in \cap_{s\in \mathbb{N}} H^s(\R\times \Omega)$ and $\tilde{J}_\varepsilon \tilde{z} \in \cap_{s\in \mathbb{N}} H^s(\R^3)$;
            \item[(2)] $J_\varepsilon z \to z$ in $H^{m}_{\tg}(\R\times \Omega)$ and $\tilde{J}_\varepsilon \tilde{z} \to \tilde{z}$ in $H^{m}_{\tg}(\R^3)$, as $\varepsilon\to 0$;
            \item[(3)] For any operator $\mathcal{D}\in \{\partial_t, x_1\partial_1, x_2\partial_2,\partial_1,\partial_2\}$,
            $[\mathcal{D},J_\varepsilon]z \to 0$ in $H^m_{\tg}(\R\times \Omega)$, $[\mathcal{D},\tilde{J}_\varepsilon]\tilde{z} \to 0$ in $H^m_{\tg}(\R^3)$, as $\varepsilon\to 0$.
        \end{enumerate}
    \end{lemma}

    \underline{Smoothing $(b,\p,s)$:}
    Define $(b_\varepsilon,\p_\varepsilon,s_\varepsilon)^t := J_\varepsilon (b,\p,s)^t$.
    Obviously, one has $(b_{\varepsilon},\p_{\varepsilon},s_{\varepsilon})^t\in \cap_{s\in \mathbb{N}} H^s_{\tg}(\R\times \Omega)$ and $(b_{\varepsilon},\p_{\varepsilon},s_{\varepsilon})^t \to (b,\p,s)^t$ in $H^m_{\tg}(\R\times \Omega)$ as $\varepsilon\to 0$. Moreover,
    \begin{equation}\label{weak=strong-1}
        \begin{aligned}
            \tilde{L}'(0,b_{\varepsilon},\p_{\varepsilon},s_{\varepsilon})^t & = J_\varepsilon \tilde{L}'(0,b,\p,s)^t - [J_\varepsilon, \tilde{L}'](0,b,\p,s)^t\\
            & \longrightarrow \tilde{L}'(0,b,\p,s)^t \quad \text{in }H^m_{\tg}(\R\times \Omega), \text{ as }\varepsilon\to 0.
        \end{aligned}
    \end{equation}

    \underline{Smoothing $u$:} Denote by $E$ the extension operator on $u=(u_1, u_2)^t$ in the way of
    extending $u_1$ oddly across $x_1=0$ and evenly across $x_2=0$, and $u_2$ oddly across $x_2=0$ and evenly across $x_1=0$.  It is straightforward to have that
    \begin{equation*}
   \nabla\cdot (E u) (t,x_1,x_2) = \nabla \cdot u (t,|x_1|,|x_2|), \quad E: H^{m}_{\tg}(\R\times \Omega) \to H^{m}_{\tg}(\R\times \R^2),\quad \forall m\ge 0, 
    \end{equation*}
    so
    $\nabla\cdot (Eu) \in \cap_{s\in \mathbb{N}} H^s_{\tg}(\R^3)$.
    Let $u_\varepsilon = \tilde{J}_\varepsilon (E u)|_{\R\times\Omega}$. 
    Then $u_\varepsilon \in \cap_{s\in \mathbb{N}}H^s(\R\times \Omega)$ and $u_\varepsilon \to u$ in $H^m_{\tg}(\R\times\Omega)$ as $\varepsilon\to 0$. 
    Since $j$ is even, $u_\varepsilon$ inherits the same boundary conditions as $u$.
    Furthermore,
    \begin{equation}\label{weak=strong-2}
        \begin{aligned}
            \tilde{L}'(u_\varepsilon,0,0,0)^t  = & \left(\tilde{J}_\varepsilon \tilde{L}'(Eu,0,0,0)^t - [\tilde{J}_\varepsilon, \tilde{L}'](Eu,0,0,0)^t\right)_{|\R\times \Omega}\\
            &\longrightarrow \tilde{L}'(u,0,0,0)^t \quad\text{in } H^m_{\tg}(\R\times \Omega), \text{ as }\varepsilon\to 0.
        \end{aligned}
    \end{equation}
    Combining \eqref{weak=strong-1} and \eqref{weak=strong-2}, we conclude that $z_\varepsilon = (u_\varepsilon,b_\varepsilon,\p_\varepsilon,s_\varepsilon)^t$ fullfills all requirements of \eqref{strong}, which completes the proof of Proposition~\ref{prop-weak=strong}.

\begin{proof}[Proof of Proposition~\ref{prop-tan-well}]
    Let $\{z_\varepsilon\}_{\varepsilon>0}$ be the approximating sequence given in \eqref{strong}. By Proposition~\ref{prop-apriori-tan}, it satisfies
    \begin{equation}\label{zeps-tan-est}
        \sum_{|\alpha|\le m}\gamma^{2(m-|\alpha|)}\|\mathbb{T}^\alpha z_\varepsilon(t)\|^2_{L^2(\Omega)} + \gamma \|z_\varepsilon\|^2_{H^{m,\gamma}_{\tg}(\R\times \Omega)}\le C\frac{1}{\gamma}\|(\tilde{L}+\mathbb{B}+\gamma S_0 A_0)z_\varepsilon\|^2_{H^{m,\gamma}_{\tg}(\R\times\Omega)}.
    \end{equation}
    It follows that $\{z_\varepsilon\}_\varepsilon$ is a Cauchy sequence in $CH^m_{\tg}(\R\times \Omega)$. Consequently, the weak solution $z$ obtained in Proposition~\ref{prop-tan-weak} actually belongs to $CH^m_{\tg}\cap H^m_{\tg}(\R\times\Omega)$.
    Passing to the limit yields that $z$ satisfies estimate \eqref{z-tan-est}.
    The uniqueness follows directly from this estimate. Finally, the last assertion given in Proposition~\ref{prop-tan-well} can be proved in the same manner as in \cite[Section~2.3]{metivier2006stability}, and is therefore omitted.
\end{proof}

\subsubsection{Normal regularity}

Finally, we employ the divergence-curl structure of the system to show that,
for a given $F\in H^m_*(\R\times \Omega)$, the solution $z \in CH^m_{\tg}\cap H^m_{\tg}(\R\times\Omega)$ of \eqref{gamma-prob} obtained in Proposition~\ref{prop-tan-well} actually belongs to $X^m_*\cap H^m_*(\R\times\Omega)$.

For $j\in \R$, $s\in \mathbb{N}$,  define 
\begin{equation*}
    H^j_x(H^s_{\tg})(\R\times \Omega) : =\{z\in L^2(\R; H^j(\Omega)): \partial_*^\alpha z\in L^2(\R; H^j(\Omega)), \forall \alpha\in \mathbb{N}^5, |\alpha|_*\le s\},
\end{equation*}
and
\begin{equation*}
    CH^j_x(H^{s}_\tg)(\R\times \Omega): =\{z\in C(\R; H^j(\Omega)): \partial_*^\alpha z\in C(\R; H^j(\Omega)), \forall \alpha\in \mathbb{N}^5, |\alpha|_*\le s\},
\end{equation*}
both equipped with their canonical norms. It follows directly from the definitions that
\begin{equation*}
    H^m_*(\R\times \Omega) = \bigcap_{j=0}^{[\frac{m}{2}]} H^j_x (H^{m-2j}_{\tg})(\R\times \Omega),\quad X^m_*(\R\times \Omega) = \bigcap_{j=0}^{[\frac{m}{2}]} CH^j_x (H^{m-2j}_{\tg})(\R\times \Omega).
\end{equation*}

Now, we  proceed by induction on $0\le j\le [\frac{m}{2}]$ to prove that $z\in \big( CH^j_x(H^{m-2j}_\tg) \cap H^j_x(H^{m-2j}_\tg) \big) (\R\times \Omega)$.
The strategy is as follows:
first use the equations to obtain the regularity of $\nabla\p$ and $\nabla\cdot u$, then analyze the system for $\nabla\times u, \nabla \cdot b, \nabla\times b$ to get the full regularity of $\nabla u, \nabla b$, finally, study the transport equation of $s$ to deduce its normal regularity.

As computed in \eqref{eq-div-b}, \eqref{vort} of Section~\ref{section-linear}, from \eqref{gamma-prob} we deduce that $(\nabla\times u, \nabla \cdot b, \nabla\times b)$
satisfies the following system
\begin{equation}\label{normal-sys}
    \begin{cases}
    R(D_t+\gamma)(\nabla\times u) - B\cdot \nabla(\nabla\times b) = G_1,\\
    (D_t+\gamma)(\nabla\cdot b) = G_2,\\
    (1+\frac{1}{Q}|B|^2)(D_t+\gamma)(\nabla\times b) - B\cdot \nabla(\nabla\times u) = G_3,
    \end{cases}
\end{equation}
where
\begin{equation*}
    \begin{aligned}
    G_1 = & \nabla\times F_1 + R[D_t, \nabla\times u] - [B\cdot\nabla, \nabla\times ] b - (\nabla R \times D_t u),
        \\
    G_2  = &\nabla\times (F_2 + F_3 B) + \partial_i B_j \partial_j u_i - (\nabla\cdot B)(\nabla\cdot u) \\
    & - \partial_i U_j \partial_j b_i + \nabla\cdot \left(h(x,B)u- h(x,U)b\right),\\
        G_3  = & \nabla\times F_2 - \frac{1}{Q} B\times D_t F_1 -[\nabla\times, D_t]b + [\nabla\times, B\cdot \nabla] u \\ & + \nabla\times \left(h(x,B)u-h(x,U)b\right)
            + \frac{1}{Q}(\nabla\times B) (D_t \p - B\cdot D_t b)  - B\times [\nabla, \frac{1}{Q} D_t] \p \\
            & + B\times [\nabla, \frac{1}{Q} B\cdot] D_t b 
            + \sum_{j=1,2}\frac{1}{Q} B\times ([B_j \partial_1, D_t] b_j, [B_j \partial_2, D_t] b_j)^t \\ &- \frac{1}{Q} (B\cdot D_t B)(\nabla\times b) + \frac{1}{Q} B\times [D_t, R] D_t u + \frac{1}{Q} B\times (R D_t^2 u).
    \end{aligned}   
\end{equation*}
We aim to express all normal derivative terms, $\partial_x u, \partial_x b$, appearing in the expression of $G_1,G_2, G_3$, in terms of the quantities $\nabla\cdot u, \nabla\times u, \nabla\cdot b, \nabla\times b$.
To this end, first we have the following result on the Div-Curl system.

\begin{prop}[Div-Curl system]\label{prop-D-C}
Let $\Omega$ be as introduced in Section~1, and assume that for a fixed $m\in \mathbb{N}$, the angles of $\partial\Omega$ satisfy $\omega_n \in (0,\frac{\pi}{m+1})$ for all $1\le n\le N$.
Then for any $v = (v_1,v_2)^t\in H^m_x(H^s_\tg)(\R\times \Omega)$ with $s\in \mathbb{N}$, the boundary value problem
    \begin{equation}\label{D-C}
    \begin{cases}
        \nabla\cdot u = v_1 - a(v_1),\quad
        \nabla\times u = v_2, \quad \text{for }x\in \Omega,\\
        u \cdot \nu = 0, \quad \text{on } \partial\Omega\backslash\{q_1,\ldots,q_N\},
    \end{cases}
\end{equation}
with $a(v_1):= \frac{1}{|\Omega|}\int_\Omega v_1 \, \d x$ being the average operator, admits a unique solution $u\in H^{m+1}_x(H^{s}_{\tg})(\R\times \Omega)$. Moreover, there exists a constant $C>0$, independent of $v$ and $u$, such that
\begin{equation*}
    \|u\|_{H^{m+1}_x(H^s_{\tg})(\R\times \Omega)}\le C \|v\|_{H^m_x(H^s_{\tg})(\R\times \Omega)}. 
\end{equation*}
\end{prop}
\begin{proof}
    By Proposition~\ref{prop-PD-PN}, for $v\in H^m_x(H^s_{\tg})(\R\times \Omega)$, one can define $K_1 := \Delta_D^{-1} (v_2)$, $K_2 := \Delta_N^{-1} (v_1-a(v_1))$ so that $K_1, K_2\in H^{m+2}_x(H^s_\tg(\R\times \Omega))$. Define $u = \nabla^\perp K_1 + \nabla K_2$, then $u\in H^{m+1}_x(H^s_{\tg})(\R\times \Omega)$ clearly satisfies \eqref{D-C}. 
    Moreover, $\|u\|_{H^{m+1}_x(H^s_{\tg})}\le \|K_1, K_2\|_{H^{m+2}_x (H^s_{\tg})} \le C\|v_1,v_2\|_{H^{m}_{\tg}}$. The uniqueness follows from Lemma~\ref{lemma-Hodge}, as $\Omega$ is simply connected.
\end{proof}

\begin{nota}
\label{def-phi}
    Let $\Omega$ and $\omega_n$ be as in Proposition~\ref{prop-D-C}. Let
$$\phi \in \cap_{s\in \mathbb{N}} \mathcal{L}\left(H^m_x(H^s_{\tg})(\R\times \Omega)\right),$$ be the map from $v\in H^m_x(H^s_{\tg})(\R\times \Omega)$ to $(\partial_1 u_1, \partial_2 u_1, \partial_1 u_2, \partial_2 u_2)^t =: \tilde{\nabla} u$, where $u$ is the solution of \eqref{D-C} provided by Proposition~\ref{prop-D-C}.
    
\end{nota}

\begin{remark}\label{rmk-phi}
    If $v$ is fixed as $(v_1,v_2)=(\nabla\cdot u, \nabla\times u)$ for some $u\in H^1_x (\R\times \Omega)$ with $u\cdot \nu =0$ on $\partial\Omega\backslash\{q_1,\ldots,q_N\}$, then $\phi(v) = \tilde{ \nabla} u$. Thus, $\phi$ provides a reconstruction of the gradient of a vector field from its divergence and curl.
\end{remark}

To analyze the higher-order normal regularity of the solution from \eqref{normal-sys}, introduce the following Helmholtz decomposition.

\begin{prop}[Helmholtz decomposition]\label{prop-Helmholtz}
Let $\Omega$ be as in Section~1, and assume that for a fixed integer $m\ge 1$, $\omega_n \in (0,\frac{\pi}{m})$ for all $n= 1,\ldots, N$.
Then for any $v=(v_1,v_2)^t \in H^m_x(H^s_{\tg})(\R\times \Omega)$ with $s\in \mathbb{N}$, there exists a unique decomposition
\begin{equation*}
    v = \nabla f + g, 
\end{equation*}
where $f\in H^{m+1}_x(H^s_{\tg})(\R\times \Omega)$, $g\in H^{m}_x(H^s_{\tg})(\R\times \Omega)$, with $\nabla\cdot g = 0$ (in the distributional sense if $m=0$). Moreover, there exists a constant $C>0$, independent of $v,f$ and $g$, such that
\begin{equation*}
    \|\nabla f\|_{H^m_x(H^s_{\tg})(\R\times \Omega)} \le C\|v\|_{H^m_x(H^s_{\tg})(\R\times \Omega)}.
\end{equation*}
Additionally, if $\int_\Omega f \d x = 0$, then $f$ is unique.
\end{prop}
The proof of Proposition~\ref{prop-Helmholtz} will be given in Appendix~B.

Denote by $I-\mathcal{P}$ the mapping $v \mapsto \nabla f$, where $\mathcal{P}$ is the Leray projection, and $f$ is given in the above proposition. Then, $I-\mathcal{P} \in \cap_{s\in \mathbb{N}} \mathcal{L}(H^m_x(H^s_{\tg})(\R\times \Omega))$.

\begin{defn}\label{def-psi}
Let $\Omega$ and $\omega_n$ be as in Proposition~\ref{prop-Helmholtz}.
We define a mapping 
    $$\psi \in \cap_{s \in \mathbb{N}} \mathcal{L}\left(H^{m}_x(H^s_{\tg})(\R\times \Omega) \to H^{m+1}_x(H^s_{\tg})(\R\times \Omega)\right),$$
    which maps any $v \in H^{m}_x(H^s_{\tg})(\R\times \Omega)$ to the function $f$ given in Proposition~\ref{prop-Helmholtz} with $\int_\Omega f \d x = 0$.
\end{defn}
\begin{remark}\label{remark-psi}
    If $v = \nabla u$ for some $u\in H^1_x(\R\times \Omega)$, then
    $\psi(v) = u - a(u)$.
\end{remark}

Now we start to prove the  Proposition~\ref{prop-regularity}.
\begin{proof}[Proof of Proposition~\ref{prop-regularity}]
We proceed by induction on $j=1, \ldots , m$ to prove that $z\in \big( CH^j_x(H^{m-2j}_\tg) \cap H^j_x(H^{m-2j}_\tg) \big) (\R\times \Omega)$.
    
    \underline{Step~1. The case $j=1$.}
    Since $z \in (CH^{m}_\tg\cap H^{m}_\tg)(\R\times \Omega)$ solves \eqref{gamma-prob}, it follows that
    \begin{equation*}
        \nabla\p, \nabla\cdot u \in H^m_*-(CH^{m-1}_\tg\cap H^{m-1}_\tg) \subset (CH^{m-2}_\tg\cap H^{m-2}_\tg)(\R\times \Omega).
    \end{equation*}
    The system \eqref{normal-sys} can then be written in the form
    \begin{equation}\label{P-normal}
        \mathbb{P}_\gamma (\nabla\times u, \nabla\cdot b, \nabla\times b)^t = \mathbb{Q}^1_1 (D_*^2 F) + \mathbb{Q}^1_2 (D_\mathcal{T}^2 z, \nabla \p) + \mathbb{Q}^1_3 (\tilde{\nabla} u, \tilde{\nabla} b),
    \end{equation}
where
\begin{equation*}
    \mathbb{P}_\gamma = \begin{pmatrix}
        R & & \\
          & 1 & \\
          & & 1 + \frac{1}{Q}|B|^2
    \end{pmatrix} (\partial_t+\gamma) + \sum_{i=1,2}
    \begin{pmatrix}
        RU_i & & - B_i \\
         & U_i & \\
        -B_i & & (1+\frac{1}{Q}|B|^2) U_i
    \end{pmatrix}\partial_i
\end{equation*}
is a symmetric transport operator, whose normal matrix vanishes on the boundary as $U\cdot \nu = B\cdot \nu = 0$ on $\partial\Omega \backslash\{q_1,\ldots, q_N\}$.
The terms on the right-hand side of \eqref{P-normal} are defined as follows:
\begin{enumerate}
\item[-] $\mathbb{Q}_1^1 (D_*^2 F)$ is a sum of terms of the form $\Theta(x,Z) \partial_*^{\alpha_n} F_n$ with $|\alpha_n|_*\le 2$ and $\Theta\in C^\infty(\overline{\Omega}\times\mathcal{K})$;

\item[-] $\mathbb{Q}_2^1 (D_\mathcal{T}^2 z, \nabla \p)$ is a sum of terms of the form $\Theta(x,Z) \partial_{\mathcal{T}}^{\alpha_n} z_n$ with $|\alpha_n|_*\le 2$, or $\Theta(x,Z) \partial_{x_i} p$, $i=1,2$, with $\Theta\in C^\infty(\overline{\Omega}\times\mathcal{K})$; 

\item[-] $\mathbb{Q}_3^1 (\tilde{\nabla} u, \tilde{\nabla} b)$ is a sum of terms of the form $\Theta(x,Z) \partial_{x_i} u_j$ or $\Theta(x,Z) \partial_{x_i} b_j$ with $i,j=1,2$ and $\Theta\in C^\infty(\overline{\Omega}\times\mathcal{K})$.

\end{enumerate}

Since $\nabla\p, \nabla \cdot u$ are already known to belong to $(CH^{m-2}_\tg \cap H^{m-2}_\tg)(\R\times\Omega)$, 
we recover $\tilde{\nabla} u$ by $$ \phi(\nabla\cdot u, 0) + \phi(0, \nabla\times u) $$ where $\phi$ is defined in Notation~\ref{def-phi}. Similarly, recover $\nabla b$ by $\phi(\nabla\cdot b, \nabla \times b)$ where both $\nabla \cdot b$ and $\nabla \times b$ are unknowns determined in \eqref{normal-sys}. 
Accordingly, define two bounded linear mappings $\Phi_1^I, \Phi_1^{II} \in \mathcal{L}(H^{m-2}_{\tg}(\R\times \Omega))$ by
\begin{equation}\label{Phi-1}
    \Phi^I_1(v_1,v_2,v_3) = (\phi(0,v_1), \phi(v_2,v_3)), \quad \Phi^{II}_1(v_4) = (\phi(v_4, 0) , 0 ),
\end{equation}
so that for any $u,b \in H^1(\Omega)$,
\begin{equation*}
    (\tilde{\nabla} u, \tilde{\nabla} b) = \Phi^I_1(\nabla\times u, \nabla\cdot b, \nabla\times b) + \Phi^{II}_1(\nabla\cdot u).
\end{equation*}

Hence we are led to the following transport equation for $v = (v_1,v_2,v_3)^t$:
\begin{equation}\label{v-eq-1}
    (\mathbb{P}_\gamma + \Phi^1) v = \mathbb{Q}^1_1(D_*^2 F) + \widetilde{\mathbb{Q}}^1_2 (D_{\mathcal{T}}^2 z, \nabla \p, \nabla\cdot u),
\end{equation}
where 
\begin{equation*}
    \widetilde{\mathbb{Q}}^1_2 (D_{\mathcal{T}}^2 z, \nabla \p, \nabla\cdot u) = \mathbb{Q}^1_2 (D_{\mathcal{T}}^2 z, \nabla \p) + \mathbb{Q}^1_3 \Phi^{II}_1 (\nabla\cdot u),\quad 
    \Phi^1(v) := -\mathbb{Q}^1_3 (\Phi_1^I(v)).
\end{equation*}
Noting that \eqref{v-eq-1} is a linear transport system that requires no any boundary condition, and the right-hand side belongs to $H^{m-2}_\tg (\R\times \Omega)$. Therefore, by adapting the arguments of Proposition~\ref{prop-tan-well}, the system is well posed in $(CH^{m-2}_\tg\cap H^{m-2}_\tg)(\R\times\Omega)$.

It remains to verify that the solution $v\in (CH^{m-2}_\tg\cap H^{m-2}_\tg)(\R\times\Omega)$ of \eqref{v-eq-1} indeed coincides with $(\nabla\times u, \nabla\cdot b, \nabla\times b)^t$. To this end, we establish the identity in the distributional sense on $\R\times \Omega$.

Let $(\Phi^1)^*\in \mathcal{L}(H^{m-2}_\tg(\R\times\Omega))$ be the adjoint of $\Phi^1$, and set $\mathscr{D} = C^\infty_0(\R\times \Omega)$. For any $\varphi= (\varphi_1,\varphi_2,\varphi_3)^t \in \mathscr{D}$, we have
\begin{equation}\label{test}
\begin{aligned}
    &\quad \inp*{v}{(\mathbb{P}_\gamma^*+(\Phi^1)^*)\varphi}_{L^2} = \inp*{(\mathbb{P}_\gamma+\Phi) v}{\varphi}_{L^2} = \inp*{\mathbb{Q}_1(D_*^2 F) + \widetilde{\mathbb{Q}}_2 (D_{\mathcal{T}}^2 z, \nabla \p, \nabla\cdot u)}{\varphi}_{\mathscr{D}'\times \mathscr{D}}\\
    & = \inp*{(\mathbb{P}_\gamma+\Phi^1)(\nabla\times u, \nabla\cdot b, \nabla\times b)^t}{\varphi}_{\mathscr{D}'\times \mathscr{D}}\\
    &= \inp*{u^\perp}{\nabla \big((\mathbb{P}_\gamma^*+(\Phi^1)^*) \varphi\big)_1}_{L^2} - \inp*{b}{\nabla \big((\mathbb{P}_\gamma^*+(\Phi^1)^*) \varphi\big)_2}_{L^2} + \inp*{b^\perp}{\nabla \big((\mathbb{P}_\gamma^*+(\Phi^1)^*) \varphi\big)_3}_{L^2},
\end{aligned}
\end{equation}
where $\inp*{\cdot}{\cdot}_{\mathscr{D}'\times \mathscr{D}}$ is the duality pairing of a distribution and a test function.
Let $\mathscr{D}_1 = (\mathbb{P}_\gamma^*+(\Phi^1)^*)\mathscr{D}\subset\mathscr{D}$. We claim that $\mathscr{D}_1$ is dense in $\mathscr{D}$ with respect to the $H^1(\R\times\Omega)$-norm. If it is true, then for any $\varphi \in \mathscr{D}$, there exists an approximating sequence $\{\varphi_\varepsilon\}_{\varepsilon >0}\subset \mathscr{D}_1$, such that $\varphi_\varepsilon \to \varphi$ in $H^1(\R\times\Omega)$ as $\varepsilon \to 0$, and we have the identity as given in \eqref{test} with $(\mathbb{P}_\gamma^*+(\Phi^1)^*) \varphi$ being replaced by $\varphi_\varepsilon$, from which by passing the limit $\varepsilon \to 0$, one obtains 
\begin{equation*}
    \inp*{v}{\varphi}_{L^2} = \inp*{u^\perp}{\nabla \varphi_1}_{L^2} - \inp*{b}{\nabla \varphi_2}_{L^2} + \inp*{b^\perp}{\nabla \varphi_3}_{L^2}, \quad \forall \varphi= (\varphi_1,\varphi_2,\varphi_3)^t \in \mathscr{D}.
\end{equation*}
So, we conclude that $v = (\nabla\times u, \nabla\cdot b, \nabla\times b)^t$.

The above density statement, $\mathscr{D}_1$ is dense in $\mathscr{D}$, follows from the solvability of the operator $\mathbb{P}_\gamma^* + (\Phi^1)^*$. 
Indeed, since $\mathbb{P}_\gamma^* + (\Phi^1)^*$ is a tangential transport operator that contains a damping term $\gamma \cdot \diag\left(R, 1, 1+ \frac{1}{Q}|B|^2\right)$, we deduce that for any $\varphi \in \mathscr{D}\subset H^2_0(\R\times\Omega)$, the solution $\bar{\varphi} = (\mathbb{P}_\gamma^*+(\Phi^1)^*)^{-1}\varphi$ belongs to $H^2_0(\R\times\Omega)$. 
Hence, there exists a sequence $\{\bar{\varphi}_\varepsilon\}_{\varepsilon>0}\subset\mathscr{D}$, such that $\bar{\varphi}_\varepsilon\to \bar{\varphi}$ in $H^2(\R\times \Omega)$ as $\varepsilon\to 0$. Consequently, $\mathscr{D}_1 \ni\varphi_\varepsilon : = (\mathbb{P}_\gamma^* + (\Phi^1)^*)\bar{\varphi}_\varepsilon$ satisfies $\varphi_\varepsilon \to \varphi$ in $H^1(\R\times \Omega)$. Therefore, the density claim is true.

As a result, we have established that $\nabla\times u, \nabla\cdot b, \nabla\times b \in (CH^{m-2}_\tg\cap H^{m-2}_\tg)(\R\times\Omega)$. Moreover, $\nabla\cdot b \in (CH^{m-2}_\tg\cap H^{m-2}_\tg)(\R\times\Omega)$ implies that the trace $(b\cdot \nu)|_{\partial\Omega\backslash\{q_1,\ldots,q_N\}}$ is well-defined in $H^{-{\frac{1}{2}}}_x H^{m-2}_\tg(\R\times \partial\Omega)$. Formally, $b\cdot \nu$ satisfies the transport equation \eqref{bcd-b-eq}. Then, by a similar argument as above, one obtains that $(b\cdot \nu)|_{\partial\Omega\backslash\{q_1,\ldots,q_N\}} = 0$. This leads to $\tilde{\nabla} u = \phi(\nabla\cdot u, \nabla\times u), \tilde{\nabla} b = \phi(\nabla\cdot b, \nabla\times b)\in (CH^{m-2}_\tg\cap H^{m-2}_\tg)(\R\times \Omega)$. As the entropy $s$ satisfies a purely tangential transport equation, a similar argument yields that it also has the same regularity. Thus, we get
\begin{equation*}
    z\in \Big(CH^1_x\big(H^{m-2}_\tg\big)\cap H^1_x\big(H^{m-2}_\tg\big)\Big)(\R\times \Omega).
\end{equation*}
Therefore, the case $j=1$ is complete.

\underline{Step~2. The case $j\ge 2$.}
Assume that the solution of \eqref{gamma-prob},
\begin{equation}\label{normal-11}
    z\in \big(CH^l_x(H^{m-2l}_\tg)\cap H^l_x(H^{m-2l}_\tg)\big)(\R\times \Omega)
\end{equation}
holds for all $l=0,\ldots, j-1$. From the equations given in \eqref{gamma-prob}, one immediately deduces that
\begin{equation}\label{normal-22}
    \nabla \p, \nabla\cdot u \in \cap_{l=0}^{j-1} \big(CH^l_x(H^{m-2l-2}_\tg)\cap H^l_x(H^{m-2l-2}_\tg)\big)(\R\times \Omega).
\end{equation}

By acting $\nabla^{j-1}$ on the system \eqref{P-normal}, one obtains
\begin{equation}\label{v-eq-j}
\begin{aligned}
    &(I_{2^{j-1}}\otimes \mathbb{P}_\gamma) \left(\tilde{\nabla}^{j-1} (\nabla\times u), \tilde{\nabla}^{j-1}(\nabla\cdot b), \tilde{\nabla}^{j-1}(\nabla\times b)\right)^t \\
    &\quad = \mathbb{Q}^j_1(D^{2j}_* F) + \mathbb{Q}^j_2(D^2_{\mathcal{T}}D^{2(j-1)}_* z, D^{j-1}_x\nabla \p) + \mathbb{Q}^j_3(\tilde{\nabla}^j u, \tilde{\nabla}^j b),
\end{aligned}
\end{equation}
where $I_{2^{j-1}}$ is the $2^{j-1}$-dimensional identity matrix, $\otimes$ denotes the Kronecker product, and $\mathbb{Q}_1^j, \mathbb{Q}_2^j,\mathbb{Q}^j_3$ are smooth coefficients and $D_{\mathcal{T}}^q, D^q_{*}, D^q_x$ are derivatives in the form of $\mathcal{T}^\alpha$ ($\alpha\in \mathbb{N}^3$, $|\alpha|\le q$),
, $\partial_{*}^\beta$ ($\beta\in \mathbb{N}^5, |\beta|_*\le q$), $\partial_x^\sigma$ ($\sigma\in \mathbb{N}^2, |\sigma|\le q$), and
\begin{equation*}
    \tilde{\nabla}^j = \left(\partial_1^j, \partial_2\partial_1^{j-1},\partial_2^2\partial_1^{j-2},\ldots, \partial_2^{j-1}\partial_1, \partial_2^j\right)^t
\end{equation*}
is the $2^{j}$-dimensional vectorized gradient operator, $\tilde{\nabla}^j u = ((\tilde{\nabla}^j u_1)^t, (\tilde{\nabla}^j u_2)^t)^t$. 

Since the operator $(I_{2^{j-1}}\otimes \mathbb{P}_\gamma)$ inherits the essential properties of $\mathbb{P}_\gamma$, the key step -- as in the case $j=1$ -- is to construct a set of bounded linear mappings 
\begin{equation}\label{Phi-j-cdt}
\begin{aligned}
    & \Phi^I_j \in \cap_{s\in \mathbb{N}}\mathcal{L}\left(H^{s}_{\tg}(\R\times \Omega)\to H^{j-1}_x (H^{s}_{\tg})(\R\times \Omega)\right),\\
    & \Phi^{II}_{j,0} \in \cap_{l=1}^{[\frac{m}{2}]}\mathcal{L}\left(H^{l-1}_x (H^{m-2l}_{\tg})(\R\times \Omega)\right),\\
    & \Phi^{II}_{j,1} \in \cap_{l=j}^{[\frac{m}{2}]}\mathcal{L}\left(H^{l-2}_x(H^{m-2(l-1)}_{\tg})(\R\times \Omega)\to H^{l-1}_x (H^{m-2l}_{\tg})(\R\times \Omega)\right),
\end{aligned}
\end{equation}
such that for any $u, b \in C_c^\infty(\R\times \overline{\Omega})$,
\begin{equation}\label{Phi-j-cdt'}
    \begin{aligned}
        (\tilde{\nabla} u, \tilde{\nabla} b)  = & \Phi^I_j\left(\tilde{\nabla}^{j-1}(\nabla\times u), \tilde{\nabla}^{j-1}(\nabla\cdot b), \tilde{\nabla}^{j-1}(\nabla\times b)\right) \\
        & +\Phi^{II}_{j,0} (\nabla\cdot u) + \Phi^{II}_{j,1} \left(\nabla\times u, \nabla \cdot b, \nabla\times b\right).
    \end{aligned}
\end{equation}

In particular, for the case $j=1$, such mappings have already been constructed by setting $$\Phi^{II}_{1,0} = \Phi^{II}_1,\quad \Phi^{II}_{1,1} = 0,$$
where $\Phi^I_1, \Phi^{II}_1$ are defined in \eqref{Phi-1}. 
Assume now that the mappings $\Phi^I_{j-1}, \Phi^{II}_{j-1,0}, \Phi^{II}_{j-1,1}$ have been constructed and satisfy conditions \eqref{Phi-j-cdt}-\eqref{Phi-j-cdt'} with $j$ replaced by $j-1$. We shall then proceed to construct the next set of mappings $\Phi^I_{j}, \Phi^{II}_{j,0}, \Phi^{II}_{j,1}$, ensuring that they satisfy \eqref{Phi-j-cdt}-\eqref{Phi-j-cdt'} for the index $j$.

Denote by $(\tilde{\nabla}^{j-1})_i$ the $i$-th component of the vectorized gradient $\tilde{\nabla}^{j-1}$, with $1\le i\le 2^{j-1}$. Then, one observes that for $i=1,\ldots, 2^{j-2}$,
\begin{equation*}
    \left((\tilde{\nabla}^{j-1})_{2(i-1)+1}, (\tilde{\nabla}^{j-1})_{2(i-1)+2}\right)  = \left(\partial_1 (\tilde{\nabla}^{j-2})_{i}, \partial_2 (\tilde{\nabla}^{j-2})_{i} \right).
\end{equation*}
Thus, by Remark~\ref{remark-psi}, for $i=1,\ldots, 2^{j-2}$, $X = \nabla\times u, \nabla\cdot b, \nabla\times b$, one has
\begin{equation*}
    (\tilde{\nabla}^{j-2})_i X = \psi\left((\tilde{\nabla}^{j-1})_{2(i-1)+1}X, (\tilde{\nabla}^{j-1})_{2(i-1)+2}X\right) + a\left((\tilde{\nabla}^{j-2})_{i}X\right),
\end{equation*}
Below we denote by $\left\{f_i\right\}_{1\le i\le n}$ the $n$-dimensional vector, whose $i$-th entry is $f_i$.
Then for any $u, b \in C_c^\infty(\R\times\overline{\Omega})$, 
\begin{equation*}
    \begin{aligned}
        (\tilde{\nabla} u, \tilde{\nabla} b) & = \Phi^I_{j-1}\left(\tilde{\nabla}^{j-2}(\nabla\times u), \tilde{\nabla}^{j-2}(\nabla\cdot b), \tilde{\nabla}^{j-2}(\nabla\times b)\right) \\
        & +\Phi^{II}_{j-1,0} (\nabla\cdot u) + \Phi^{II}_{j-1,1} \left(\nabla\times u, \nabla \cdot b, \nabla\times b\right)\\
        & \begin{aligned}
            =\Phi^I_{j-1}\Bigg( &
        \left\{
        \psi\left((\tilde{\nabla}^{j-1})_{2(i-1)+1}(\nabla\times u), (\tilde{\nabla}^{j-1})_{2(i-1)+2}(\nabla\times u)\right)
        \right\}_{1\le i\le 2^{j-2}},\\
        & \left\{
        \psi\left((\tilde{\nabla}^{j-1})_{2(i-1)+1}(\nabla\cdot b), (\tilde{\nabla}^{j-1})_{2(i-1)+2}(\nabla\cdot b)\right)
        \right\}_{1\le i\le 2^{j-2}}\\
        & \left\{
        \psi\left((\tilde{\nabla}^{j-1})_{2(i-1)+1}(\nabla\times b), (\tilde{\nabla}^{j-1})_{2(i-1)+2}(\nabla\times b)\right)
        \right\}_{1\le i\le 2^{j-2}}\Bigg)
        \end{aligned} \\
        & \begin{aligned}
            + 
        \Phi^I_{j-1}\bigg(&
        \left\{a\left((\tilde{\nabla}^{j-2})_{i}(\nabla\times u)\right)
        \right\}_{1\le i\le 2^{j-2}}
        , \left\{a\left((\tilde{\nabla}^{j-2})_{i}(\nabla\cdot b)\right)
        \right\}_{1\le i\le 2^{j-2}}
        , \\ &\left\{a\left((\tilde{\nabla}^{j-2})_{i}(\nabla\times b)\right)
        \right\}_{1\le i\le 2^{j-2}}\bigg)
        \end{aligned}\\
        & +\Phi^{II}_{j-1,0} (\nabla\cdot u) + \Phi^{II}_{j-1,1} \left(\nabla\times u, \nabla \cdot b, \nabla\times b\right)\\
        &=: \Phi^I_j\left(\tilde{\nabla}^{j-1}(\nabla\times u), \tilde{\nabla}^{j-1}(\nabla\cdot b), \tilde{\nabla}^{j-1}(\nabla\times b)\right). 
    \end{aligned}
\end{equation*}

Now we define $\Phi^I_j$ by
\begin{equation*}
    \begin{aligned}
        &\quad \Phi^I_j \left(\tilde{\nabla}^{j-1}(\nabla\times u), \tilde{\nabla}^{j-1}(\nabla\cdot b), \tilde{\nabla}^{j-1}(\nabla\times b)\right) \\
    & \begin{aligned}
            =\Phi^I_{j-1}\Bigg( &
        \left\{
        \psi\left((\tilde{\nabla}^{j-1})_{2(i-1)+1}(\nabla\times u), (\tilde{\nabla}^{j-1})_{2(i-1)+2}(\nabla\times u)\right)
        \right\}_{1\le i\le 2^{j-2}},\\
        & \left\{
        \psi\left((\tilde{\nabla}^{j-1})_{2(i-1)+1}(\nabla\cdot b), (\tilde{\nabla}^{j-1})_{2(i-1)+2}(\nabla\cdot b)\right)
        \right\}_{1\le i\le 2^{j-2}}\\
        & \left\{
        \psi\left((\tilde{\nabla}^{j-1})_{2(i-1)+1}(\nabla\times b), (\tilde{\nabla}^{j-1})_{2(i-1)+2}(\nabla\times b)\right)
        \right\}_{1\le i\le 2^{j-2}}\Bigg),
        \end{aligned}
    \end{aligned}
\end{equation*}
$\Phi^{II}_{j,0}$ by $\Phi^{II}_{j-1,0}$, and $\Phi^{II}_{j,1}$ by
\begin{equation*}
    \begin{aligned}
        & \quad \Phi^{II}_{j,1} \left(\nabla\times u, \nabla\cdot b, \nabla\times b\right)  = \Phi^{II}_{j-1,1} \left(\nabla\times u, \nabla\cdot b, \nabla\times b\right)
        \\
     & \begin{aligned}
            + 
        \Phi^I_{j-1}\bigg(&
        \left\{a\left((\tilde{\nabla}^{j-2})_{i}(\nabla\times u)\right)
        \right\}_{1\le i\le 2^{j-2}}
        , \left\{a\left((\tilde{\nabla}^{j-2})_{i}(\nabla\cdot b)\right)
        \right\}_{1\le i\le 2^{j-2}}
        , \\ &\left\{a\left((\tilde{\nabla}^{j-2})_{i}(\nabla\times b)\right)
        \right\}_{1\le i\le 2^{j-2}}\bigg).
        \end{aligned}
    \end{aligned}
\end{equation*}
Since $\omega_n \in (0,\frac{\pi}{[\frac{m}{2}]})$, the constructed mappings $\Phi^{I}_j, \Phi^{II}_{j,0}, \Phi^{II}_{j,1}$ satisfy the desired conditions \eqref{Phi-j-cdt}-\eqref{Phi-j-cdt'}.

Now, for any $v^j = ((v^j_1)^t, (v^j_2)^t, (v^j_3)^t)^t \in H^{m-2j}_{\tg}(\R\times \Omega)$ with each $v^j_i$ for $i=1,2,3$ being a $2^{j-1}$-dimensional vector, we define the mapping
\begin{equation*}
    \Phi^j (v^j) := - \mathbb{Q}^j_3 \left(\tilde{\nabla}^{j-1}\Phi^I_j\left(\tilde{\nabla}^{j-1}(v_1^j), \tilde{\nabla}^{j-1}(v_2^j), \tilde{\nabla}^{j-1}(v_3^j)\right)\right).
\end{equation*}
Then $\Phi^j \in \mathcal{L}\left(H^{m-2j}_{\tg}(\R\times \Omega)\right)$.

From \eqref{v-eq-j}, we are led to consider the purely tangential transport equation
\begin{equation}\label{v-eq-j'}
    \begin{aligned}
    (I_{2^{j-1}}\otimes \mathbb{P}_\gamma + \Phi^j) v^j
    & = \mathbb{Q}^j_1(D^{2j}_* F) +
        \mathbb{Q}^j_2(D^2_{\mathcal{T}}D^{2(j-1)}_* z, D^{j-1}_x\nabla \p) \\
    & + \mathbb{Q}^j_3\left(\tilde{\nabla}^{j-1}\left(
    \Phi^{II}_{j,0} (\nabla\cdot u) + \Phi^{II}_{j,1} \left(\nabla\times u, \nabla \cdot b, \nabla\times b\right)
    \right)\right).
    \end{aligned}
\end{equation}
By \eqref{normal-11}-\eqref{normal-22} and \eqref{Phi-j-cdt}, the right-hand side of \eqref{v-eq-j'} belongs to $H^{m-2j}_{\tg}(\R\times \Omega)$.
Following a similar procedure as in Step~1, the well-posedness of \eqref{v-eq-j'} in $(CH^{m-2j}_\tg\cap H^{m-2j}_\tg)(\R\times \Omega)$ implies that $\tilde{\nabla}^j (u,b)\in (CH^{m-2j}_\tg\cap H^{m-2j}_\tg)(\R\times \Omega)$. 

Since $s$ also satisfies a tangential transport equation, an analogous but simpler argument yields $$z\in \big(CH^j_x(H^{m-2j}_{\tg})\cap H^j_x(H^{m-2j}_{\tg})\big)(\R\times \Omega)\quad\text{for all }j=0,\ldots, [\frac{m}{2}].$$ 
Consequently, $z\in X^{m}_*(\R\times\Omega)\cap H^{m}_*(\R\times \Omega)$.

Having established the full regularity of the solution $z$ to \eqref{gamma-prob}, we now return to the original boundary value problem \eqref{linear},~\eqref{bcd-linear}, namely, $(L+\mathbb{B},{\cal B})$. 
For any $F\in e^{\gamma t} H^m_{*}(\R\times \Omega)$, let $$\tilde{F}:= e^{-\gamma t} S_0(Z) F \in H^m_{*}(\R\times \Omega).$$ 
Then there exists a unique solution $\tilde{z} \in X^{m}_*(\R\times\Omega)\cap H^{m}_*(\R\times \Omega)$ to \eqref{gamma-prob} with $F$ replaced by $\tilde{F}$. Setting $z:= e^{\gamma t}\tilde{z}$, we obtain the solution of 
$$(L+\mathbb{B})(Z) z = F \text{ in }\R\times \Omega, \quad \mathcal{B} z = 0 \text{ on }\R\times (\partial\Omega\backslash\{q_1,\ldots,q_N\}).$$ 

For $F\in H^m_{*}((-\infty,T)\times \Omega)$ and $F = 0$ for $t<0$, one can extend $F$ to a function $\tilde{F}'\in H^m_*(\R\times \Omega)$ that vanishes for $t\ge T+1$. Since $\tilde{F}' \in e^{\gamma t}H^m_*(\R\times \Omega)$ for any $\gamma >0$, then by the above analysis and a restriction to $(-\infty, T]$, the existence part of Proposition~\ref{prop-regularity} follows immediately.

The uniqueness of the solution, together with the property that $z$ vanishes when $t<0$, follows from Corollary~2.3.4 and Remark~2.3.5 of \cite{metivier2006stability}.
\end{proof}

\begin{remark}\label{rmk-regularity}
From the proof of Proposition~\ref{prop-regularity}, we note that the condition $$\omega_n \in \left(0,\frac{\pi}{[\frac{m}{2}]}\right),\quad n=1,\ldots, N$$ 
can be weakened to be
\begin{equation*}
    \omega_n \in \left(0,\frac{\pi}{[\frac{\tilde{m}}{2}]}\right),\quad n=1,\ldots, N,
\end{equation*}
for some $\tilde{m}\in \mathbb{N}\cap [2,m]$.
In this case, the regularity $\big(X^{m}_*\cap H^{m}_*\big)((-\infty,T]\times\Omega)$ of the solution $z$ obtained in Proposition~\ref{prop-regularity} is also weakened as
\begin{equation*}
    \bigcap_{l=0}^{[\frac{\tilde{m}}{2}]}\left( CH^l_x(H^{m-2l}_{\tg})\cap H^l_x(H^{m-2l}_{\tg})\right)((-\infty,T]\times \Omega).
\end{equation*}
\end{remark}

\subsection{Compatibility lifting}
When treating the IBVP \eqref{linear}-\eqref{icd-linear} with inhomogeneous initial data, we shall first construct an auxiliary function satisfying the prescribed initial condition and the compatibility condition at $\{t=0, x\in\partial\Omega\}$, thereby transforming the IBVP into a pure boundary value problem, which we have studied in the previous subsection. 
If the data $z_0, F$ are taken in corresponding $H^m$ spaces and satisfy compatibility conditions only up to order $m-1$, then the constructed auxiliary function $\zeta$ will only belong to $CH^m([0, T]\times \Omega)$ and it will eventually appear on the right-hand side of the pure boundary value problem in the term $-(L+\mathbb{B})(Z)\zeta$, which only belongs to $H^{m-1}(\R\times \Omega)$ after a further extension on the time direction. 
Thus, the solution of the boundary value problem will only belong to $CH^{m-1}_*\cap H^{m-1}_*(\R\times\Omega)$ in general.
Therefore, to preserve the desired regularity of the resulting solutions, we first approximate $(z_0, F)$ by a sequence $\{(z_0^k, F^k)\}_k$ with higher-order regularity and meanwhile satisfying more compatibility conditions. The latter is called the compatibility lifting.

As noted in Remark~\ref{rmk-sp}, functions in $H^m_*(\Omega)$ exhibit different behaviors in the interior, near smooth portions of the boundary, and in neighborhoods of corner points. 
Therefore, our approach begins by localizing the data near finitely many points, and then applying arguments based on the finite propagation speeds of the hyperbolic operator to reconstruct solutions corresponding to general data. 
Consequently, if the data $z_0, F$ given in \eqref{linear}-\eqref{icd-linear} are globally supported in $\Omega$, then even for the linearized problem, we can only obtain a local-in-time solution.

\begin{prop}\label{prop-lift-cpbt}
    Assume that the angles $\omega_n$ ($1\le n\le N$) on the boundary $\partial\Omega$ obey the condition \eqref{angle-cdt} for a fixed integer $m\ge 3$, and the background state satisfies \eqref{assum-Z}. Then, for any given $(z_0, F)$ satisfying \eqref{b-icd}, \eqref{F-cdt} and \eqref{z0F-cdt}, and for each $\bar{x}\in \overline{\Omega}$, there exist two open neighborhoods $V_1(\bar{x})\subset V_2(\bar{x})$ of $\bar{x}$ in $\overline{\Omega}$, one can find $z_0^{\bar{x}}\in H^{m}(\Omega), F^{\bar{x}}\in H^{m}((0,T)\times\Omega)$ satisfying the following properties:
    \begin{enumerate}
        \item[(1)] $z^{\bar{x}}_0 = z_0(x)$ in $V_1(\bar{x})$ and $z^{\bar{x}}_0(x) = 0$ in $(V_2(\bar{x}))^c$, $F^{\bar{x}} = \varphi^{\bar{x}}(x) F(t,x)$ for some $\varphi^{\bar{x}}\in C^\infty(\bar{\Omega})$ independent of $z_0, F$, with $\varphi^{\bar{x}} = 1$ in $V_1(\bar{x})$ and vanishes outside $V_2(\bar{x})$. The mapping $(z_0, F)\mapsto z_0^{\bar{x}}$ is linear, and there exists a continuous, increasing function $\phi: \overline{\R^+}\to \overline{\R^+}$ such that, for $l=0,1$,
        \begin{equation}\label{t=0-est}
            \|z^{\bar{x}}_0\|_{m-l} \le \phi(\normmm{Z(0)}_{m-1})(\|z_0\|_{m-l} + \normmm{F(0)}_{m-l-1}).
        \end{equation}
        \item[(2)] $z_0^{\bar{x}}, F^{\bar{x}}$ satisfy the compatibility conditions of \eqref{linear} and \eqref{bcd-linear}  up to order $m-1$ when $t=0$, and $\mathcal{B}' z^{\bar{x}}_0 = 0$ on $\partial\Omega\backslash\{q_1,\ldots, q_N\}$.
        
        \item[(3)] 
        One can find approximate sequences
        $\{(z^{\bar{x}}_0)^k\}_{k} \subset H^{m+\frac{5}{2}}(\Omega)$, $\{(F^{\bar{x}})^k\}_{k} \subset C^\infty([0,T]\times\overline{\Omega})$
        satisfying the compatibility conditions  of \eqref{linear} and \eqref{bcd-linear} up to order $m+1$, $(F^{\bar{x}})^k$ satisfies \eqref{F-cdt}, $\mathcal{B}' (z_0^{\bar{x}})^k = 0$ on $\partial\Omega\backslash\{q_1,\ldots, q_N\}$,
        and $(z_0^{\bar{x}})^k \to z_0^{\bar{x}}$ in $H^{m}(\Omega)$ and $(F^{\bar{x}})^k \to F^{\bar{x}}$ in $H^{m}((0,T)\times\Omega)$ as $k\to +\infty$.
    \end{enumerate}
\end{prop}
The proof of this proposition closely follows that of \cite[Proposition~3.1]{godin20212d}, so we shall provide only an outline and highlight the key steps in Appendix A.

\subsection{Proof of Theorem~\ref{thm-LIBVP}}

Let $\bar{x}, (z_0^{\bar{x}})^k, (F^{\bar{x}})^k$ be given as in Proposition~\ref{prop-lift-cpbt}. To prove Theorem~\ref{thm-LIBVP}, we first construct an auxiliary function to eliminate the initial data.
The following lemma provides a modified trace lifting theorem.
\begin{lemma}\label{lemma-modified-lift}
    Let $m\ge 1$ be an integer and $T>0$ fixed. For any given $f_j \in H^{m-j}((\overline{\R^+})^2)$ ($j = 0,\ldots, m$) satisfying $f_j = 0$ at $x_1 = 0$ for all $0\le j\le m-1$, one can find a function $f \in CH^{m}([0,T]\times (\overline{\R^+})^2)$ such that $\partial_t^j f(0,x) = f_j(x)$ for all $j=0,\ldots,m$, and $f = 0$ at $x_1 = 0$.
\end{lemma}
\begin{proof}
First, we extend $f_j$ ($0\le j\le m$) to be defined in the whole $\R^2$ as $\tilde{f}_j\in H^{m-j}(\R^2)$ satisfying the odd symmetry $\tilde{f}_j(x_1,x_2) = -\tilde{f}_j(-x_1,x_2)$. Now, for the following Cauchy problem of a strictly hyperbolic equation
    \begin{equation}\label{CP}
        \begin{cases}
            \left(\prod_{k=1}^{m+1} (\partial_t - k (-\Delta)^{\frac{1}{2}})\right) \tilde{f} = 0,  \text{ in } [0,T]\times \R^2,\\
            \partial_t^j \tilde{f}(0,x) = \tilde{f}_j(x), \quad x\in \R^2,\, j=0,\ldots,m
        \end{cases}
    \end{equation}
  there is a unique solution $\tilde{f}\in CH^{m}([0,T]\times \R^2)$ by using the classical theory of hyperbolic problems. cf.  \cite[Chapter~6, Theorem~4.9]{chazarain2011introduction}. Moreover, since $-\tilde{f}(-x_1, x_2)$ is also a solution of \eqref{CP}, the uniqueness implies $\tilde{f} = 0$ when $x_1 = 0$. Hence, the restriction of $\tilde{f}$ to $(\overline{\R^+})^2$ provides the desired function $f$.
\end{proof}

\begin{lemma}\label{lemma-0thappx} For the approximate data $(z_0^{\bar{x}})^k, (F^{\bar{x}})^k$ given in Proposition~\ref{prop-lift-cpbt}, 
    there exists $(\zeta^{\bar{x}})^k\in CH^{m+2}([0,T]\times\Omega)$ such that
    \begin{equation*}
        (\zeta^{\bar{x}})^k(0,x) = (z^{\bar{x}}_0)^k(x),\quad \partial_t^j\big((L+\mathbb{B})(Z)(\zeta^{\bar{x}})^k - (F^{\bar{x}})^k\big)(0,x) = 0\quad \text{ for } 0\le j\le m+1,
    \end{equation*}
    and  ${\cal B} (\zeta^{\bar{x}})^k = {\cal B}' (\zeta^{\bar{x}})^k = 0$ on $\partial\Omega\backslash\{q_1,\ldots,q_N\}$.
\end{lemma}
\begin{proof} If $z\in CH^{m+2}((0,T)\times\Omega)$ satisfies
the following problem
$$\begin{cases}
(L+\mathbb{B})(Z) z = (F^{\bar{x}})^k\\
z|_{t=0}=(z^{\bar{x}}_0)^k(x)
\end{cases}$$
then one can easily determine
$$
\partial_t^j z(0,x)=(z^{\bar{x}})^k_{(j)}(x)\in H^{m+2-j}(\Omega) 
$$    
for all $j=0,1,\ldots, m+2$.
     The compatibility conditions satisfied by $(z^{\bar{x}}_0)^k, (F^{\bar{x}})^k$ imply ${\cal B}(z^{\bar{x}})^k_{(j)} = 0$ on $\partial\Omega\backslash\{q_1,\ldots,q_N\}$, for $j =0,\ldots, m+1$. Moreover, ${\cal B}' (z^{\bar{x}}_0)^k = 0$ on $\partial\Omega\backslash\{q_1,\ldots, q_N\}$. By \eqref{bcd-b-eq} and its time derivatives, we also obtain ${\cal B}'(z^{\bar{x}})^k_{(j)} = 0$ on $\partial\Omega\backslash\{q_1,\ldots,q_N\}$ for $j = 0,\ldots, m+1$. Then, by Lemma~\ref{lemma-modified-lift}, one can choose $(\zeta^{\bar{x}})^k \in CH^{m+2}([0,T]\times\Omega)$ such that $\partial_t^j(\zeta^{\bar{x}})^k(0,x) = (z^{\bar{x}})^k_{(j)}(x)$ for $j=0,\ldots, m+2$, and ${\cal B} (\zeta^{\bar{x}})^k = {\cal B}' (\zeta^{\bar{x}})^k = 0$ on $\partial\Omega\backslash\{q_1,\ldots,q_N\}$.
\end{proof}

\begin{proof}[Proof of Theorem~\ref{thm-LIBVP}]
    Let $\bar{x}, (z^{\bar{x}}_0)^k, (F^{\bar{x}})^k, (\zeta^{\bar{x}})^k$ be as in Lemma~\ref{lemma-0thappx}. Define $(G^{\bar{x}})^k = (F^{\bar{x}})^k - (L+\mathbb{B})(Z) (\zeta^{\bar{x}})^k$. Then $(G^{\bar{x}})^k \in CH^{m+1}([0,T]\times\Omega)\subset H^{m+1}((0,T)\times \Omega)$, $(G^{\bar{x}})^k$ satisfies \eqref{F-cdt}, and $\partial_t^j (G^{\bar{x}})^k = 0$ at $t=0$ for $0\le j\le m+1$. By extending $(G^{\bar{x}})^k$ by zero for $t<0$, we obtain a function belonging to $H^{m+1}((-\infty,T)\times \Omega)\subset H^{m+1}_*((-\infty,T)\times \Omega)$.
    
    By Proposition~\ref{prop-regularity} and Remark~\ref{rmk-regularity}, there exists $$(\sigma^{\bar{x}})^k \in \bigcap_{j=0}^{[\frac{m}{2}]} \left(C H^j_x(H^{m+1-2j}_{\tg})\cap H^j_x(H^{m+1-2j}_{\tg})\right)((-\infty,T]\times\Omega),$$ 
    such that $(L+\mathbb{B})(Z)(\sigma^{\bar{x}})^k = (G^{\bar{x}})^k$ in $(-\infty,T)\times \Omega$, ${\cal B}(\sigma^{\bar{x}})^k={\cal B}'(\sigma^{\bar{x}})^k=0$ on $\partial\Omega\backslash\{q_1,\ldots,q_N\}$, and $(\sigma^{\bar{x}})^k=0$ when $t\le 0$.

    Since
    $(\zeta^{\bar{x}})^k \in CH^{m+2}([0,T]\times\Omega) \subset X^{m+1}_*([0,T]\times\Omega)$, we set $$ (z^{\bar{x}})^k = ((u^{\bar{x}})^k, (b^{\bar{x}})^k, (\p^{\bar{x}})^k, (s^{\bar{x}})^k)^t = {(\sigma^{\bar{x}})^k}+ (\zeta^{\bar{x}})^k.$$ 
    Then $(z^{\bar{x}})^k \in \cap_{j=0}^{[\frac{m}{2}]} C H^j_x(H^{m+1-2j}_{\tg})([0,T]\times\Omega)$ is the solution to the IBVP \eqref{linear}-\eqref{icd-linear} with the data $z_0,F$ being replaced by $(z^{\bar{x}}_0)^k, (F^{\bar{x}})^k$ respectively, and satisfies ${\cal B}(z^{\bar{x}})^k={\cal B}'(z^{\bar{x}})^k=0$ on $\partial\Omega\backslash\{q_1,\ldots,q_N\}$.
    
    Hence, $\nabla (\p^{\bar{x}})^k, \nabla\cdot (u^{\bar{x}})^k \in \cap_{j=0}^{[\frac{m}{2}]} C H^j_x(H^{m-2j}_{\tg})([0,T]\times\Omega) = X^m_*([0,T]\times \Omega)$. 
    By Proposition~\ref{prop-regularity}, the estimate \eqref{apriori} (and \eqref{apriori-fine} if $m\ge 8$) holds with $z,F$ being replaced by $(z^{\bar{x}})^k$, $(F^{\bar{x}})^k$. Moreover, $\{(z^{\bar{x}})^k\}_{k}$ forms a Cauchy sequence in $X^{m}_*([0,T]\times\Omega)$, and its limit $z^{\bar{x}}$ solves \eqref{linear}-\eqref{icd-linear} in $X^{m}_*([0,T]\times\Omega)$, satisfying the same estimates as given in \eqref{apriori} and \eqref{apriori-fine} respectively,
    with $z_0,F$ being replaced by $z^{\bar{x}}_0, F^{\bar{x}}$.

    Now we turn to the general case of $z_0, F$ (not localized near $\bar{x}$). For $\bar{x}\in \overline{\Omega}$, let $V_1(\bar{x})$ be as in Proposition~\ref{prop-lift-cpbt}, and choose $r_{\bar{x}}>0$ such that $\{x\in \overline{\Omega}, |x-\bar{x}| < r_{\bar{x}}\}\subset V_1(\bar{x})$. Set $B^{\bar{x}}_r = \{x\in \overline{\Omega}, |x-\bar{x}|< r\}$, $C^{\bar{x}}_{\varepsilon} = \{(t,x)\in [0,T]\times \overline{\Omega}, t+\varepsilon |x-\bar{x}| < \varepsilon r_{\bar{x}}\}$, where $\varepsilon>0$ is small. We fix $\varepsilon < (S_p)^{-1}$, where 
    \begin{equation*}
        \begin{aligned}
            S_p &= \esssup_{\omega \in \mathbb{S}^1, (t,x)\in \overline{(0,T)\times\Omega}} \max\{|\text{eigenvalue of  } \omega_1 A_1(Z) + \omega_2 A_2(Z)|\}\\
            &\le \|U\|_{L^\infty}+2\|\sqrt{R(|B|^2 + Q)}\|_{L^\infty}.
        \end{aligned}
    \end{equation*}
    This ensures that $C^{\bar{x}}_\varepsilon$ lies within the domain of dependence of $B^{\bar{x}}_{r_{\bar{x}}}$. 
    Choose finitely many points $\bar{x}_1,\ldots, \bar{x}_d\in \overline{\Omega}$ so that $B^{\bar{x}_1}_{\frac{1}{2}r_{\bar{x}_1}}, \ldots, B^{\bar{x}_d}_{\frac{1}{2}r_{\bar{x}_d}}$ cover $\overline{\Omega}$. 
    For each $i=1,\ldots, d$, there exists a unique solution $z^{\bar{x}_i}\in X^{m}_*([0,T]\times\Omega)$ of \eqref{linear}-\eqref{icd-linear} associated with $z_0^{\bar{x}_i}, F^{\bar{x}_i}$.
    Let $ T_0 = \frac{\varepsilon}{2}\min (r_{\bar{x}_1},\ldots,r_{\bar{x}_d})$, so that $[0,T_0]\times\overline{\Omega} \subset \cup_{1\le i\le d} C^{\bar{x}_i}_\varepsilon$. If $C^{\bar{x}_i}_\varepsilon\cap C^{\bar{x}_j}_\varepsilon \ne \emptyset$, then $z^{\bar{x}_i} = z^{\bar{x}_j}$ in $C^{\bar{x}_i}_\varepsilon\cap C^{\bar{x}_j}_\varepsilon$, that is because $(z^{\bar{x}_i}_0,F^{\bar{x}_i}) =(z_0,F)= (z^{\bar{x}_j}_0,F^{\bar{x}_j})$ on $B^{\bar{x}_i}_{r_{\bar{x}_i}}\cap B^{\bar{x}_j}_{r_{\bar{x}_j}}$.
    Set $z(t,x) = z^{\bar{x}_i}(t,x)$ for $(t,x)\in C^{\bar{x}_i}_\varepsilon$ and $t\le T_0$. Then $z$ is the unique solution in $X^{m}_*([0,T_0]\times \Omega)$ to the linear IBVP \eqref{linear}-\eqref{icd-linear} associated with the data $z_0, F$. Proposition~\ref{prop-lift-cpbt} implies that $r_{\bar{x}}$ depends on the geometry quantities of $\Omega$ and $\mathcal{K}$, while $\varepsilon$ also depends on $\mathcal{K}$. Hence, the lifespan $T_0$ of $z$ depends on both $\Omega$ and $\mathcal{K}$.

    Finally, since each $(z^{\bar{x}})^k$ satisfies the estimate \eqref{apriori} (and \eqref{apriori-fine} if $m\ge 8$) with $z_0, F$ being replaced by $z^{\bar{x}}_0, F^{\bar{x}}$, it follows, together with \eqref{t=0-est}, that $z$ satisfies \eqref{apriori} (and \eqref{apriori-fine} if $m\ge 8$) with $z_0, F$ for $t\le T_0$.
\end{proof}

\section{Linearized problem with a non-smooth background state}

The main proposal of this section is to study the  linearized IBVP \eqref{linear}-\eqref{icd-linear} with non-smooth background state. 
We begin by defining the space
\begin{equation*}
 	 \begin{aligned}
 	 Y^{m}_*((0,T)\times\Omega) = \cap_{k=0}^m W^{k,\infty}((0,T);H^{m-k}_*(\Omega))
 	 \end{aligned}
\end{equation*}
with the norm $\|f\|_{Y^{m}_*((0,T)\times\Omega)} = \sum_{k=0}^m \esssup_{t\in (0,T)}\|\partial_t^k f(t)\|_{H^{m-k}_*(\Omega)}$.
This gives the embedding
\begin{equation*}
\begin{aligned}
    Y^m_*((0,T)\times \Omega) \hookrightarrow \bigcap_{j=0}^{m-1} C^j([0,T];H^{m-1-j}_*(\Omega)) = X^{m-1}_*([0,T]\times\Omega),
\end{aligned}
\end{equation*}
which implies that for any $z \in Y^m_*((0,T)\times \Omega)$, the time traces $\partial_t^j z(0)$ are well-defined and belong to $H^{m-1-j}_*(\Omega)$ for all $0\le j\le m-1$.

From now on, we study the linearized IBVP \eqref{linear}-\eqref{icd-linear} with background states $Z\in \tilde{Y}^{m}_*([0,T]\times\Omega)$, where
\begin{equation*}
    \tilde{Y}^{m}_*([0,T]\times \Omega) = \{z\in Y^{m}_*((0,T)\times \Omega):\partial_t^j z(0) \in H^{m-j}(\Omega), \text{ for } 0\le j\le m-1\},
\end{equation*}
where $(0)$ refers to the value at $t=0$.
We say that $Z^k\to Z$ in $\tilde{Y}^{m}_*([0,T]\times \Omega)$ if
\begin{equation*}
    \begin{aligned}
        Z^k\to Z \text{ in }Y^{m}_*((0,T)\times \Omega), \quad
        \partial_t^j Z^k(0)\to \partial_t^j Z(0) \text{ in }H^{m-j}(\Omega), \quad 0\le j\le m-1.
    \end{aligned}
\end{equation*}

In this section, we shall extend Theorem~\ref{thm-LIBVP} to the case $Z\in \tilde{Y}^{m}_*([0,T]\times\Omega)$ by using an approximation procedure. Introduce the notations
\begin{equation}\label{4.1}
    M^\flat_m(z,t) = \esssup_{\tau\in (0,t)} \normmm{z(\tau)}_{m,*},\quad M^\sharp_m(z,t) = M^\flat_m(z,t) + \sum_{j=0}^{m-1} \|\partial_t^j z(0)\|_{H^{m-j}_*(\Omega)}.
\end{equation}
\begin{theorem}\label{thm-LIBVP'}
    Assume that $m\ge 7$ is an integer, the angle $\omega_n$ obeys \eqref{angle-cdt} for all $n=1,\ldots, N$, that \eqref{assum-Z} holds except that now $Z \in \tilde{Y}^{m}_*([0,T]\times\Omega)$, and that \eqref{F-cdt}, \eqref{z0F-cdt} hold. Then the linearized problem \eqref{linear}-\eqref{icd-linear} has a unique solution $z\in \tilde{Y}^{m}_*([0,T_0]\times \Omega)$ for some $T_0\in (0,T]$ depending on $\Omega$ and $\mathcal{K}$. 
    Moreover, there exist continuous increasing functions $\phi,\phi_1,\phi_2 : \overline{\R^+} \to \overline{\R^+}$, independent of $T_0$ and may depending on $\mathcal{K}$, such that for $0\le t\le T_0$, the solution $z$ satisfies the following estimates:
    \begin{equation}\label{apriori-Linfty}
        M^\flat_m(z,t) \le \phi(M^\sharp_m(Z,T)) \left(
        \sum_{j=0}^{m-1} \|\partial_t^j z(0)\|_{H^{m-j}_*(\Omega)} + \normmm{F(0)}_{m-1,*} + \int_0^t \normmm{F(\tau)}_{m,*} \d \tau
        \right),
    \end{equation}
and \begin{equation}\label{apriori'-Linfty}
        \begin{aligned}
            M^\flat_m(z,t) &\le \phi_1(M_{m-1}(Z,T))\Bigg( \sum_{j=0}^{m-1} \|\partial_t^j z(0)\|_{H^{m-j}_*(\Omega)} + \normmm{F(0)}_{m-1,*} + \int_0^t \normmm{F(\tau)}_{m,*} \d \tau\\
            & +\phi_2(M^\sharp_m(Z,t)) \left(
            \int_0^t \normmm{F(\tau)}_{m-1,*} + \Lambda_{m-1}(z,F,\tau) \d \tau
            \right)
            \Bigg),
        \end{aligned}
    \end{equation}
        as $m\ge 8$.   
\end{theorem}

The proof of Theorem~\ref{thm-LIBVP'} will be done by an approximation procedure, where we shall first approximate $Z$ by a sequence $\{Z^k\}_k \subset C^\infty ([0, T]\times \overline{\Omega})$. This necessitates the simultaneous approximation of $z_0$ and $F$ by sequences $\{z_0^k\}_k$ and $\{F^k\}_k$, respectively, so that the compatibility conditions for $L(Z^k)$ and \eqref{bcd-linear} are satisfied up to order $m-1$. In the process of preserving (though not improving) these compatibility conditions, we have to localize $z_0^k$ and $F^k$ to the neighborhoods of finitely many points in $\overline{\Omega}$ as was done in Proposition~\ref{prop-lift-cpbt}. Accordingly, we establish the following proposition.
\begin{prop}\label{prop-nonC}
    Assume that $m\ge 3$ is an integer, the angle $\omega_n$ obeys the condition \eqref{angle-cdt} for all $n=1,\ldots, N$, that \eqref{assum-Z} holds except that $Z\in \tilde{Y}^{m}_*([0,T]\times \Omega)$, and that \eqref{z0F-cdt} holds. Then, the following four points hold.
    \begin{enumerate}
        \item[(1)] For $k \in \mathbb{N}\backslash\{0\}$, there exists $Z^k \in C^\infty([0,T]\times \overline{\Omega})$, such that
        \begin{itemize}
            \item ${\cal{B}} Z^k = {\cal{B}}' Z^k = 0$ on $\partial\Omega\backslash\{q_1,\ldots,q_N\}$;
            \item $Z^k \to Z$ in $\tilde{Y}^{m}_*([0,T]\times \Omega)$ as $k\to +\infty$;
            \item the sequence $\{Z^k\}_{k}$ is bounded in $X^m_*([0,T]\times \Omega)$, and there exists a constant $C$ independent of $Z,k,t$, such that
            \begin{equation}\label{Z_ess}
                \normmm{Z^k(t)}_{m,*} \le C\big( M^\sharp_m(Z,t) + \frac{1}{k}\big);
            \end{equation}
            \item each $Z^k$ takes its value in a compact subset $\mathcal{K}'$ of $\R^4\times \mathcal{V}$ independent of $k$.
        \end{itemize}
        
        \item[(2)] For any $\bar{x}\in \overline{\Omega}$, one can construct the localized $z_0^{\bar{x}}\in H^{m}(\Omega)$, $F^{\bar{x}}\in H^{m}((0,T)\times\Omega)$, such that (1) and (2) of Proposition~\ref{prop-lift-cpbt} are satisfied. 
        
        \item[(3)]
        There exists a sequence $(z^{\bar{x}}_0)^k\in H^{m+\frac{1}{2}}(\Omega)$ with $(z^{\bar{x}}_0)^k\to z^{\bar{x}}_0$ in $H^{m}(\Omega)$, such that the compatibility conditions up to order $m-1$ are satisfied for the equation $(L+\mathbb{B})(Z)z= F^{\bar{x}}$ with the boundary condition \eqref{bcd-linear} and the initial data $z(0,x)=(z^{\bar{x}}_0)^k(x)$, and ${\cal B} (z^{\bar{x}}_0)^k = {\cal B}' (z^{\bar{x}}_0)^k = 0$ hold on $\partial\Omega\backslash\{q_1,\ldots,q_N\}$.
        
        \item[(4)] One can find $(F^{\bar{x}})^k \in H^{m}((0,T)\times\Omega)$
        satisfying \eqref{F-cdt}, and $(F^{\bar{x}})^k \to F^{\bar{x}}$ in $H^{m}((0,T)\times\Omega)$ as $k\to +\infty$, such that the compatibility conditions up to order $m-1$ hold for the equation $(L+\mathbb{B})(Z^{l(k)})z=(F^{\bar{x}})^k$ with the boundary condition \eqref{bcd-linear} and the initial data $z(0,x)=(z^{\bar{x}}_0)^k(x)$, where $l(k)\to +\infty$ as $k\to \infty$. 

    \end{enumerate}    
\end{prop}
The proof of Proposition~\ref{prop-nonC} shall be given later. Now, we verify how from Proposition~\ref{prop-nonC}, one can derive Theorem~\ref{thm-LIBVP'}.

\begin{proof}[Proof of Theorem~\ref{thm-LIBVP'}.]
    Let $\bar{x}, Z^k, (z^{\bar{x}}_0)^k, (F^{\bar{x}})^k$ be as given in Proposition~\ref{prop-nonC}. By using Theorem~\ref{thm-LIBVP}, one can find a solution $(z^{\bar{x}})^k\in X^{m}_*([0,T]\times \Omega)$ to the problem 
    \begin{equation}\label{4.5}
        \begin{cases}
            (L+\mathbb{B})(Z^k) (z^{\bar{x}})^k = (F^{\bar{x}})^k, & (0,T)\times \Omega,\\
        {\cal B} (z^{\bar{x}})^k = 0,&(0,T)\times (\partial\Omega\backslash\{q_1,\ldots,q_N\}),\\
        (z^{\bar{x}})^k(0,x) = (z^{\bar{x}}_0)^k(x), & t=0.
        \end{cases}
    \end{equation}
Moreover, the estimates in Theorem~\ref{thm-LIBVP} shows that $\{(z^{\bar{x}})^k\}_{k\ge 1}$ is bounded in $X^{m}_*([0,T]\times \Omega)\subset Y^m_*((0,T)\times \Omega)$. By the weak-$\star$ compactness of $Y^m_*((0,T)\times \Omega)$, and the compact embedding $X^{m}_*([0,T]\times\Omega)\hookrightarrow X^{m-2}_*([0,T]\times\Omega)$, 
one can find $z^{\bar{x}}\in Y^m_*((0,T)\times \Omega)$ such that, up to extracting a subsequence, $(z^{\bar{x}})^{k} \to z^{\bar{x}}$ weakly-$\star$ in $Y^m_*((0,T)\times \Omega)$ and strongly in $X^{m-2}_*([0, T]\times \Omega)$. 

    The strong convergence in $X^{m-2}_*([0,T]\times \Omega)$ implies that $z^{\bar{x}}$ actually solves the problem \begin{equation*}
        \begin{cases}
            (L+\mathbb{B})(Z) z^{\bar{x}} = F^{\bar{x}}, & (0,T)\times \Omega,\\
        {\cal B} z^{\bar{x}} = 0,&(0,T)\times (\partial\Omega\backslash\{q_1,\ldots,q_N\}),\\
        z^{\bar{x}}(0,x) = z^{\bar{x}}_0(x), & t=0.
        \end{cases}
    \end{equation*}
    Recall that  $L(Z) = A_0(Z)\partial_t + A_1(Z)\partial_1 + A_2(Z)\partial_2$.
    Since $z^{\bar{x}}\in Y^m_*((0,T)\times\Omega) \hookrightarrow X^{m-1}_*([0,T]\times \Omega)$, one can apply times-derivatives on $(L+\mathbb{B})(Z) z^{\bar{x}} = F^{\bar{x}}$ and then take the restriction at $t=0$, obtaining for $1\le j\le m-1$,
    \begin{equation}\label{t-trace-formula}
        \begin{aligned}
            \partial_t^j z^{\bar{x}}(0)  = & A_0^{-1}(Z)(0) \Bigg( \partial_t^{j-1} F^{\bar{x}}(0) 
            - \sum_{q=1}^{j-1} \binom{j-1}{q}\partial_t^q A_0(Z)(0) \partial_t^{j-q} z^{\bar{x}}(0)\\
            & - \sum_{q=0}^{j-1} \binom{j-1}{q} \left( \sum_{i=1,2}\partial_t^q A_i(Z)(0) \partial_t^{j-1-q} \partial_i z^{\bar{x}}(0) + \partial_t^{q} \mathbb{B}(Z)(0)\partial_t^{j-1-q} z^{\bar{x}}(0)
        \right)
        \Bigg),
        \end{aligned}
    \end{equation}
    where $f(0)$ denotes the value of a related function $f$ at $t=0$ as usual.
    Since $\partial_t^j Z(0) \in H^{m-j}(\Omega)$ for $0\le j\le m-1$, by
    \eqref{t-trace-formula} and $z^{\bar{x}}(0) = z^{\bar{x}}_0$, one has $\partial_t^{j} z^{\bar{x}}(0) \in H^{m-j}(\Omega)$ for $0\le j\le m-1$. Therefore, $z^{\bar{x}}\in \tilde{Y}^m_*([0,T]\times \Omega)$.

    Moreover, by applying Theorem~\ref{thm-LIBVP} for the problem \eqref{4.5}, it follows that  $(z^{\bar{x}})^k$ satisfies the estimate
    \begin{equation}\label{zbark-est}
        \begin{aligned}
            \normmm{(z^{\bar{x}})^k(t)}_{m,*}  \le & \phi (\sup_{t\in [0,T]}\normmm{Z^k(t)}_{m,*}) \bigg(\sum_{j=0}^m \|\partial_t^j
            (z^{\bar{x}})^k(0)\|_{H^{m-j}_*(\Omega)}\\
            & + \normmm{(F^{\bar{x}})^k(0)}_{m-1,*} + \int_0^t \normmm{(F^{\bar{x}})^k(\tau)}_{m,*} \d \tau\bigg),
        \end{aligned}
    \end{equation}
    with $\phi:\overline{\R^+}\to \overline{\R^+}$ being a continuous increasing function.
    On the other hand, from the equation given in \eqref{4.5}, we have
    \begin{equation*}
        \begin{aligned}
            &\partial_t^m (z^{\bar{x}})^k(0)  = A_0^{-1}(Z^k)(0) \Bigg( \partial_t^{m-1} (F^{\bar{x}})^k(0) 
            - \sum_{q=1}^{m-1} \binom{m-1}{q}\partial_t^q A_0(Z^k)(0) \partial_t^{m-q} (z^{\bar{x}})^k(0)\\
            & - \sum_{q=0}^{m-1} \binom{m-1}{q} \bigg( \sum_{i=1,2}\partial_t^q A_i(Z^k)(0) \partial_t^{m-1-q} \partial_i (z^{\bar{x}})^k(0)
            + \partial_t^{q} \mathbb{B}(Z)(0)\partial_t^{m-1-q} (z^{\bar{x}})^k(0)
        \bigg)
        \Bigg),
        \end{aligned}
    \end{equation*}
   which implies
    \begin{equation*}
        \begin{aligned}
            \|\partial_t^m (z^{\bar{x}})^k (0)\|_{L^2(\Omega)} \le & P(\normmm{Z^k(0)}_{m-1,*}) \bigg(\normmm{(z^{\bar{x}})^k(0)}_{m-1,*} \\
            &+ \|\partial_t^{m-1} (z^{\bar{x}})^k(0)\|_{H^1_*(\Omega)} + \normmm{(F^{\bar{x}})^k(0)}_{m-1,*}
        \bigg).
        \end{aligned}
    \end{equation*}
    Plugging the above estimate and \eqref{Z_ess} into \eqref{zbark-est}, one obtains
    \begin{equation*}
        \begin{aligned}
            &\esssup_{\tau \in (0,t)}\normmm{(z^{\bar{x}})^k(\tau)}_{m,*} \le 
            \phi_1 \left(
            M^\sharp_m(Z,T) + \frac{1}{k}
            \right) \cdot
            \\
            &\quad \bigg(\sum_{j=0}^{m-1} \|\partial_t^j
            (z^{\bar{x}})^k(0)\|_{H^{m-j}_*(\Omega)}
            + \normmm{(F^{\bar{x}})^k(0)}_{m-1,*} + \int_0^t \normmm{(F^{\bar{x}})^k(\tau)}_{m,*} \d \tau\bigg).
        \end{aligned}
    \end{equation*}
    Using the lower semi-continuity of the norm $\|\cdot\|_{Y^m_*((0,T)\times\Omega)}$ with respect to the weak-$\star$ convergence, and taking $k\to +\infty$, one obtains the estimate \eqref{apriori-Linfty} with $z, Z, F$ being replaced by $z^{\bar{x}}, Z, F^{\bar{x}}$.
    If $m\ge 8$, by a similar analysis, one also can deduce the estimate \eqref{apriori'-Linfty} with $z, Z, F$ being replaced by $z^{\bar{x}}, Z, F^{\bar{x}}$ respectively.

    For general, non-localized $z_0, F$, we can use the same arguments as given in the last part of the proof of Theorem~\ref{thm-LIBVP} to conclude the results given in Theorem \ref{thm-LIBVP'}.
\end{proof}

Before proving Proposition~\ref{prop-nonC}, we introduce another auxiliary lemma.
\begin{lemma}\label{lemma-Jsp}
    There exists a smoothing operator $J^{\text{sp}}_\varepsilon$, such that for any $s\in \mathbb{N}$, if $f_1\in H^s(\Omega)$, $f_2 \in H^s_*(\Omega)$, ${\cal B}f_1 = {\cal B}' f_1 = {\cal B} f_2 = {\cal B}' f_2 = 0 $ on $\partial\Omega \backslash\{q_1,\ldots,q_N\}$, then
    \begin{itemize}
        \item $J^{\text{sp}}_\varepsilon f_1, J^{\text{sp}}_\varepsilon f_2 \in C^\infty(\overline{\Omega})$;
        \item $J^{\text{sp}}_\varepsilon f_1 \to f_1$ in $H^s(\Omega)$, $J^{\text{sp}}_\varepsilon f_2 \to f_2$ in $H^s_*(\Omega)$, as $\varepsilon\to 0$;
        \item ${\cal B} (J^{\text{sp}}_\varepsilon f_1) = {\cal B}' (J^{\text{sp}}_\varepsilon f_1) = {\cal B} (J^{\text{sp}}_\varepsilon f_2) = {\cal B}' (J^{\text{sp}}_\varepsilon f_2) = 0 $ on $\partial\Omega \backslash\{q_1,\ldots,q_N\}$.
    \end{itemize}
\end{lemma}
\begin{proof}
    After applying a partition of unity and a coordinate transformation, the proof reduces to the case that $f \in H^s(\Omega')$ or $H^s_*(\Omega')$, and $f=0$ when $x_1 = 0$, where $\Omega' = \R^+\times \R$ or $\R^+\times \R^+$. The case $\Omega' = \R^+\times \R^+$ can be reduced to $\Omega' = \R^+\times \R$ by a simple Stein-type extension of $f$ to be defined in $x_2<0$.
    We then extend $f$ oddly in $x_1<0$ to obtain $\tilde{f}$. Choose $j\in C^\infty_0(\R)$, with $\supp \, j\subset [-1,1]$, $j$ is even, non-negative, and $\int_\R j \d x = 1$.
    Define 
    $$
    (J^{\text{sp}}_\varepsilon f)(x) = \left(\int_{\R^2} \frac{1}{\varepsilon^2}j(\frac{x_1-y_1}{\varepsilon})j(\frac{x_2-y_2}{\varepsilon}) \tilde{f}(y) \d y\right)|_{x\in \Omega},$$ 
    which gives the desired operator.
\end{proof}

Now we prove Proposition~\ref{prop-nonC}.
\begin{proof}[Proof of Proposition~\ref{prop-nonC}]
(1)
Let $g_j = \partial_t^j Z(0)$ for $j=0,\ldots, m-1$. Obviously, $g_j \in H^{m-j}(\Omega)$, ${\cal B} g_j = {\cal B}' g_j = 0$ on $\partial\Omega\backslash\{q_1,\ldots, q_N\}$.
By Lemma~\ref{lemma-modified-lift}, one can construct $g\in CH^m([0,T+1]\times \Omega)$, such that $\partial_t^j g(0) = g_j$ for $j=0,\ldots, m-1$, $\partial_t^m g(0) = 0$, and ${\cal B} g = {\cal B}' g = 0$ on $\partial\Omega\backslash\{q_1,\ldots, q_N\}$. Then, one has
\begin{equation*}
    \normmm{g(0)}_{m,*} \le \sum_{j=0}^{m-1} \|\partial_t^j Z(0)\|_{H^{m-j}_*(\Omega)}.
\end{equation*}

Since $Z-g \in Y^{m}_*([0,T]\times \Omega)$, and $\partial_t^j(Z-g)(0) = 0$ for $0\le j \le m-1$, let $E_t(Z-g)$ denote the zero extension of $Z-g$ for $t<0$. Then $E_t(Z-g)\in Y^{m}_*((-\infty,T]\times \Omega)$. 
Let $J^{\text{sp}}_\varepsilon$ be the spatial smoothing operator given in Lemma~\ref{lemma-Jsp}, and let $j\in C^\infty_0(\R)$ satisfy $\supp \, j\subset [-1,1]$, $j$ even and non-negative, and $\int_\R j \d x = 1$.
For small $\varepsilon>0$, define
\begin{equation*}
    \begin{aligned}
    & Y^\varepsilon = \left(\int_{\R} \frac{1}{\varepsilon}j(\frac{t-\tau - 2\varepsilon}{\varepsilon}) \big(J^{\text{sp}}_\varepsilon E_t(Z-g)\big)(\tau) \d \tau \right)_{\big| t\in [0,T]},\\
    & W^\varepsilon = \left(\int_{\R} \frac{1}{\varepsilon}j(\frac{t-\tau + 2\varepsilon}{\varepsilon}) \big(J^{\text{sp}}_\varepsilon g\big)(\tau) \d \tau \right)_{\big| t\in [0,T]}.
\end{aligned}
\end{equation*}
Then, it is easy to have 
$$
\begin{cases}
Y^{\varepsilon}, W^{\varepsilon} \in C^\infty([0,T]\times \overline{\Omega}), \quad
\partial_t^j Y^{\varepsilon}(0) = 0 \quad (j=0,\ldots, m-1),\\ 
{\cal B} Y^\varepsilon = {\cal B}' Y^{\varepsilon} = {\cal B} W^{\varepsilon} = {\cal B}' W^{\varepsilon} = 0\quad {\rm on}~\partial\Omega\backslash\{q_1,\ldots, q_N\},
\end{cases}
$$
and
$$
\begin{cases}
Y^{\varepsilon} \to Z-g \quad {\rm in}~ Y^{m}_*((0,T)\times \Omega), \quad W^{\varepsilon} \to g\quad {\rm in}~CH^{m}([0,T]\times \Omega), \\
\partial_t^j W^{\varepsilon} (0) \to \partial_t^j g(0) = g_j\quad {\rm in}~H^{m-j}(\Omega), \quad j= 0,\ldots, m-1,     
\end{cases}
$$
as $\varepsilon\to 0$.

Obviously, one has, 
\begin{equation*}
    \begin{aligned}
        Y^\varepsilon + W^\varepsilon = \left(\int_{\R} \frac{1}{\varepsilon}j(\frac{t-\tau - 2\varepsilon}{\varepsilon}) J^{\text{sp}}_\varepsilon 
        \Big(
        E_t(Z-g)(\tau) + g(\tau+4\varepsilon)
        \Big)
        \d \tau \right)_{\big| t\in [0, T]},
    \end{aligned}
\end{equation*}
which implies that for any $t\in [0,T]$,
\begin{equation*}
    \begin{aligned}
        &\quad\normmm{(Y^\varepsilon + W^\varepsilon) (t)}_{m,*} = \sum_{j=0}^m \|\partial_t^j(Y^\varepsilon + W^\varepsilon) (t)\|_{H^{m-j}_*(\Omega)}\\
        & \le \sum_{j=0}^m 
        \int_{\R} \frac{1}{\varepsilon}j(\frac{t-\tau- 2\varepsilon}{\varepsilon}) \left\|J^{\text{sp}}_\varepsilon \Big(
        \partial_t^j E_t(Z-g)(\tau) + \partial_t^j g(\tau+4\varepsilon)
        \Big)\right\|_{H^{m-j}_*(\Omega)} \d\tau \\
        & \lesssim \esssup_{\tau \in [t-3\varepsilon, t-\varepsilon]} 
        \normmm{E_t(Z-g)(\tau) + g(\tau + 4\varepsilon)}_{m,*}
        +c(\varepsilon)
        \\
        & \lesssim
        \esssup_{\substack{\tau \in [t-3\varepsilon, t-\varepsilon]\\ \tau >0}} \left(\normmm{Z(\tau)}_{m,*}
        + \normmm{g(\tau+4\varepsilon) - g(\tau) }_{m,*}
        \right)\\
        & 
        \hspace{.2in}
        + \esssup_{\substack{\tau \in [t-3\varepsilon, t-\varepsilon]\\ \tau <0}} \normmm{g(\tau+4\varepsilon)}_{m,*} + c(\varepsilon),
    \end{aligned}
\end{equation*}
where $c(\varepsilon) \to 0$ as $\varepsilon\to 0$.
Since $g\in CH^m([0,T+1]\times \Omega)\subset X^m_*([0,T]\times \Omega)$, then for any $k\in \mathbb{N}\backslash\{0\}$, there exists $\varepsilon(k)\in (0,\frac{1}{k})$, such that for any $\tau_1,\tau_2 \in [0,T+1]$ with $|\tau_1-\tau_2|\le 4\varepsilon(k)$, one has
$\normmm{g(\tau_1) - g(\tau_2)}_{m,*} \le \frac{1}{k}$ and, moreover, $c(\varepsilon(k)) \le \frac{1}{k}$. Therefore, letting $\varepsilon = \varepsilon(k)$, one has
\begin{equation*}
    \begin{aligned}
        \normmm{(Y^\varepsilon + W^\varepsilon) (t)}_{m,*} &\lesssim M^\flat_m(Z,t) + \esssup_{\tau\in [0,t]} \normmm{g(\tau+4\varepsilon(k)) - g(\tau)}_{m,*}\\
        &+ \esssup_{\tau \in [0,4\varepsilon(k)]} \normmm{g(\tau)-g(0)}_{m,*} + \normmm{g(0)}_{m,*} + \varepsilon(k)\\
        &\lesssim M^\sharp(Z,t) + \frac{1}{k}.
    \end{aligned}
\end{equation*}
Hence, setting $Z^k = Y^{\varepsilon(k)} + W^{\varepsilon(k)}$, we find that $Z^k$ satisfies all the properties listed in Proposition~\ref{prop-nonC}(1).

(2)
Since we have $\partial_t^j Z(0) \in H^{m-j}(\Omega)$ for $0\le j\le m-1$, and no additional compatibility conditions are required in this proposition, the proofs of the second and third points given in Proposition \ref{prop-nonC} follow from an adaptation of Proposition~\ref{prop-lift-cpbt}, and are completely parallel to the argument in \cite[Proposition~5.1]{godin20212d}. We remark that the extra $\frac{1}{2}$-order regularity of $(z^{\bar{x}}_0)^k$ is introduced to facilitate the proof of the fourth point given in Proposition~\ref{prop-nonC}.

(3) Now, we turn to prove the assertion given in the fourth point of Proposition \ref{prop-nonC}.
Given $(z^{\bar{x}}_0)^k\in H^{m+\frac{1}{2}}(\Omega)$, $Z\in \tilde{Y}^m_*([0,T]\times \Omega)$, $F^{\bar{x}}\in H^m((0,T)\times \Omega)$, one can formally compute the values of $\partial_t^jz(0)$ for $0\le j\le m-1$
from the equations
\begin{equation}\label{sys-zjk}
z(0) = (z^{\bar{x}}_0)^k,\quad 
\partial_t^{j-1}\left((L+\mathbb{B})(Z)z\right)(0) = (\partial_t^{j-1} F^{\bar{x}})(0),\quad \forall 1\le j\le m-1,
\end{equation}
where $f(0)$ indicates the value a related function $f$ at $t=0$.
Explicitly, we define the functions $z^k_{(j)}(x)$, $j=0,\ldots, m-1$, by 
$z^k_{(0)} (x)= (z^{\bar{x}}_0)^k(x)$, and
\begin{equation}\label{eq-zjk}
\begin{aligned}
    z^k_{(j)}(x) = A^{-1}_0(Z)(0)\bigg(& \partial_t^{j-1} F^{\bar{x}}(0)- \sum_{q=1}^{j-1}\binom{j-1}{q} \partial_t^q A_0(Z)(0) z^k_{(j-q)}\\
    &-\sum_{q=0}^{j-1} \binom{j-1}{q} \big(\sum_{i=1,2} \partial_t^q A_i(Z)(0) \partial_i z^k_{(j-1-q)} + \partial_t^q \mathbb{B}(Z)(0) z^k_{(j-1-q)}\big)\bigg)
\end{aligned}
\end{equation}
for $j = 1,\ldots, m-1$. 
It is easy to have that $z^k_{(j)} \in H^{m-j+\frac{1}{2}}(\Omega)$, $0\le j\le m-1$ and ${\cal B} z^k_{(j)} = 0$ on $\partial\Omega\backslash\{q_1,\ldots, q_N\}$.

Recall that $F^{\bar{x}}$ satisfies \eqref{F-cdt}, we have ${\cal B}'(\partial_t^{j-1}F^{\bar{x}}(0)) = 0$ for all $0\le j\le m-1$ on $\partial\Omega\backslash\{q_1,\ldots, q_N\}$.
Comparing with \eqref{bcd-b-eq}, we see that \eqref{eq-zjk} implies, for $j= 1,\ldots, m-1$,
\begin{equation*}
    {\cal B}'z^k_{(j)} + \sum_{q=0}^{j-1} \binom{j-1}{q} \partial_t^q U(0) \cdot \nabla ({\cal B}'z^k_{(j-1-q)} ) = 0,
    \quad \text{on } \partial\Omega\backslash\{q_1,\ldots, q_N\}.
\end{equation*}
Since ${\cal B}' (z^{\bar{x}}_0)^k = 0$ as $x\in \partial\Omega\backslash\{q_1,\ldots, q_N\}$, we conclude that for all $0\le j\le m-1$, ${\cal B}'z^k_{(j)} = 0$ on $x\in \partial\Omega\backslash\{q_1,\ldots, q_N\}$. 

Now, we replace $Z$ with $Z^l$ and aim to find an approximate sequence $(F^{\bar{x}})^k$ such that the corresponding $\partial_t^j z(0)$ obtained from the system \eqref{sys-zjk} (with $Z, F^{\bar{x}}$ replaced by $Z^l, (F^{\bar{x}})^k$) also satisfies ${\cal B} (\partial_t^j z(0)) =0$ for $j=0,\ldots,m-1$.
It suffices to arrange that
    \begin{equation}\label{F-fix}
        \begin{aligned}
            \partial_t^{j-1}(F^{\bar{x}})^{k}(0) &= \partial_t^{j-1} F^{\bar{x}}(0) + \sum_{q=1}^{j-1} \binom{j-1}{q} \partial_t^q \left(A_0(Z^{l(k)}) - A_0(Z)\right)(0) z^k_{(j-q)}\\
            & + \sum_{q=0}^{j-1} \binom{j-1}{q} 
            \Bigg(
            \sum_{i=1,2}\partial_t^q \left(A_i(Z^{l(k)}) - A_i(Z)\right)(0) \partial_i z^k_{(j-1-q)}\\
            &
            + \partial_t^q \left(\mathbb{B}(Z^{l(k)}) - \mathbb{B}(Z)\right)(0) z^k_{(j-1-q)}
            \Bigg)
        \end{aligned}
    \end{equation}
    for $1\le j\le m-1$, where $l(k)$ is an increasing function.
    Since for any $q= 0,\ldots, m-1$, we have $z^k_{(q)} \in H^{m-q+\frac{1}{2}}(\Omega)$ for each $k$, and $\partial_t^q Z^l(0)\to \partial_t^q Z(0)$ in $H^{m-q}(\Omega)$ as $l\to +\infty$, one can take $l = l(k)$ sufficiently large for each $k$, so that by standard Sobolev embeddings, the right-hand side of \eqref{F-fix} belongs to $H^{m-j+\frac{1}{2}}(\Omega)$ and converges to $\partial_t^{j-1} F^{\bar{x}}(0)$ in $H^{m-j+\frac{1}{2}}(\Omega)$ as $k\to +\infty$.
    
    Moreover, we observe from \eqref{bcd-b-eq} that for $j=1,\ldots, m-1$,
    \begin{equation*}
        \begin{aligned}
        {\cal B}' (\text{RHS of \eqref{F-fix}}) &= 
        {\cal B}' \partial_t^{j-1} F^{\bar{x}}(0) + 
        {\cal B}'z^k_{(j)} \\
        &+ \sum_{q=0}^{j-1} \binom{j-1}{q} \partial_t^q \big(U^{l(k)}-U\big)(0) \cdot \nabla ({\cal B}'z^k_{(j-1-q)} ) = 0.
        \end{aligned}
    \end{equation*}
    
    Therefore, by applying a modified trace lifting theorem, one can construct $(F^{\bar{x}})^{k}\in H^{m}((0, T)\times \Omega)$, such that \eqref{F-fix} holds for all $j=1,\ldots, m-1$,
    $\partial_t^{m-1} (F^{\bar{x}})^k(0) = 0$,
    ${\cal B}'(F^{\bar{x}})^k = 0$ on $\partial\Omega \backslash\{q_1,\ldots, q_N\}$ and $(F^{\bar{x}})^k\to F^{\bar{x}}$ in $H^m((0,T)\times \Omega)$, as $k\to +\infty$.
\end{proof}

\section{Proof of Theorem \ref{main-thm}}

Based on the well-posedness result of the linearized problem given in Theorem \ref{thm-LIBVP'}, we shall use the standard Picard iteration scheme to prove the well-posedness of the nonlinear problem \eqref{MHD}-\eqref{bcd}.

Consider the initial boundary value problem \eqref{MHD}-\eqref{bcd}, for given initial data $z_0= (u_0,b_0,p_0,s_0)^t \in H^m(\Omega)$ satisfying the assumption given in Theorem \ref{main-thm}, one can inductively compute $z_{(j)}(x)=(\partial_t^ju,\partial_t^jb,\partial_t^jp,\partial_t^js)|_{t=0} $, $j=1,\ldots, m-1$, in terms of $z_{(0)} = z_0(x)$ and its derivatives in $x$. 
It is easy to have $z_{(j)} \in H^{m-j}(\Omega)$ for all  $0\le j\le m-1$. Since $z_0$ satisfies the compatibility conditions of \eqref{MHD} and \eqref{bcd} up to order $m-1$, it follows that $u_{(j)}\cdot \nu = b_{(j)} \cdot \nu = 0$ on $\partial\Omega\backslash\{q_1,\ldots,q_N\}$ for $0\le j\le m-1$.
Then, by Lemma~\ref{lemma-modified-lift}, for any $T>0$, one can construct $z^0\in CH^m([0, T]\times \Omega)\subset X^m_*([0, T]\times \Omega)$ such that $\partial_t^j z^0(0,x) = z_{(j)}(x)$ for $0\le j\le m-1$,  $\partial_t^m z^0(0) = 0$, and
${\cal B} z^0 = {\cal B}' z^0 = 0$ on $\partial\Omega\backslash\{q_1,\ldots, q_N\}$.

Now we introduce the following iteration scheme
\begin{equation}\label{Picard}
\begin{cases}
    (L+\mathbb{B})(z^k) y^{k+1} = 0, & \text{ in }    (0,T)\times \Omega,\\
    {\cal B}y^{k+1} = 0, & \text{ on } (0,T)\times (\partial\Omega \backslash\{q_1,\ldots,q_N\}),\\
    y^{k+1}(0,x) = \mathrm{p}(z_0), & \text{ in } \Omega,\\
    z^{k+1} = \mathrm{p}^{-1}(y^{k+1}),
\end{cases}
\end{equation}
where $\mathrm{p}$ is the mapping 
$$\mathrm{p}(u,b,p,s)^t = (u,b,p+\frac{1}{2}|b|^2,s)^t,\quad \mathrm{p}^{-1}(u,b,\p,s)^t = (u,b,\p-\frac{1}{2}|b|^2,s)^t.$$
Define
\begin{equation*}
\begin{aligned}
    Y(\lambda,T)= \{& z\in Y^m_*([0,T]\times \Omega): \partial_t^j z (0)= z_{(j)} \text{ for } j=0,\ldots, m-1 ;\, M^\flat_m(z,T)\le \lambda; \\
    &{\cal B}z = {\cal B}' z = 0 \text{ on }\partial\Omega\backslash\{q_1,\ldots, q_N\}; \, z([0,T]\times \Omega) \subset \mathcal{K}_\delta
    \},
\end{aligned}
\end{equation*}
where $\mathcal{K}_\delta = \{x\in \R^6: \dist (x,\mathcal{K}) \le \delta\} \subset\subset \R^4\times \mathcal{V}$ for some fixed small $\delta>0$.

The next goal is to prove that the approximate solution sequence $z^k(t,x)$ determined by \eqref{Picard} is bounded in 
$Y(\lambda,T)$ by choosing  $\lambda$ and $T$ properly. First, assume that $z^k \in Y(\lambda,T)$ by induction, since $\mathrm{p}(z_0)$ satisfies the compatibility conditions of \eqref{Picard} up to order $m-1$, by using Theorem~\ref{thm-LIBVP'}, there exists a unique solution $y^{k+1}\in \tilde{Y}^m_*([0,T_0]\times \Omega)$ to \eqref{Picard}, where $T_0$ depends only on $\Omega$ and $\mathcal{K}_\delta$. Thus, by taking $T$ properly small, we may assume $T_0 = T$, and $z^{k+1} = \mathrm{p}(y^{k+1})\in \tilde{Y}^m_*([0,T]\times\Omega)$. It is easy to check that $\partial_t^j z^{k+1}(0) = z_{(j)}$ for $j=0,\ldots, m-1$, and ${\cal B}z^{k+1} = {\cal B}' z^{k+1} = 0$ on $\partial\Omega\backslash\{q_1,\ldots,q_N\}$. 

Fix a constant $c_0 \ge \sum_{j=0}^{m-1} \normmm{z_{(j)}}_{H^{m-j}_*(\Omega)}$. We shall choose $\lambda$ sufficiently large and $T$ properly small so that $z^{k+1}\in Y(\lambda,T)$. By applying the estimate \eqref{apriori'-Linfty} in the problem \eqref{Picard}, one has
\begin{equation*}
\begin{aligned}
    M^\flat_m(y^{k+1},t) &\le \phi_1\left(\normmm{z^k(0)}_{m-1,*} + \int_0^T \normmm{z^k(\tau)}_{m,*} \d \tau\right)\cdot\\
    & \hspace{.2in} \left(
    \sum_{j=0}^{m-1} \|z_{(j)}\|_{H^{m-j}_*(\Omega)} + \phi_2 \left(M^\sharp_m(z^k, t)\right) \cdot t \cdot \normmm{z^{k+1}(0)}_{m-1,*}
    \right)\\
    & \le c_0 \phi_1(c_0 + \lambda T) \left(1+\phi_2(c_0+\lambda)\cdot T\right).
\end{aligned}
\end{equation*}
Therefore, we get
\begin{equation*}
\begin{aligned}
    M^\flat_m(z^{k+1},t) &
    \le c_1(1+M^\flat_m(y^{k+1},t)^2)\\
    &\le c_1\left(1 + c_0^2 \phi_1^2(c_0 + \lambda T) \left(1+\phi_2(c_0+\lambda)\cdot T\right)^2\right),
\end{aligned}
\end{equation*}
where $c_1$ is some fixed constant depending on the Moser-type estimate in $Y^m_*((0,T)\times \Omega)$,
and
\begin{equation*}
    \begin{aligned}
        \|z^{k+1} - z_0\|_{L^\infty((0,T)\times \Omega)} &\le {\operatorname{ess}\sup}_{t\in (0,T)} \int_0^t \|\partial_t z^{k+1}(\tau)\|_{L^\infty(\Omega)} \d \tau \\
        & \le c_2 T M^\flat_5(z^{k+1},T),
    \end{aligned}
\end{equation*}
where $c_2$ is some fixed constant depending on the Sobolev embedding theorem.
Then by choosing $\lambda, T$ properly such that
\begin{equation*}
    \lambda T\le 1, \quad \phi_2(c_0+ \lambda) T \le 1, \quad c_1(1+4c_0^2\phi_1^2(c_0 + 1)) \le \lambda,\quad c_2 T\lambda \le \delta,
\end{equation*}
one has $z^{k+1}\in Y(\lambda,T)$.

Thus, we obtain a sequence $\{z^k\}_{k\in \mathbb{N}}$ bounded in $Y(\lambda,T)$.
As usual,  by applying the estimate \eqref{apriori'-Linfty} again for the problem of $z^{k+1}-z^k$ derived from \eqref{Picard}, one can show that $z^k$ converges to a limit $z$ in $X^{m-2}_*([0,T]\times \Omega)$ as $k\to +\infty$, and $z\in Y_*^m((0, T)\times\Omega)$. Passing to the limit in \eqref{Picard}, we conclude that $z$ is a solution of the problem \eqref{MHD}-\eqref{bcd}. The uniqueness of this solution follows directly by using the estimate \eqref{apriori'-Linfty} again, and the proof of Theorem~\ref{main-thm} is complete.

\section{Appendix A: Proof of Proposition \ref{prop-lift-cpbt}}
\renewcommand{\thelemma}{A.\arabic{lemma}}
\setcounter{lemma}{0}
\renewcommand{\theprop}{A.\arabic{prop}}
\setcounter{prop}{0}
\renewcommand{\thecoro}{A.\arabic{coro}}
\setcounter{coro}{0}
\renewcommand{\theremark}{A.\arabic{remark}}
\setcounter{remark}{0}

The proof of Proposition~\ref{prop-lift-cpbt} follows in an argument similar to that given in \cite[Proposition 3.1]{godin20212d}. 
For completeness, we provide an outline here and highlight the differences that arise in our setting,
in particular, the additional complexity introduced by the magnetic field $b$ 
and the fact that we only require $\omega_n \in (0, \frac{\pi}{[\frac{m}{2}]})$ for each $n=1,\ldots, N$, instead of asking $\omega_n \in (0,\frac{\pi}{m})$.

We focus on the case where $\bar{x}$ in the statement of Proposition~\ref{prop-lift-cpbt} is a corner point, since the other cases -- when $\bar{x}$ is an interior point of $\Omega$ or lies on the smooth part of $\partial\Omega$ -- are much simpler. 
Let $\omega \in (0,\pi)$ be the original angle at a corner $\bar{x}$. Recall that in \eqref{angle-cdt}, we require
\begin{equation}\label{angle-cdt'}
    \omega \in \left(0, \frac{\pi}{[\frac{m}{2}]}\right) \backslash \left\{\frac{\pi}{[\frac{m}{2}]+1}, \frac{\pi}{[\frac{m}{2}]+2}, \ldots, \frac{\pi}{m}\right\}.
\end{equation}
By introducing a coordinate transformation as in the proof of Proposition~\ref{prop-weak=strong}, we may assume without loss of generality that $\bar{x} = 0$, the angle at $\bar{x}$ becomes $\frac{\pi}{2}$, and the IBVP \eqref{linear}-\eqref{icd-linear} in a neighborhood $\mathcal{U}^{\delta_0} = \{x\in \R^2, 0< x_1, x_2 < \delta_0\}$ of $\bar{x}$ is reduced to
\begin{equation}\label{qtr}
\begin{cases}
    \begin{aligned}
        &R(P,S) \big(\partial_t u + U \cdot \nabla u \big) + G^{-1}\nabla \p - B \cdot\nabla b \\
        &\qquad\qquad = F_1 + E(B)b - R(P,S) E(U) u,
    \end{aligned}\\
    \begin{aligned}
        \partial_t b + U \cdot \nabla b - B \cdot \nabla u + B (\nabla \cdot u) = F_2 - B(e \cdot u),
    \end{aligned}\\
    \begin{aligned}
        &\partial_t \p + U\cdot \nabla \p - (G B) \cdot \big(
        B\cdot \nabla u
        - B(\nabla\cdot u)\big) + Q(P,S) \nabla\cdot u\\
        &\qquad\qquad= F_3 + (GB)\cdot \big(E(U) b - B(e\cdot u)\big)- Q(P,S) e\cdot u,
    \end{aligned}\\
    \begin{aligned}
        \partial_t s + U \cdot \nabla s = F_4.
    \end{aligned}
\end{cases}
\end{equation}
in $(0,T)\times \mathcal{U}^{\delta_0} $, supplemented with
\begin{equation}\label{qtr-bcd}
    u_i|_{x_i=0} = 0, \quad\text{ for } i=1,2.
\end{equation}
and
\begin{equation}\label{qtr-icd}
\begin{aligned}
    & (u,b,\p, s)^t|_{t=0} = (u_0, b_0, \p_0, s_0)^t, \quad x\in \mathcal{U}^{\delta_0},\\ 
    &\text{where }(b_0)_i|_{x_i=0} = (F_2)_i|_{x_i=0} = 0, \quad i=1,2,
\end{aligned}
\end{equation}
where $G$, $E(\cdot)$, $e$ are a $2\times 2$ invertible matrix, a $2\times 2\times 2$ tensor and a $2$-dimensional vector, respectively, as introduced before or after \eqref{qtr-eq}. They are related to the Jacobian matrix of the coordinate transformation, and hence are essentially influenced by $\omega$.

Let $\mathcal{K}_1\subset \R^6$ be a compact set of $\R^4\times \mathcal{V}$ and such that $Z([0,T]\times \overline{\mathcal{U}^{\delta_0}})\subset \mathcal{K}_1$.
Set $\mathcal{U}_1 = \{0\}\times (0,\delta_0)$, $\mathcal{U}_2 = (0,\delta_0)\times \{0\}$.

Denote by 
$$z^{(j)}(x_1,x_2)=(\partial_t^jz)(0,x_1,x_2)$$
the value of the $j-th$ derivative in time of the solution to the problem \eqref{qtr}-\eqref{qtr-icd} for any fixed integer $j\ge 0$. The compatibility condition of order $j$ for the problem \eqref{qtr}-\eqref{qtr-icd} is that
$$z_1^{(j)}(0,x_2)=0,\quad z_2^{(j)}(x_1,0)=0.
$$ 
Obviously, it can be represented by the initial data $z|_{t=0}=z_0(x_1,x_2):=(u_0, b_0, \p_0, s_0)^t$ and its derivatives up to order $j$ in space variables from the equations \eqref{qtr}.
More precisely, at $x_i = 0$ ($i=1,2$), one has
\begin{enumerate}
    \item[(1)] $(u_i)^{(1)} = - \frac{1}{R(P,S)}J_i \p_0 + \mathcal{P}^i_0 z_0 + \tilde{\mathcal{P}}^i_0 F(0)$;
    \item[(2)] $(u_i)^{(2l)} = \frac{1}{(\partial_p R)^l (P,S)} J_i K(Z) H^{l-1}(Z) u_0 + \mathcal{P}^i_{2l-1} z_0 + \tilde{\mathcal{P}}^i_{2l-1} F(0)$, as $2l\le m-1$;
    \item[(3)] $(u_i)^{(2l+1)} = \frac{1}{((\partial_p R)^l R)(P,S)} J_i K(Z) H^{l-1}(Z) (B\cdot \nabla b_0 - G^{-1} \nabla \p_0) + \mathcal{P}^i_{2l} z_0 + \tilde{\mathcal{P}}^i_{2l} F(0)$, as $2l+1\le m-1$,
\end{enumerate}
where $f(0)$ denotes the value at $t=0$ of a related function, $J_1$ and $J_2$ are two first-order differential operators such that $(J_1, J_2)^t = G^{-1} \nabla$, $K(Z)$ and $H(Z)$ are differential operators defined by
    $$
    \begin{array}{l}
        K(Z)v = \nabla\cdot v + {Q(P,S)}^{-1} (GB)\cdot \big(B(\nabla \cdot v) - B\cdot \nabla v\big),
        \\
        H(Z)v = {Q(P,S)}^{-1} (B\cdot\nabla)\big(B(\nabla \cdot v) - B\cdot \nabla v\big) + G^{-1}\nabla \big(K(Z)v\big),
    \end{array}
    $$
 for any $C^1$ vector field $v(x)=(v_1(x),v_2(x))^t$, and $\mathcal{P}^i_q$, $\tilde{\mathcal{P}}^i_q$ represent classes of differential operators of order $q$, which may differ from line to line. In each instance, $\mathcal{P}^i_q z$ refers to a sum of terms either being $\Theta(x_{i'}, Z)\partial_x^{\beta_n}{z}_n$ with $|\beta_n|\le q$ or $\Theta(x_{i'}, Z)\big(\Pi_{\gamma,\sigma}\partial^{\alpha_\sigma^\gamma} Z_\sigma\big)\partial_x^{\beta_n}{z}_n$ with $\alpha_\sigma^\gamma\ne 0$, $\sum_{\gamma,\sigma} |\alpha_\sigma^\gamma|+|\beta_n|\le q+1$, $\sum_{\gamma,\sigma} |\alpha_\sigma^\gamma|,|\beta_n|\le q$, $\Theta \in C^\infty([0,\delta_0]\times \mathcal{K}_1)$, while
$\tilde{\mathcal{P}}^i_q F$ refers to a sum of terms either being $\Theta(x_{i'}, Z)\partial^{\beta_n} F_n$ with $|\beta_n|\le q$, or $\Theta(x_{i'}, Z)(\Pi_{\gamma,\sigma} \partial^{\alpha_\sigma^\gamma} Z_\sigma)\partial^{\beta_n}F_n$ with $\alpha_\sigma^\gamma \ne 0, \sum_{\gamma,\sigma} |\alpha_\sigma^\gamma|+|\beta_n|\le q$, $\Theta \in C^\infty([0,\delta_0]\times \mathcal{K}_1)$.

Therefore, the zeroth-order compatibility condition reads
        \begin{equation}\label{cptb-low-0}
                (u_0)_i|_{x_i=0} = 0, \quad (i=1,2),
        \end{equation}
the first order compatibility condition is equivalent to
        \begin{equation}\label{cptb-low-1}
                J_i \p_0 = \mathcal{P}^i_{0}{z}_0 + \tilde{\mathcal{P}}^i_0{F}(0), \quad {\rm on}~x_i=0~(i=1,2), 
        \end{equation}
and for $l\in \mathbb{N}$ and $2l\le m-1$, the $(2l)$-th compatibility conditions are equivalent to
        \begin{equation}\label{cptb-even}
                J_i K(Z) H^{l-1}(Z) u_0 = \mathcal{P}^i_{2l-1}{z}_0 + \tilde{\mathcal{P}}^i_{2l-1}{F}(0), \quad {\rm on}~x_i=0~(i=1,2), 
        \end{equation}
while for $l\in \mathbb{N}$ and $2l + 1\le m-1$, the $(2l+1)$-th compatibility conditions are equivalent to
        \begin{equation}\label{cptb-odd}
            J_i K(Z) H^{l-1}(Z) (G^{-1}\nabla \p_0) = J_i K(Z) H^{l-1}(Z) (B\cdot \nabla b_0)+\mathcal{P}^i_{2l}{z}_0 + \tilde{\mathcal{P}}^i_{2l}{F}(0),  ~{\rm on}~x_i=0~(i=1,2).
        \end{equation}
\begin{remark}\label{rmk-PF-struc}
    For later use, we point out that for fixed $Z$, $\tilde{\mathcal{P}}^i_0 F = \frac{1}{R} (F_1)_i$;
        $\tilde{\mathcal{P}}^i_{2l}{F}$ ($l\ge 1$) is a sum of terms of the form:
        \begin{equation*}
            \tilde{P}_{2l-1} F, \quad \tilde{P}_{2l-1}\partial_t F, \quad \tilde{P}_{2l-1}(\partial_1 F_1 + \partial_2 F_2), \quad a(x_{i'})\tilde{P}_{2l} F,
        \end{equation*}
        where $\tilde{P}_\lambda$ is of the form $\sum_{|\alpha|\le \lambda} c_{\alpha}(x_{i'}) \partial^\alpha$ with $c_\alpha\in C^\infty([0,\delta_0])$, whereas $a\in C^\infty([0,\delta])$ and $a(0) = 0$.
\end{remark}

The localized version in the neighborhood $\mathcal{U}^\delta$ of Proposition~\ref{prop-lift-cpbt} can be stated as the following one.

\begin{prop}\label{prop-lift-local}
    Assume that $m\ge 3$ is an integer, the angle at the origin of $\partial\Omega$ satisfies $\omega \in (0,\frac{\pi}{[\frac{m}{2}]})$, and assume also that $z_0\in H^{m}({\mathcal U}^{\delta_0}), F\in H^{m}((0,T)\times {\mathcal U}^{\delta_0})$ satisfy the compatibility conditions up to order $m-1$ of the problem \eqref{qtr}-\eqref{qtr-bcd}.
    Then there exists $0<\delta_1\le \delta_0$, depending on $\mathcal{K}$ and the geometry of $\Omega$, such that for any $0<\delta\le \delta_1$, one can find sequences $z_0^k \in H^{m+\frac{5}{2}}(\mathcal{U}^\delta)$, $F^k \in C^\infty([0,T]\times \overline{\mathcal{U}^\delta})$ with ${\mathcal U}^\delta=\{x\in \R^2: 0<x_1, x_2<\delta\}$, such that the following holds:
    \begin{enumerate}
        \item[(1)] $z_0^k, F^k$ satisfies the compatibility conditions up to order $m+1$ for the problem \eqref{qtr}-\eqref{qtr-bcd} with $(b_0^k)_i = (F^k_2)_i = 0$ on $x_i = 0$ for $i=1,2$;
        \item[(2)] $z_0^k \to z_0$ in $H^{m}(\mathcal{U}^\delta)$ and $F^k \to F$ in $H^{m}((0,T)\times \mathcal{U}^\delta)$ as $k\to +\infty$.
    \end{enumerate}
\end{prop}

Once we have the above proposition, to conclude the results given in Proposition \ref{prop-lift-cpbt}, it suffices to localize the data as stated in the next lemma.
\begin{lemma}\label{lemma-local-reduce}
    For any given $z_0 \in H^m(\mathcal{U}^\delta)$, $F\in H^m((0,T)\times \mathcal{U}^\delta)$ with $m\ge 1$, satisfying the compatibility conditions  of \eqref{qtr}, \eqref{qtr-bcd} up to order $m-1$  and $((b_0)_i, (F_2)_i)|_{t=x_i=0} = 0$ ($i=1,2$), then for any fixed $\epsilon\in (0, \frac{\delta}{4})$, there exists $\hat{z}_0 \in H^m(\mathcal{U}^\delta)$, $\hat{F}\in H^m((0,T)\times \mathcal{U}^\delta)$, satisfying the compatibility conditions of \eqref{qtr}, \eqref{qtr-bcd} up to order $m-1$, and
    \begin{enumerate}
        \item[(1)] $\hat{z}_0(x) = z_0(x)$ when $|x|\le \varepsilon$, and vanishes when $\max\{x_1,x_2\} \ge \delta - \varepsilon$,
        \item[(2)] $\hat{F}(t,x) = \varphi(x) F(t,x)$ with a smooth function $\varphi$ satisfying that $\varphi(x) = 1$ when $|x|\le \varepsilon$, and vanishes when $\max\{x_1,x_2\} \ge \delta - \varepsilon$.
       
    \end{enumerate}
\end{lemma}
\begin{proof}
    First construct $z_0^I\in H^m(\mathcal{U}^\delta), F^I \in H^m((0,T)\times \mathcal{U}^\delta)$, such that $(z_0^I,F^I)$ equals to $(z_0, F)$ when $x_2 \le 3\varepsilon$, vanishes when $x_2 \ge \delta - \varepsilon$, and they satisfy the compatibility conditions of \eqref{qtr}, \eqref{qtr-bcd} on $\mathcal{U}_1$ up to order $m-1$. 

    Note that the compatibility conditions on $\mathcal{U}_1$ given in the formula \eqref{cptb-low-0}-\eqref{cptb-odd} can be rewritten as
    \begin{equation*}
        (u_0)_1 = 0, \quad \partial_1 \p_0 = \mathcal{E}_1\Big((\partial_2^{\alpha} \p_0)_{|\alpha|\le 1}, u_0, b_0, F(0)\Big);
    \end{equation*}
    for $l\in \mathbb{N}$ and $2l\le m-1$,
    \begin{equation*}
        \partial_1^{2l} (u_0)_1 = \mathcal{E}_{2l}\Big((\partial_x^{\alpha} \p_0)_{|\alpha|\le 2l-1}, (\partial_x^\beta u_0)_{|\beta|\le 2l, \beta_1 \le 2l-1}, (\partial_x^\gamma b_0)_{|\gamma|\le 2l-1}, (\partial^\sigma F(0))_{|\sigma|\le 2l-1}\Big);
    \end{equation*}
    and for $l\in \mathbb{N}$ and $2l+1\le m-1$,
    \begin{equation*}
        \partial_1^{2l+1} \p_0 = \mathcal{E}_{2l+1}\Big((\partial_x^{\alpha} \p_0)_{|\alpha|\le 2l+1, \alpha_1 \le 2l}, (\partial_x^\beta u_0)_{|\beta|\le 2l}, (\partial_x^\gamma b_0)_{|\gamma|\le 2l+1}, (\partial^\sigma F(0))_{|\sigma|\le 2l}\Big),
    \end{equation*}
    where $\mathcal{E}_j$ are linear in their arguments with $C^\infty(\overline{\mathcal {U}^\delta})$ coefficients.
    
    Choose $\varphi\in C^\infty_0(\overline{\R^+})$ such that $\varphi(x_2)= 1$ as $x_2 \le \delta_1$, $\varphi(x_2) = 0$ as $x_2 \ge \delta_2$, for certain $\delta_1, \delta_2$ satisfying $3\varepsilon< \delta_1< \delta_2 < \delta - \varepsilon$.
    Inductively define $\bar{z}^j(x_2) = (\bar{u}^j, \bar{b}^j, \bar{\p}^j, \bar{s}^j)^t(x_2)$ with $j=0,\ldots, m-1$ by:
    \begin{equation*}
        \bar{z}^0(x_2) = \varphi(x_2) z_0(0,x_2);
    \end{equation*}
    \begin{equation*}
        \begin{cases}
            (\bar{u}^1, \bar{b}^1, \bar{s}^1)(x_2) = \varphi(x_2) \partial_1 (u_0, b_0, s_0) (0,x_2),\\
            \bar{\p}^1(x_2) = \mathcal{E}_1\Big((\partial_2^{\alpha} 
            \bar{\p}^0)_{|\alpha|\le 1}, \bar{u}^0, \bar{b}^0, \varphi(x_2)F(0,0,x_2)\Big);
        \end{cases}
    \end{equation*}
    for $l\in \mathbb{N}$ and $2l\le m-1$,
    \begin{equation*}
        \begin{cases}
            (\bar{u}_2^{2l}, \bar{b}^{2l}, \bar{\p}^{2l}, \bar{s}^{2l})(x_2) = \varphi(x_2) \partial_1^{2l} ((u_0)_2, b_0, \p_0, s_0) (0,x_2),\\
            \begin{aligned}
                \bar{u}^{2l}_1(x_2) = \mathcal{E}_{2l}\Big(
                &(\partial_{x_2}^{\alpha_2} \bar{\p}^{\alpha_1})_{|\alpha|\le 2l-1}, (\partial_{x_2}^{\beta_2} \bar{u}^{\beta_1})_{|\beta|\le 2l, \beta_1 \le 2l-1},\\
                &(\partial_{x_2}^{\gamma_2} \bar{b}^{\gamma_1})_{|\gamma|\le 2l-1}, (\partial^\sigma (\varphi F(0,0,x_2)))_{|\sigma|\le 2l-1}\Big);
            \end{aligned}
        \end{cases}
    \end{equation*}
    for $l\in \mathbb{N}$ and $2l+1\le m-1$,
    \begin{equation*}
        \begin{cases}
            (\bar{u}^{2l+1}, \bar{b}^{2l+1}, \bar{s}^{2l+1})(x_2) = \varphi(x_2) \partial^{2l+1}_1 (u_0, b_0, s_0) (0,x_2),\\
            \begin{aligned}
                \partial_1^{2l+1} \p_0(x_2) = \mathcal{E}_{2l+1}\Big(&(\partial_{x_2}^{\alpha_2} \bar{\p}^{\alpha_1})_{|\alpha|\le 2l+1, \alpha_1 \le 2l}, (\partial_{x_2}^{\beta_2} \bar{u}^{\beta_1})_{|\beta|\le 2l}, \\&(\partial_{x_2}^{\gamma_2} \bar{b}^{\gamma_1})_{|\gamma|\le 2l+1}, (\partial^\sigma (\varphi F(0,0,x_2)))_{|\sigma|\le 2l}\Big).
            \end{aligned}
        \end{cases}
    \end{equation*}
    Let $w_j(x_2) = \bar{z}^j(x_2) - 
    \partial_1^j
    (\varphi z_0)(0,x_2)$ for $j=0,\ldots, m-1$, then 
    $w_0\equiv 0$, and it can be shown by induction that for all $1\le j \le m-1$, there are at most $(j-1)$-th order derivatives falling on $z_0$ in the expression of $w_j$, thus
    $w_j \in H^{m-j+\frac{1}{2}}
    (0,\delta)$,
    also $w_j(x_2) = 0$ if $x_2 \not\in (\delta_1, \delta_2)$. Now, by using the classical trace lifting in Sobolev spaces, one can find 
    $\tilde{w} \in H^{m+1}(\mathcal{U}^\delta)$
    such that $\partial_1^j\tilde{w} = w_j$ for $j=0,\ldots, m-1$. Let $\psi \in C^\infty_0(\overline{\R^+})$ such that $\psi (x_2) = 1$ when $x_2 \in (\delta_1, \delta_2)$, $\psi (x_2) = 0$ when $x_2\in (0, 3\varepsilon)\cup (\delta- \varepsilon,\delta)$. Set $w(x) = \psi(x_2) \tilde{w}(x)$. Then $\supp_{x_2}w(x_1, x_2)
    \subset (3\varepsilon,\delta - \varepsilon)$
    and $\partial_1^j w(0,x_2) = w_j(x_2)$ for $j=0,\ldots, m-1$. Define
    $$z_0^I(x) = \varphi(x_2) z_0(x) + w(x), \quad
    F^I(t,x) = \varphi(x_2) F(t,x).$$ 
    Then, it is easy to see that $z_0^I, F^I$ satisfy the compatibility conditions of \eqref{qtr}, \eqref{qtr-bcd}  up to order $m-1$ on $\mathcal{U}_1$.

    Next, one can  construct $z_0^{I\!\!I}\in H^m(\mathcal{U}^\delta), F^{I\!\!I} \in H^m((0,T)\times \mathcal{U}^\delta)$ in the same way as above, such that $(z_0^{I\!\!I},F^{I\!\!I})$ equals to $(z_0, F)$ when $x_1 \le 3\varepsilon$, vanishes when $x_1 \ge \delta - \varepsilon$, and satisfies the compatibility conditions of \eqref{qtr}, \eqref{qtr-bcd}  up to order $m-1$ on $\mathcal{U}_2$. Note that $z_0^I = z_0^{I\!\!I} = z_0$ when $x\in (0,3\varepsilon)^2$.

    Define $\tilde{z}_0$ on $\mathcal{U}^\delta$ by: 
    \begin{equation*}
        \tilde{z}_0(x) = \begin{cases}
            z_0^I(x) & \text{if } x_1\in (0,3\varepsilon), x_2 \in (0,\delta),\\
            z_0^{I\!\!I}(x) & \text{if } x_2\in (0,3\varepsilon), x_1 \in (0,\delta),\\
            0 & \text{if } x\in (3\varepsilon,\delta)^2.
        \end{cases}
    \end{equation*}
    Choose $\chi\in C^\infty_0(\overline{\R^+})$ such that $\chi(\xi) = 1$ when $\xi\le 1$, and $\chi(\xi) = 0$ when $\xi \ge 2$. Finally, by setting $\hat{z}_0(x) = \chi(\frac{x_1 x_2}{\varepsilon}) \tilde{z}_0(x)$ and $\hat{F}(t,x) = \varphi(x_1)\varphi(x_2)F(t,x)$, one concludes the results given in Lemma \ref{lemma-local-reduce}.
\end{proof}

Now we proceed to prove Proposition \ref{prop-lift-local}.
To shorten the notations, we shall  write $z_0 = (u_1,u_2,b_1,b_2,\p,s)^t$ (instead of $((u_0)_1, (u_0)_2, (b_0)_1, (b_0)_2, \p_0, s_0 )^t$), $z^k_0 = (u^k_1,u^k_2,b^k_1,b^k_2,\p^k, s^k)^t$ and omit the index $\delta$ of the domain $\mathcal U$ throughout the proof. 

We shall first choose smooth sequences $v^k = (v^k_1, v^k_2)^t, b^k = (b^k_1, b^k_2)^t, q^k,s^k$ and $F^k = (F_1^k,\ldots, F_4^k)^t$, such that $(v^k, b^k, q^k, s^k)\to (u,b,\p,s)$ in $H^{m}(\mathcal{U})$ and $F^k \to F$ in $H^{m}((0, T)\times \mathcal{U})$ as $k\to +\infty$, and $v^k_i = b^k_i = (F^k_2)_{i} = 0$ at $x_i = 0, i=1,2$. 
Next, seek correctors $w^k=(w^k_1,w^k_2)^t, r^k \in H^{m+\frac{5}{2}}(\mathcal{U})$,
satisfying $(w^k, r^k) \to 0$ in $H^m(\mathcal{U})$ as $k\to +\infty$, and by setting $u^k = v^k - w^k, \p^k = q^k - r^k$, such that $z^k = (u^k, b^k, \p^k, s^k)^t$ and $F^k$ will satisfy all the properties given in Proposition~\ref{prop-lift-local}.

Denote by
\begin{equation*}
    \begin{aligned}
        & \alpha_j^k=\partial_1^j r^k(0, \cdot), \quad
        \xi_j^k=\partial_1^j w_1^k(0, \cdot), \quad
        \tau_j^k=\partial_1^j w_2^k(0, \cdot),
        \\
        & \beta_j^k=\partial_2^j r^k(\cdot, 0), \quad
        \eta_j^k=\partial_2^j w_1^k(\cdot, 0), \quad
        \theta_j^k=\partial_2^j w_2^k(\cdot, 0),
    \end{aligned}
\end{equation*}
and
\begin{equation*}
    \mathrm{z}_{1 j}^k(0, \cdot)=\left(\alpha_j^k, \xi_j^k, \tau_j^k \right),\quad \mathrm{z}_{2 j}^k(\cdot, 0)=\left(\beta_j^k, \eta_j^k, \theta_j^k \right).
\end{equation*}

The proof of Proposition \ref{prop-lift-local} will be based on the construction of 
$\mathrm{z}_{1 j}^k \in C^\infty(\overline{\mathcal{U}}_1),
 \mathrm{z}_{2 j}^k \in C^\infty(\overline{\mathcal{U}}_2)$, satisfying the following properties:
\begin{enumerate}
    \item[(1)] \begin{equation}\label{zkj-1}
    \partial_2^i \mathrm{z}_{1 j}^k (0,0) = \partial_1^j \mathrm{z}_{2 i}^k (0,0) \quad \text{for } i+j \le m+1.
\end{equation}
\item[(2)] There exists $\delta>0$, such that for any $j\le m-1$, \begin{equation}\label{zkj-2}
    \begin{aligned}
    &\mathrm{z}_{1 j}^k \to 0 \text{ in } H^{m-j-\frac{1}{2}}(\mathcal{U}_1),\quad 
        \mathrm{z}^k_{2 j} \to 0 \text{   in } H^{m-j-\frac{1}{2}}(\mathcal{U}_2),\quad \text{ as } k\to +\infty; \\
        & \int_0^\delta |\partial_2^i \mathrm{z}^k_{1 j}(0,y)-\partial_1^j \mathrm{z}^k_{2 i}(y,0)|^2 y^{-1} \d y \to 0  \text{ as }k\to +\infty, \text{ when }i+j = m-1.
    \end{aligned}
\end{equation}
\item[(3)] \begin{equation}\label{zkj-3}
\begin{aligned}
    & (u^k,b^k,\p^k,s^k) \text{ and } F^k
    \text{ satisfy the compatibility conditions }\\
    & \text{ of \eqref{qtr}  and \eqref{qtr-bcd} up to order } m+1  \text{ when } t=0.
\end{aligned}
\end{equation}
\end{enumerate}

\begin{lemma}\label{lemma-zkj}
    There exists $\delta_0>0$ depending on the matrix $G$ and the compact set $\mathcal{K}$, such that for any $0<\delta \le \delta_0$, one can find $\mathrm{z}_{1 j}^k \in C^\infty(\overline{\mathcal{U}}_1), \mathrm{z}_{2 j}^k \in C^\infty(\overline{\mathcal{U}}_2)$ such that \eqref{zkj-1}-\eqref{zkj-3} hold.
\end{lemma}

One can refer to \cite[p.19]{godin20212d} for the proof of Proposition~\ref{prop-lift-local} from Lemma~\ref{lemma-zkj}.

The proof of Lemma~\ref{lemma-zkj} proceeds by three main steps: first,  determine the values $\partial_2^j\mathrm{z}^k_{1i}(0,0) = \partial_1^i\mathrm{z}^k_{2j}(0,0)$
at the corner point with $i+j\le m+1$ from the compatibility conditions restricted at $x=0$; next,  construct $\mathrm{z}^k_{1j}, \mathrm{z}^k_{2j}$ for $0\le j\le m-1$ on the edges, such that \eqref{zkj-2} and the compatibility conditions up to order $m-1$ are satisfied; finally, choose $\mathrm{z}^k_{1j}, \mathrm{z}^k_{2j}$ for $m\le j\le m+1$ properly such that the compatibility conditions up to order $m+1$ are satisfied.

\paragraph{Step~1. Construction of $\bar{\sigma}^k_{ij}$.}
Set 
$$
\bar{\sigma}^k_{ij} = 
\begin{cases}
\partial_2^j \sigma_i^k(0), \quad {\rm as}~\sigma = \alpha,\, \xi,\,\tau,\\[2mm]
\partial_1^j \sigma_i^k(0), \quad 
{\rm as}~\sigma = \beta, \, \eta, \, \theta.
\end{cases}
$$
Then \eqref{zkj-1} means that
\begin{equation*}
    \bar{\alpha}^k_{ij} = \bar{\beta}^k_{ji}, \quad \bar{\xi}^k_{ij} = \bar{\eta}^k_{ji},\quad \bar{\tau}^k_{ij} = \bar{\theta}^k_{ji},\qquad \text{when } i+j \le m+1,
\end{equation*}
and \eqref{zkj-3} also implies some additional relations which must be satisfied by the $\bar{\sigma}^k_{ij}$.
The next corollary elaborates some relations derived from \eqref{cptb-low-0}-\eqref{cptb-odd} on $\mathcal{U}_h$, $h=1,2$, which would further imply a linear system that must be satisfied by $\bar{\sigma}^k_{ij}$.

\begin{coro}\label{coro-corner-cpbt}
    If $z^k \in X^{m+2}_*([0,T]\times \Omega)$ is a solution of \eqref{qtr}-\eqref{qtr-icd} with $F$ being replaced by $F^k$, and $(u_0,b_0,\p_0,s_0)$ replaced by $(v^k+w^k, b^k, q^k + r^k, s^k)$, then for all $0\le n \le m+1$, $h,h' \in \{1,2\}$ and $h\ne h'$, 
    the following relations hold when $t=0$ and $x\in \mathcal{U}_h$:
    \begin{enumerate}
        \item[(1)] \begin{equation}\label{csys-low}
                    \partial_{h'}^n w^{k}_h = 0,
                    \end{equation}
        \item[(2)] if $n\ge 1$, for $0\le l\le [\frac{n-1}{2}]$,
        \begin{equation}\label{csys-odd}
            \hat{\mathcal{L}}^h_{n-2l-1,2l+1} r^k = \mathcal{G}^{h,k}_{n,l} - \partial_{h'}^{n-2l-1} \tilde{\mathcal{P}}^h_{2l} F^k + \big(\hat{\mathcal{L}}^h_{n-2l-1,2l+1}-\mathcal{L}^h_{n-2l-1,2l+1}\big)r^k,
        \end{equation}
        \item[(3)] if $n\ge 2$, for $1\le l\le [\frac{n}{2}]$,
        \begin{equation}\label{csys-even}
            \hat{\mathcal{L}}^h_{n-2l,2l-1} (\nabla\cdot w^k) = \mathcal{H}^{h,k}_{n,l} - \partial_{h'}^{n-2l} \tilde{\mathcal{P}}^h_{2l-1} F^k + \big(\hat{\mathcal{L}}^h_{n-2l,2l-1}(\nabla\cdot w^k)-\mathcal{L}^h_{n-2l,2l}(w^k)\big),
        \end{equation}
    \end{enumerate}
    where 
    \begin{equation*}
        \hat{\mathcal{L}}^h_{n-2l-1,2l+1} = \partial_{h'}^{n-2l-1} \hat{J}_h (\partial_1 \hat{J}_1 + \partial_2 \hat{J}_2)^{l},\quad
        \hat{\mathcal{L}}^h_{n-2l-1,2l-1} = \partial_{h'}^{n-2l-1} \hat{J}_h (\partial_1 \hat{J}_1 + \partial_2 \hat{J}_2)^{l-1},
    \end{equation*}
    $\hat{J}_j$ denotes that the coefficients of $J_j$ are frozen at $x=0$; 
    \begin{equation*}
        \begin{aligned}
        &\mathcal{L}^h_{n-2l-1,2l+1} = 
        \begin{cases}
            \partial_{h'}^{n-1} J_h, & \text{if }l=0,\\
            \partial_{h'}^{n-2l-1} \big(J_h K(Z) H^{l-1}(Z) (G^{-1}\nabla (\cdot))\big), &\text{if } l\ge 1,
        \end{cases}
        \\&\mathcal{L}^h_{n-2l,2l} = \partial_{h'}^{n-2l} \big(J_h K(Z) H^{l-1}(Z) (\cdot)\big),
        \end{aligned}
    \end{equation*}
and
\begin{equation*}
    \mathcal{G}^{h,k}_{n,l}=\begin{cases}
        \begin{aligned}
            \partial_{h'}^{n-1} \big( J_h q^k - \mathcal{P}^h_0(v^k-w^k,b^k,q^k-r^k) \big),
        \end{aligned} & \text{if }l=0,
        \\
        \begin{aligned}
            \partial_{h'}^{n-2l-1}\Big(J_h K(Z) H^{l-1}(Z)(G^{-1}\nabla q^k + B\cdot \nabla b^k)&\\
            - \mathcal{P}^h_{2l}(v^k-w^k,b^k,q^k-r^k)& \Big),
        \end{aligned}& \text{if } l\ge 1,
    \end{cases}
\end{equation*}
\begin{equation*}
    \begin{aligned}
        \mathcal{H}^{h,k}_{n,l} = \partial_{h'}^{n-2l} \big(J_h K(Z) H^{l-1}(Z)v^k
        - \mathcal{P}^h_{2l-1}(v^k-w^k,b^k,q^k-r^k)\big),
    \end{aligned}
\end{equation*}
with other notations such as $J_j, {\cal P}^h_{l}, \tilde{\cal P}^h_{l}, K(Z), H(Z)$ being the same as given in
\eqref{cptb-low-0}-\eqref{cptb-odd}.

\end{coro}
We see that the left-hand sides of \eqref{csys-odd} and \eqref{csys-even} consist of homogeneous spatial derivatives of order $n$, while the right-hand sides consist of three types of terms:  
\begin{itemize}
    \item contributions from $v^k,b^k,q^k,s^k, F^k$;  
    \item $n$-th order spatial derivatives acting on $w^k$ and $r^k$, whose coefficients vanish at $x=0$;  
    \item lower-order derivative terms of $w^k$ and $r^k$.
\end{itemize}

Therefore, if $n\ge 1$, we restrict \eqref{csys-odd} for $h=1,2$, $0\le l\le [\frac{n-1}{2}]$ at $x=0$, and obtain a system of $2[\frac{n+1}{2}]$ equations for the components
\begin{equation*}
    \bar{\alpha}^k_n:= (\bar{\alpha}^k_{0,n},\bar{\alpha}^k_{1,n-1},\ldots,\bar{\alpha}^k_{n,0})^t,
\end{equation*}
which shall be written as
\begin{equation}\label{lnsys-1}
    B_n \bar{\alpha}^k_n = \mathcal{G}^k_n(0)- \mathcal{F}^k_n(0),
\end{equation}
where $B_n$ is a $(2[\frac{n+1}{2}])\times (n+1)$ real matrix defined by the coefficients of $\hat{J}_1, \hat{J}_2$, and thus depends on the matrix $G$ from the coordinate transformation;
$\mathcal{G}^k_n$ is the $2[\frac{n+1}{2}]$-dimensional vector whose $(j+1)$-th entry is 
\begin{equation*}
    \mathcal{G}^{1,k}_{n,j} + \big(\hat{\mathcal{L}}^1_{n-2j-1,2j+1}-\mathcal{L}^1_{n-2j-1,2j+1}\big)r^k,
\end{equation*} 
and $(j+1+[\frac{n+1}{2}])$-th entry is 
\begin{equation*}
    \mathcal{G}^{2,k}_{n,j} + \big(\hat{\mathcal{L}}^2_{n-2j-1,2j+1}-\mathcal{L}^2_{n-2j-1,2j+1}\big)r^k,
\end{equation*}
with $0\le j\le [\frac{n-1}{2}]$, $\mathcal{G}^k_n(0)$ is its value at $x=0$;
\begin{equation*}
    \begin{aligned}
        \mathcal{F}^k_n = \Big( &
    \partial_2^{n-1}\tilde{\mathcal{P}}^1_{0} F^k, \partial_2^{n-3}\tilde{\mathcal{P}}^1_{2} F^k, \ldots, \partial_2^{n-2[\frac{n-1}{2}]-1}\tilde{\mathcal{P}}^1_{2[\frac{n-1}{2}]} F^k, \\
    &\partial_1^{n-1}\tilde{\mathcal{P}}^2_{0} F^k, \partial_1^{n-3}\tilde{\mathcal{P}}^2_{2} F^k,
    \ldots, \partial_1^{n-2[\frac{n-1}{2}]-1}\tilde{\mathcal{P}}^2_{2[\frac{n-1}{2}]} F^k 
    \Big)^t,
    \end{aligned}
\end{equation*}
and $\mathcal{F}^k_n(0)$ is its value at $x=0$.

Similarly, at $x=0$, \eqref{csys-even} for $h=1,2$, $1\le l\le [\frac{n}{2}]$, is a system of $2[\frac{n}{2}]$ equations for the components 
\begin{equation*}
    \begin{aligned}
        &\bar{\lambda}^k_{n-1}:= (\bar{\lambda}^k_{0,n-1},\ldots,\bar{\lambda}^k_{n-1,0})^t \quad \text{ where}\\
        & \bar{\lambda}^k_{i,j} = \partial_1^i\partial_2^j(\nabla\cdot w^k)(0,0) = \bar{\xi}^k_{i+1,j} + \bar{\tau}^k_{i,j+1},
    \end{aligned}
\end{equation*}
which can be written as
\begin{equation}\label{lnsys-2}
    B_{n-1} \bar{\lambda}^k_{n-1} = \mathcal{H}^k_n(0) - \tilde{\mathcal{F}}^k_n(0),
\end{equation}
where $B_{n-1}$ is exactly the matrix $B_n$ in \eqref{lnsys-1} with $n$ replaced by $n-1$, and $\mathcal{H}^k_{n}$ is the $2[\frac{n}{2}]$-dimensional vector whose $j$-th entry is 
\begin{equation*}
    \mathcal{H}^{1,k}_{n,j} + \big(\hat{\mathcal{L}}^1_{n-2j,2j-1}(\nabla\cdot w^k)-\mathcal{L}^1_{n-2j,2j}(w^k)\big),
\end{equation*}
and $(j+[\frac{n}{2}])$-th entry is 
\begin{equation*}
    \mathcal{H}^{2,k}_{n,j} + \big(\hat{\mathcal{L}}^2_{n-2j,2j-1}(\nabla\cdot w^k)-\mathcal{L}^2_{n-2j,2j}(w^k)\big),
\end{equation*}
with $1\le j\le [\frac{n}{2}]$, $\mathcal{H}^k_n(0)$ is its value at $x=0$;
\begin{equation*}
    \begin{aligned}
        \tilde{\mathcal{F}}^k_n = \Big( &
    \partial_2^{n-2}\tilde{\mathcal{P}}^1_{1} F^k, \partial_2^{n-4}\tilde{\mathcal{P}}^1_{3} F^k, \ldots, \partial_2^{n-2[\frac{n}{2}]}\tilde{\mathcal{P}}^1_{2[\frac{n}{2}]-1} F^k, \\
    & \partial_1^{n-2}\tilde{\mathcal{P}}^2_{1} F^k, \partial_1^{n-4}\tilde{\mathcal{P}}^2_{3} F^k, \ldots, \partial_1^{n-2[\frac{n}{2}]}\tilde{\mathcal{P}}^2_{2[\frac{n}{2}]-1} F^k
    \Big)^t,
    \end{aligned}
\end{equation*}
and $\tilde{\mathcal{F}}^k_n(0)$ is its value at $x=0$.

Now we claim the following lemma.
\begin{lemma}\label{lemma-solvable}
    Assume that $m\in \mathbb{N}\backslash\{0\}$, and that $\omega \in (0, \frac{\pi}{[\frac{m}{2}]})$.
    For any $n=0,\ldots, m+1$,
    one can find $\bar{\sigma}^k_{ij}$ for all $i+j \le n$ satisfying that
    \begin{equation}\label{sigma-cdt}
    \eqref{csys-low}-\eqref{csys-even} \text{ hold at } x=0; \quad  \bar{\sigma}^k_{ij} \to 0 \text{ as } k\to +\infty \quad (i+j\le m-2).
\end{equation}
\end{lemma}
\begin{proof}
    When $n=0$, to satisfy \eqref{csys-low}, we simply set $\bar{\xi}^k_{00} = \bar{\eta}^k_{00}=\bar{\tau}^k_{00} = \bar{\theta}^k_{00}=0$, and
    $\bar{\alpha}^k_{00} = \bar{\beta}^k_{00}$, with $\bar{\alpha}^k_{00}\to 0$ as $k\to +\infty$.
    
Now suppose that for some $1\le n\le m+1$, the real numbers $\bar{\sigma}^k_{ij}$ with $i+j\le n-1$ have been determined so that \eqref{sigma-cdt} holds for $i+j\le n-1$,
we are going to solve the linear systems \eqref{lnsys-1} and \eqref{lnsys-2} to determine $\bar{\sigma}^k_{ij}$ for $i+j = n$. Now we cite a lemma from \cite{godin20212d}, which provides the information to solve \eqref{lnsys-1} and \eqref{lnsys-2}. 

\begin{lemma}(\cite[Lemma 3.5]{godin20212d})\label{lemma-B-surj}
Assume that $m\ge 2$ is an integer, and each angle $\omega$ on the boundary $\partial\Omega$ satisfies the condition \eqref{angle-cdt'}.
    \begin{enumerate}
    \item[(1)] If $\omega \ne \frac{\pi}{m+1}$, then for $1\le n \le m+1$, $B_n$ is surjective.
    \item[(2)] If $\omega = \frac{\pi}{m+1}$, then for $n = 1,\ldots, m$, $B_n$ is surjective, and there exist constants $c_j\in \R$ for $2\le j\le 2([\frac{m}{2}]+1)$, such that the range of $B_{m+1}$ contains
    \begin{equation*}
        \mathrm{R}_{B_{m+1}}:=\left\{y=(y_1,\ldots,y_{2([\frac{m}{2}]+1)})^t\in \R^{2([\frac{m}{2}]+1)}: y_1 = \sum_{2\le j\le 2([\frac{m}{2}]+1)} c_j y_j\right\}.
    \end{equation*}
    \end{enumerate}
\end{lemma}
\begin{remark}
    According to \cite[Lemma 3.4]{godin20212d}, 
    if $\omega = \frac{\pi}{[\frac{m}{2}]+1+n}$ with $0\le n \le m- [\frac{m}{2}]$, then $$
    B_{[\frac{m}{2}]+1+n+ 2l}, \quad \text{for all } l =0,1,\ldots, \left[\frac{m-([\frac{m}{2}]+n)}{2}\right]
    $$ are not surjective. 
    So, for any fixed integer $m\ge 2$, whenever $\omega = \frac{\pi}{[\frac{m}{2}]+1}, \frac{\pi}{[\frac{m}{2}]+2},\ldots, \frac{\pi}{m}$, there are at least two of $B_1,\ldots, B_{m+1}$ being not surjective.
\end{remark}

Under the assumption \eqref{angle-cdt'}, we know from Lemma \ref{lemma-B-surj}, only when $\omega = \frac{\pi}{m+1}$, the matrix $B_{m+1}$ is not surjective. So to solve \eqref{lnsys-1} for $n=m+1$, one needs that 
\begin{equation}\label{rangeB}
    \mathcal{G}^k_{m+1}(0)- \mathcal{F}^k_{m+1}(0) \in 
\mathrm{R}_{B_{m+1}}.
\end{equation}

Note from Remark \ref{rmk-PF-struc} that the leading term of the first entry of $\mathcal{F}^k_{m+1}(0)$ is $\frac{1}{R}\partial_2^{m} (F_1^k)_1(\mathbf{0})$, and it appears only once among all the entries of $\mathcal{F}^k_{m+1}(0)$. 
So, in a similar way as given in \cite[p. 18]{godin20212d}, one can modify $F^k$, such that
\begin{equation}\label{F-arrg}
    \partial_2^{m} (F^k_1)_1(\mathbf{0}) = \sum_{|\gamma|\le m, \gamma_2\le m-1} c_\gamma \partial^\gamma (F_1^k)_1(\mathbf{0}) 
    + \sum_{\substack{|\gamma|\le m,\\ 
    \bar{F}_i \in \{(F_1)_2, (F_2)_1,\\ (F_2)_2, F_3, F_4\}
    }} d_{\gamma n} \partial^\gamma \bar{F}^k_n(\mathbf{0}) + a_k,
\end{equation}
with suitable $c_\gamma, d_{\gamma n}, a_k$ determined by the expression of $\mathcal{G}^k_{m+1}(0)- \mathcal{F}^k_{m+1}(0)$, and $\{c_j\}_{j=2}^{2([\frac{m}{2}]+1)}$ given in Lemma \ref{lemma-B-surj}(2), $\mathbf{0} = (0,0,0)$, without changing the boundary condition $(F^k_2)_i = 0$ on $x_i = 0$ for $i=1,2$.

Therefore, we solve \eqref{lnsys-1}, \eqref{lnsys-2} and fix a choice of $\bar{\alpha}^k_n$ and $\bar{\lambda}^k_{n-1}$. Since $z,F$ satisfy the compatibility conditions up to order $m-1$ and $z^k \to z$, $F^k \to F$ in $H^{m}$ spaces, it follows that 
$$\mathcal{G}^k_n(0)-\mathcal{F}^k_n(0)\to 0,\quad \mathcal{H}^k_{n-1}(0)-\tilde{\mathcal{F}}^k_{n-1}(0) \to 0 \quad \text{ as } k\to +\infty \quad (n\le m-2 ),$$ and hence $\bar{\alpha}^k_n, \bar{\lambda}^k_{n-1} \to 0$ as $k\to +\infty$.
Recalling that $\bar{\lambda}^k_{i,j} = \bar{\xi}^k_{i+1,j}+ \bar{\tau}^k_{i,j+1}$, we set $\bar{\xi}^k_{i+1,j}=\bar{\tau}^k_{i,j+1} = \frac{1}{2}\bar{\lambda}^k_{ij}$ for $i+j=n-1$,
and, according to \eqref{csys-low}, we take $\bar{\xi}^k_{0,n} = \bar{\tau}^k_{n,0}$.
With these choices, we have now constructed the desired coefficients $\bar{\sigma}^k_{ij}$ ($\sigma = \alpha,\beta,\xi,\eta,\tau,\theta$) for $i+j=n$ in \eqref{sigma-cdt}.

\end{proof}

\paragraph{Step~2. Construction of $\mathrm{z}^k_{1j}, \mathrm{z}^k_{2j}$ for $0\le j\le m-1$.}
\begin{lemma}\label{lemma-zkj-low}
    There exists $\delta_0 >0$ depending on the matrix $G$ and $\mathcal{K}$, such that for any $\delta \le \delta_0$, one can find $\mathrm{z}^k_{1j}\in C^\infty(\overline{\mathcal{U}_1})$, $ \mathrm{z}^k_{2j} \in C^\infty(\overline{\mathcal{U}_2})$ such that the following properties hold.
    \begin{enumerate}
        \item[(1)] For $i\le m+1-j$,
        \begin{equation*}
        \partial_2^i \mathrm{z}^k_{1j}(0,0) = (\bar{\alpha}^k_{ji}, \bar{\xi}^k_{ji},\bar{\tau}^k_{ji}), \quad \partial_1^i \mathrm{z}^k_{2j}(0,0) = (\bar{\beta}^k_{ji},\bar{\eta}^k_{ji},\bar{\theta}^k_{ji}).
    \end{equation*}
    \item[(2)] \eqref{zkj-2} holds.
    \item[(3)] \eqref{csys-low}-\eqref{csys-even} hold with $n=m-1$, on both $\mathcal{U}_1$ and $\mathcal{U}_2$.
    \end{enumerate}
\end{lemma}

To have Lemma~\ref{lemma-zkj-low}(2), we introduce the next result given in \cite{godin20212d}.
\begin{lemma}\label{lemma-D-exst}(cf.~\cite[Lemma~3.6]{godin20212d})
    There exists $D^{I,k}_{m-1-j}, D^{II,k}_{m-1-j}, D^{III,k}_{m-1-j}\in C^\infty([0,\delta])$ with $j = 0,\ldots, m-1$, such that
    \begin{enumerate}
        \item[(1)] the mappings $y\mapsto yD^{I,k}_{m-1-j}, yD^{II,k}_{m-1-j}, yD^{III,k}_{m-1-j}$ tend to $0$ in $H^{\frac{1}{2}}(0,\delta)\cap L^2((0,\delta),\d y / y)$  as $k\to +\infty$;
        \item[(2)] let $(E,\sigma)$ be one of $(D^{I}, \alpha), (D^{II}, \xi), (D^{III}, \tau)$, then
        \begin{equation*}
            2^{l} (E^k_{m-1-j})^{(l)}(0) = \bar{\sigma}^k_{m-1-j,j+l+1}- \bar{\sigma}^k_{m-j+l,j},
        \end{equation*}
        for $l=0,1$, and $0\le j\le m-1$.
    \end{enumerate}
\end{lemma}
\noindent Then, by setting that for $j = 0,\ldots, m-1$,
\begin{equation}\label{cysys-*}
\begin{aligned}
    &(\beta^k_j)^{(m-1-j)}(y) = (\alpha^k_{m-1-j})^{(j)}(y) - y D^{I,k}_{m-1-j}(y),\\
    &(\eta^k_j)^{(m-1-j)}(y) = (\xi^k_{m-1-j})^{(j)}(y) - y D^{II,k}_{m-1-j}(y),\\
    &(\theta^k_j)^{(m-1-j)}(y) = (\tau^k_{m-1-j})^{(j)}(y) - y D^{III,k}_{m-1-j}(y),
\end{aligned}
\end{equation}
one gets \eqref{zkj-2} for $0\le j\le m-1$.

To obtain Lemma~\eqref{lemma-zkj-low}(1), we study \eqref{csys-low}-\eqref{csys-even} on both $\mathcal{U}_1$ and $\mathcal{U}_2$.
First, we take
\begin{equation}\label{xi-theta}
    (\xi^k_0)^{(m-1)} \equiv (\theta^k_0)^{(m-1)} \equiv 0
\end{equation}
to take care of \eqref{csys-low}.
Then by \eqref{cysys-*}, one immediately has
\begin{equation}\label{eta-tau}
    (\eta^k_{m-1})^{(0)} \equiv - y D^{II,k}_0(y),\quad (\tau^k_{m-1})^{(0)}(y) \equiv y D^{III,k}_{m-1}(y).
\end{equation}

Next, by the Newton-Leibniz formula and \eqref{cysys-*}, we have that for $i+j\le m-2$ and $\sigma = \alpha,\xi,\tau$,
\begin{equation}\label{NL-1}
    \begin{aligned}
        (\sigma^k_i)^{(j)}(y) = &\sum_{0\le l\le m-2-i-j} \bar{\sigma}^k_{i,j+l} \frac{y^l}{l!} \\ &+ \frac{1}{(m-2-i-j)!} \int_0^y (y-y')^{m-2-i-j} (\sigma^k_i)^{(m-2-i)}(y') \d y';
    \end{aligned}
\end{equation}
and for $(\sigma,\zeta,E) = (\beta,\alpha,D^I), (\eta,\xi,D^{II}), (\theta,\tau,D^{III})$,
\begin{equation}\label{NL-2}
\begin{aligned}
    &(\sigma^k_i)^{(j)}(y) = \sum_{0\le l\le m-2-i-j} \bar{\zeta}^k_{j+l,i} \frac{y^l}{l!} \\& \qquad + \frac{1}{(m-2-i-j)!} \int_0^y (y-y')^{m-2-i-j} \Big((\zeta^k_{m-1-i})^{(i)}(y') - y' E^k_{m-1-i}(y')\Big)\d y'.
\end{aligned}
\end{equation}
Hence, the equations \eqref{csys-odd}-\eqref{csys-even} with $n= m-1$, $h=1,2$ can be regarded as an integral equation for
\begin{equation*}
    \varphi^k(y) = \left(\hat{\alpha}^k,\hat{\xi}^k\right)^t (y)
    = \left((\alpha^k_0)^{(m-1)}, \ldots, (\alpha^k_{m-1})^{(0)},(\xi^k_1)^{(m-2)}, \ldots, (\xi^k_{m-1})^{(0)}\right)^t(y),
\end{equation*}
whose leading-order coefficient matrix at $x=0$ is given by
\begin{equation*}
    B = \begin{pmatrix}
        B_{m-1} & \\ & B_{m-2}
    \end{pmatrix}.
\end{equation*}
Then $B$ is a $\left(2[\frac{m}{2}]+2[\frac{m-1}{2}]\right) \times (2m-1)$ surjective matrix.
Note that $2m-1 = 2[\frac{m}{2}] + 2[\frac{m-1}{2}] + 1$. This means that $B$ has exactly one more column to be a square matrix. 
This allows us to delete a column from $B$, say the $d$-th column, so that the resulting matrix $B^*$ is invertible.
When $1\le d\le m$, we delete $(\alpha^k_{d-1})^{m-d}$ from $\varphi^k$, and while $m+1\le d\le 2m-1$, we delete $(\xi^k_{d-m})^{2m-1-d}$ from $\varphi^k$. Denote by $\varphi^k_*$ the resulting vector.

Now we are free to choose $\tau^k_j\in C^\infty([0,\delta])$ for $j = 0,\ldots, m-2$, and $\alpha^k_{d-1} \in C^\infty([0,\delta])$ if $1\le d\le m$; $\xi^k_{d-m} \in C^\infty([0,\delta])$ if $m+1\le d\le 2m-1$,
such that
\begin{equation}\label{tau}
    \begin{aligned}
        \tau^k_j \to 0 \text{ in } H^{m-\frac{1}{2}-j} \text{ as } k\to +\infty,\quad (\tau^k_j)^{(i)}(0) = \bar{\tau}^k_{ji},\, 0\le i\le m+1-j;
    \end{aligned}
\end{equation}
\begin{equation}\label{alpha}
    \begin{aligned}
        \alpha^k_{d-1} \to 0 \text{ in } H^{m+\frac{1}{2}-d} \text{ as } k\to +\infty,\quad (\alpha^k_{d-1})^{(i)}(0) = \bar{\alpha}^k_{d-1,i},\, 0\le i\le m+2-d,
    \end{aligned}
\end{equation}
if $1\le d\le m$; and
\begin{equation}\label{xi}
    \begin{aligned}
        \xi^k_{d-m} \to 0 \text{ in } H^{2m-\frac{1}{2}-d} \text{ as } k\to +\infty,\quad (\xi^k_{d-m})^{(i)}(0) = \bar{\xi}^k_{d-m,i},\, 0\le i\le 2m+1-d,
    \end{aligned}
\end{equation}
if $m+1\le d\le 2m-1$. This is achievable, see \cite[Lemma~3.7]{godin20212d} for the proof.

Finally, \eqref{csys-odd}-\eqref{csys-even} are transformed into the integral equation
\begin{equation}\label{Bsys}
    B_* \varphi^k_*(y) = y a(y,Z) \varphi^k_*(y) + \sum_{\gamma\in I} b_\gamma (y,Z)\int_0^y P_\gamma (y-y') \varphi^k_*(y') \d y' + \psi^k(y,Z),
\end{equation}
where $a, b_\gamma$ and $P_\gamma$ are given $(2m-2)\times (2m-2)$ smooth matrices (independent of $\varphi^k_*$),
the set $I$ is finite, and $\psi^k \in C^\infty([0,\delta]\times \mathcal{K}_1)$ is a given $(2m-2)$-dimensional vector (independent of $\varphi^k_*$) consisting of the terms contributed from
\begin{enumerate}
    \item[(1)] $\bar{\alpha}^k_{ij},i+j \le m-2$, except for $i = d-1$ if $1\le d\le m$, and $\bar{\xi}^k_{ij},i+j \le m-2$, except for $i = d-m$ if $m+1\le d\le 2m-1$;
    \item[(2)] $\tau^k_j(y)$, $0\le j\le m-2$, and $q^k(y),v^k(y),b^k(y),{F}^k(y)$;
    \item[(3)] $\alpha^k_{d-1}(y)$ if $1\le d\le m$, or $\xi^k_{d-m}$ if $m+1\le d\le 2m-1$;
    \item[(4)] $y D^{I,k}_j(y), y D^{II,k}_j(y), y D^{III,k}_j(y)$, $0\le j\le m-1$,
\end{enumerate}
with coefficients being the form of $\Theta(y,Z)$ for some $\Theta\in C^\infty ([0,\delta]\times \mathcal{K}_1)$.

Then \eqref{sigma-cdt}, \eqref{tau}-\eqref{xi}, the fact that $q^k, v^k, b^k, {F}^k \to \p, u, b, {F}$ in respective $H^{m}$ spaces and the convergence properties of $D^{I,k}_j,D^{II,k}_j,D^{III,k}_j$ show that for any fixed $Z$, $\psi^k(y,Z) \to 0$ in $H^{\frac{1}{2}}(0,\delta)$ as $k\to +\infty$.

\begin{lemma}\label{solution}(cf.~\cite[Lemma~3.9]{godin20212d})
    There exists $\delta_0 >0$ depending on the matrix $G$ and $\mathcal{K}$, such that for any $\delta \in (0,\delta_0]$, \eqref{Bsys} has a unique solution $\varphi^k_* \in C^\infty([0,\delta])$; moreover, $\varphi^k_* \to 0$ in $H^{\frac{1}{2}}(0,\delta)$ as $k\to +\infty$.
\end{lemma}

By solving \eqref{Bsys}, we have constructed $\mathrm{z}^k_{1j}, \mathrm{z}^k_{2j}$ for $j=0,\ldots, m-1$ such that all properties stated in Lemma~\ref{lemma-zkj-low} are satisfied. Then the proof of Lemma~\ref{lemma-zkj-low} is done.

\paragraph{Step~3. Suitable choices of $\mathrm{z}^k_{1j}, \mathrm{z}^k_{2j}$ with $j=m,m+1$.}
\begin{lemma}\label{lemma-final}
    There exists $\delta_0 >0$ depending on the matrix $G$ and $\mathcal{K}$, such that for any $\delta \in (0,\delta_0]$, one can find
    $\mathrm{z}^k_{1j}\in C^\infty(\overline{\mathcal{U}}_1)$, $\mathrm{z}^k_{2j} \in C^\infty(\overline{\mathcal{U}}_2)$, such that for $j=m, m+1$,
    \begin{enumerate}
        \item[(1)] $\partial_2^i \mathrm{z}^k_{1j}(0,0) = (\bar{\alpha}^k_{ji}, \bar{\xi}^k_{ji}, \bar{\tau}^k_{ji})$, with  $i\le m+1-j, j = m, m+1$, and $\partial_1^j \mathrm{z}^k_{2i}(0,0) = (\bar{\alpha}^k_{ji}, \bar{\xi}^k_{ji}, \bar{\tau}^k_{ji})$ with $j\le m+1-i, i=m,m+1$;
        \item[(2)] \eqref{csys-odd}-\eqref{csys-even} hold for $n=m, m+1$.
    \end{enumerate}
\end{lemma}
In contrast to Step~2, the construction of $\mathrm{z}^k_{1j}$ and $\mathrm{z}^k_{2j}$ for $j = m, m+1$ is much easier, as these cases possess more freedom and do not require properties like \eqref{zkj-2}.
The proof is similar to that given in \cite[pp.~29–30]{godin20212d} and is therefore omitted.

Therefore, Lemma~\ref{lemma-zkj} is established. This in turn implies Proposition~\ref{prop-lift-cpbt}.

\section{Appendix B}
\renewcommand{\thelemma}{B.\arabic{lemma}}
\setcounter{lemma}{0}
\renewcommand{\theprop}{B.\arabic{prop}}
\setcounter{prop}{0}
\renewcommand{\theremark}{B.\arabic{remark}}
\setcounter{remark}{0}
In this appendix, we collect some auxiliary materials that are used throughout the paper.

\begin{lemma}\label{lemma-h}
    Let $\tilde{x} = q_n$ be a corner on $\partial\Omega$ for some $n$, and $\Phi_1(x),\Phi_2(x)$ be as introduced in Notation \ref{nota-1}.
    The two matrices $h^1, h^2$ given in  Lemma \ref{lemma-h-def} can be explicitly represented as
    \begin{equation*}
        h^1 = \frac{1}{\partial_1\Phi_1 \partial_2 \Phi_2 - \partial_1 \Phi_2 \partial_2 \Phi_1} \begin{pmatrix}
            \partial_2 \Phi_2 \partial_1^2 \Phi_1 - \partial_2 \Phi_1 \partial_1^2 \Phi_2 & \partial_2 \Phi_2 \partial_1\partial_2 \Phi_1 - \partial_2 \Phi_1 \partial_1\partial_2 \Phi_2\\
            -\partial_1 \Phi_2 \partial_1^2 \Phi_1 + \partial_1 \Phi_1 \partial_1^2 \Phi_2 & -\partial_1 \Phi_2 \partial_1\partial_2 \Phi_1 + \partial_1 \Phi_1 \partial_1\partial_2 \Phi_2
        \end{pmatrix},
    \end{equation*}
    \begin{equation*}
        h^2 = \frac{1}{\partial_1\Phi_1 \partial_2 \Phi_2 - \partial_1 \Phi_2 \partial_2 \Phi_1} \begin{pmatrix}
            \partial_2 \Phi_2 \partial_1\partial_2 \Phi_1 - \partial_2 \Phi_1 \partial_1\partial_2 \Phi_2 & \partial_2 \Phi_2 \partial_2^2 \Phi_1 - \partial_2 \Phi_1 \partial_2^2 \Phi_2\\
            -\partial_1 \Phi_2 \partial_1\partial_2 \Phi_1 + \partial_1 \Phi_1 \partial_1\partial_2 \Phi_2 & -\partial_1 \Phi_2 \partial_2^2 \Phi_1 + \partial_1 \Phi_1 \partial_2^2 \Phi_2
        \end{pmatrix}.
    \end{equation*}
\end{lemma}

The next proposition is important in the proof of the normal regularity through the divergence-curl system in Section~\ref{Section-wellposed}.
First introduce the notations of the spatial tangential operators with $\alpha = (\alpha_1, \alpha_2)\in \mathbb{N}^2$:
\begin{equation*}
    \mathcal{T}^\alpha_x = \mathcal{T}_1^{\alpha_1} \mathcal{T}_2^{\alpha_2}, \quad \mathscr{T}^\alpha_x = \mathscr{T}_1^{\alpha_1} \mathscr{T}_2^{\alpha_2}, \quad \partial_x^\alpha = \partial_1^{\alpha_1}\partial_2^{\alpha_2}.
\end{equation*}
Denote by $D_{\mathcal{T}_x}^s$ (resp. $D_x^s$) a family of tangential operators of the form $\sum_{|\alpha|\le s}\Theta^\alpha(x) \mathcal{T}_x^{\alpha}$ (resp. $\sum_{|\alpha|\le s}\Theta^\alpha(x) \partial_x^\alpha$), with $\Theta^\alpha$ being some smooth functions, and may differ from line to line.

\begin{prop}\label{prop-PD-PN}
    Assume that $m, s\in \mathbb{N}$, and $\omega_n < \frac{\pi}{m+1}$ for $n=1,\ldots, N$. 
    One can find $C>0$ such that the following holds: 
    \begin{enumerate}
        \item[(1)] For any given $f_1\in H^{m}_x (H^s_\tg (\R\times \Omega))$, the homogeneous Poisson-Dirichlet problem 
    \begin{equation}\label{PD-eq}
        \begin{cases}
            \Delta K_1 = f_1, & \text{in }\R\times \Omega,\\
            K_1 = 0, & \text{on }\R\times (\partial\Omega \backslash\{q_1,\ldots,q_N\})
        \end{cases}
    \end{equation}
    has a unique solution $K_1 := \Delta_D^{-1} f_1$ in $H^{m+2}_x(H^s_\tg (\R\times\Omega))$, satisfying 
    $$\|K_1\|_{H^{m+2}_x(H^{s}_\tg(\R\times \Omega))}\le C\|f_1\|_{H^{m}_x(H^{s}_\tg(\R\times \Omega))}.$$
    \item[(2)]  For any given $f_2\in H^{m}_x (H^s_\tg (\R\times \Omega))$ with $\int_\Omega f_2 ~\d x = 0$, the homogeneous Poisson-Neumann problem 
    \begin{equation}\label{PN-eq}
        \begin{cases}
            \Delta K_2 = f_2, & \text{in }\R\times \Omega,\\
            \partial_\nu K_2 = 0, & \text{on }\R\times(\partial\Omega \backslash\{q_1,\ldots,q_N\})\\
            \int_\Omega K_2 ~\d x = 0 &
        \end{cases}
    \end{equation}
    has a unique solution $K_2 := \Delta_N^{-1} f_2$ in $H^{m+2}_x(H^s_\tg (\R\times\Omega))$ satisfying 
    $$\|K_2\|_{H^{m+2}_x(H^{s}_\tg(\R\times \Omega))}\le C\|f_2\|_{H^{m}_x(H^{s}_\tg(\R\times \Omega))}.$$
    \end{enumerate}
\end{prop}
\begin{proof}
We begin with the proof of Proposition~\ref{prop-PD-PN}~(1).

Note that \cite[Proposition~A.3]{godin20212d} has shown that for any $f\in H^{m}(\Omega)$, there exists a unique solution $u = \Delta_D^{-1} f \in H^{m+2}(\Omega)$, satisfying
    \begin{equation*}
        \|u\|_{H^{m+2}(\Omega)}\lesssim \|f\|_{H^{m}(\Omega)},
    \end{equation*}
under the assumptions $\omega_n < \frac{\pi}{m+1}$.
    
Now if, in addition, $f$ has tangential regularity, $f\in H^{m}_x(H^s_\tg(\Omega))$, we approximate $f$ by $f_\epsilon \in C^\infty_c(\overline{\Omega})$, such that $f_\epsilon \to f$ in $H^{m}_x(H^s_\tg(\Omega))$ as $\epsilon\to 0$ (this is guaranteed by the density result given in Lemma~\ref{lemma-density}). Then $u_\epsilon := \Delta_D^{-1} f_\epsilon \in \cap_{n\in \mathbb{N}} H^{n}_x(\Omega)$. 

By applying $D_{\mathcal{T}_x}^k$ with $1\le k\le s$ on $\Delta_D u_\varepsilon = f_\varepsilon$, one has
\begin{equation*}
    \begin{cases}
        \Delta (D^k_{\mathcal{T}_x} u_\epsilon) = D^k_{\mathcal{T}_x} f_\epsilon + D^{k-1}_{\mathcal{T}_x}\partial_x^2 u_\epsilon, & \text{in }\R\times \Omega,\\
            D^k_{\mathcal{T}_x} u_\epsilon = 0, & \text{on }\R\times(\partial\Omega \backslash\{q_1,\ldots,q_N\}),
    \end{cases}
\end{equation*}
which implies $$
\begin{aligned}
    \|D^k_{\mathcal{T}_x} u_\epsilon\|_{H^{m+2}(\Omega)} &\lesssim \|D^k_{\mathcal{T}_x} f_\epsilon\|_{H^m(\Omega)} + \|D^{k-1}_{\mathcal{T}_x} \partial_x^2 u_\epsilon\|_{H^m(\Omega)} \\
    &\lesssim \|f_\epsilon\|_{H^m_x(H^{k}_{\tg})(\Omega)} + \|u_\epsilon\|_{H^{m+2}_x(H^{k-1}_{\tg})(\Omega)}.
\end{aligned} $$
By taking induction on $k=1,\ldots, s$, one concludes the estimate
$\|u_\epsilon\|_{H^{m+2}_x(H^{s}_\tg(\Omega))}\lesssim \|f_\epsilon\|_{H^{m}_x(H^{s}_\tg(\Omega))}$. 

By taking the differences $u_\epsilon-u_{\epsilon'}$, one sees that $\{u_\epsilon\}_{\epsilon>0}$ is a Cauchy sequence in $H^{m+2}_x(H^{s}_\tg(\Omega))$, and it converges to the solution of $\Delta_D u = f$. This shows that $u = \Delta_D^{-1} f$ belongs to $H^{m+2}_x(H^s_\tg(\Omega))$, and satisfies the a priori estimate $\|u\|_{H^{m+2}_x(H^{s}_\tg(\Omega))}\lesssim \|f\|_{H^{m}_x(H^{s}_\tg(\Omega))}$.

    Next, we consider when $f \in L^2(\R;H^{m}_x(H^s_\tg(\Omega)))$. 
    For almost every $t\in \R$, we set $u(t)=\Delta_D^{-1}f(t)$, so that $u(t)\in H^{m+2}_x(H^s_\tg (\Omega))$. 
    Since we have shown that $\Delta_D^{-1}$ is continuous from $H^{m}_x(H^s_\tg(\Omega))$ to $H^{m+2}_x(H^s_\tg(\Omega))$, it follows that the map $t\mapsto u(t)$ is measurable, and
    \begin{equation*}
        \int_{\R} \|u(t)\|^2_{H^{m+2}_x(H^s_\tg(\Omega))}\d t\lesssim \int_{\R} \|f(t)\|^2_{H^{m}_x(H^s_\tg(\Omega))}\d t.
    \end{equation*}
    Hence $u\in L^2(\R;H^{m+2}_x(H^s_\tg(\Omega))$.
    Moreover, if 
    \begin{equation*}
        f\in H^{m}_x(H^s_\tg(\R\times \Omega)) = \bigcap_{j=0}^s H^j(\R; H^{m}_x (H^{s-j}_\tg (\Omega))),
    \end{equation*} then, by applying the difference quotient $D_h u(t) = h^{-1}\big(u(t+h)-u(t)\big)$ to both sides of the equation and proceeding by induction, one shows that $u\in H^{m+2}_x(H^s_\tg(\R\times \Omega))$. The desired estimate follows as a by-product.
    This completes the proof of Proposition~\ref{prop-PD-PN}(1).

    Next, we prove Proposition~\ref{prop-PD-PN}(2).
    Note that \cite[Proposition~A.4]{godin20212d} (for the homogeneous case) and \cite[Theorem~5.1.2.4]{grisvard1985elliptic} (for the existence of an auxiliary function to reduce the nonhomogeneous case into the homogeneous case) have shown that under the assumptions $\omega_n < \frac{\pi}{m+1}$, for any $f\in H^{m}(\Omega), g\in H^{m+{\frac{1}{2}}}(\partial\Omega)$ that satisfy the compatibility conditions at the corner points, and $\int_\Omega f~ \d x = \int_{\partial\Omega} g ~\d \sigma$, there exists a solution $u \in H^{m+2}(\Omega)$ to the nonhomogeneous boundary value problem
    \begin{equation}\label{PN-eq'}
        \begin{cases}
            \Delta  u =  f, & \text{in } \R\times \Omega,\\
            \partial_\nu u = g, & \text{on }\R\times(\partial\Omega\backslash\{q_1,\ldots,q_N\}).
        \end{cases}
    \end{equation}
    Moreover, the solution is unique modulo constants, and the estimate
    \begin{equation}\label{PN-non-est}
        \|\nabla u\|_{H^{m+1}(\Omega)} \lesssim \|f\|_{H^m(\Omega)} + \|g\|_{H^{m+\frac{1}{2}}(\partial\Omega)}
    \end{equation}
    holds.

    Similar to the proof of Proposition~\ref{prop-PD-PN}(1), it suffices to prove the a priori estimates in $H^{m+2}_x(H^s_{\tg})(\Omega)$.

By applying $D_{\mathcal{T}_x}^k$ with $1\le k\le s$ to the equations: $\Delta u = f$ in $\Omega$, $\partial_\nu u = 0$ on $\partial\Omega\backslash\{q_1,\ldots, q_N\}$ with $u, f$ sufficiently smooth, one has
    \begin{equation*}
        \begin{cases}
            \Delta (D_{\mathcal{T}_x}^k u) = D_{\mathcal{T}_x}^k f +  D_{\mathcal{T}_x}^{k-1}\partial_x^2 u, & \text{in }\Omega,\\
            \partial_\nu (D_{\mathcal{T}_x}^k u) = D_{\mathcal{T}_x}^{k-1} \partial_x^1 u, & \text{on }\partial\Omega \backslash\{q_1,\ldots,q_N\}.
        \end{cases}
    \end{equation*}
    Then by \eqref{PN-non-est}, and the trace theorem \cite[Theorem~1.5.1.3]{grisvard1985elliptic}, one has
    \begin{equation*}
        \begin{aligned}
            \|\nabla D_{\mathcal{T}_x}^k u\|_{H^{m+1}(\Omega)} &\lesssim \|D_{\mathcal{T}_x}^k f +  D_{\mathcal{T}_x}^{k-1}\partial_x^2 u\|_{H^m(\Omega)} + \|D_{\mathcal{T}_x}^{k-1}\partial_x^1 u\|_{H^{m+\frac{1}{2}}(\partial\Omega)}\\
            & \lesssim \|f\|_{H^m_x(H^k_{\tg})(\Omega)} + \|\nabla u\|_{H^{m+1}_x(H^{k-1}_{\tg})(\Omega)}.
        \end{aligned}
    \end{equation*}
    By \eqref{tan-cmtt-1} and taking induction on $k=1,\ldots, s$, one gets the estimates $\|\nabla u\|_{H^{m+1}_x(H^{s}_{\tg})(\Omega)} \lesssim \|f\|_{H^m_x(H^s_{\tg})(\Omega)}$.

    The rest of the proof is similar to that of Proposition~\ref{prop-PD-PN}(1). At the end, the estimate given in Proposition~\ref{prop-PD-PN}(2) is obtained by a Poincar\'e-type inequality. 

\end{proof}

In the following, we present the proof of Proposition~\ref{prop-Helmholtz}.
Analogously to the argument for Proposition~\ref{prop-PD-PN}(1), to establish the Helmholtz decomposition in $H^m_x(H^s_{\tg})(\R\times \Omega)$, it suffices to prove the following lemma, which is a purely spatial version of Proposition~\ref{prop-Helmholtz}.
\begin{lemma}\label{lemma-Helmholtz}
    For any given $v=(v_1,v_2)^t \in H^m_x(H^s_{\tg})(\Omega)$ with $m,s \in \mathbb{N}$, there is a unique decomposition
    \begin{equation*}
        v = \nabla f + g,
    \end{equation*}
    for some $f\in H^{m+1}_x(H^s_{\tg})(\Omega)$ and $g\in H^m_x(H^s_{\tg})(\Omega)$, such that $\nabla\cdot g = 0$ (if $m=0$, this is understood in the sense of distribution), and there exists a constant $C$, independent of $v,f$ and $g$, such that
    \begin{equation*}
        \|\nabla f\|_{H^m_x(H^s_{\tg})(\Omega)} \le C\|v\|_{H^m_x(H^s_{\tg})(\Omega)}.
    \end{equation*}
    Moreover, when $\int_\Omega f \d x = 0$, then $f \in H^{m+1}_x(H^s_{\tg})(\Omega)$ is unique.
\end{lemma}
\begin{proof}
If $v\in H^m_x(H^s_{\tg})(\Omega) \subset L^2(\Omega)$, from \cite[Section~II.2.5]{sohr2012navier}, we know that there is a unique decomposition $v = \nabla f + g$, with $f \in H^1_x(\Omega), g\in L^2(\Omega)$, $\int_\Omega f \d x = 0, \nabla \cdot g = 0$, and $\| f\|_{H^1(\Omega)} \le \|v\|_{L^2(\Omega)}$.

Now we approximate $v$ by $v_\epsilon \in C_c^\infty(\overline{\Omega})$, such that $v_\varepsilon \to v$ in $H^m_x(H^s_{\tg})(\Omega)$ as $\epsilon \to 0$. Let $\nabla f_\epsilon + g_\epsilon$ be the $L^2$ Helmholtz decomposition of $v_\epsilon$, then it is easy to see that $f_\epsilon$ is the solution of the following elliptic  problem
    \begin{equation}\label{PN-eq-0}
        \begin{cases}
            \Delta f_\epsilon = \nabla\cdot v_\epsilon, & \text{in } \Omega,\\
            \partial_\nu f_\epsilon = v_\epsilon \cdot \nu, & \text{on } \partial\Omega\backslash\{q_1,\ldots, q_N\}.
        \end{cases}
    \end{equation}
    We claim that for any $m, s\in \mathbb{N}$, the solution $f_\epsilon$ satisfies the estimate
    \begin{equation}\label{epsilon-est}
        \|\nabla f_\epsilon\|_{H^m_x(H^s_{\tg})(\Omega)} \lesssim \|v_\epsilon\|_{H^m_x(H^s_{\tg})(\Omega)}.
    \end{equation}
    Suppose that in addition, $\int_\Omega f_\epsilon \d x = 0$.
    Then by taking the difference $f_\epsilon - f_{\epsilon'}$, one sees that $\{f_\epsilon\}_{\epsilon>0}$ is a Cauchy sequence in $H^{m+1}_x(H^s_{\tg})(\Omega)$, and converges to some $f_0$ in $H^{m+1}_x(H^s_{\tg})(\Omega)$. Then $\int_\Omega f_0 \d x=0$.
    By the uniqueness of the $L^2$ Helmholtz decomposition of $v$, one finds that $f_0$ is identical to $f$. Thus, the Helmholtz decomposition of $v$ is indeed bounded in $H^m_x(H^s_{\tg})(\Omega)$.

    Now we are left to prove the estimate \eqref{epsilon-est}. For simplicity of notations, we drop the $\epsilon$ of $f_\epsilon, v_\epsilon$, and drop the $x$ of $\mathcal{T}_x, \mathscr{T}_x$, and assume that $f, v$ are sufficiently smooth functions.

    For $m = 0$, we apply $\mathcal{T}^\alpha$ with $\alpha\in \mathbb{N}^2, |\alpha| = k$, $1\le k\le s$ on \eqref{PN-eq-0}, and obtain
    \begin{equation}\label{PN-eq-1}
        \begin{cases}
        \mathcal{T}^\alpha (\Delta f) = \mathcal{T}^\alpha (\nabla \cdot v), & \text{in } \Omega,\\
            (\mathscr{T}^\alpha \nabla f)\cdot \nu = \mathcal{T}^\alpha(v\cdot \nu), & \text{on } \partial\Omega\backslash\{q_1,\ldots, q_N\},
        \end{cases}
    \end{equation}
    where we have used the fact $\mathcal{T}^\alpha (\nabla f \cdot \nu) = (\mathscr{T}^\alpha \nabla f)\cdot \nu$.
    Denote by
    \begin{equation*}
        \mathscr{H}^\alpha v:= \nabla\cdot (\mathscr{T}^\alpha v) - \mathcal{T}^\alpha (\nabla\cdot v) =
        \nabla \cdot \left((\mathscr{T}^\alpha - \mathcal{T}^\alpha) v\right) + [\nabla\cdot, \mathcal{T}^\alpha]v.
    \end{equation*}
    Then, on one hand, 
    \begin{equation*}
        \begin{aligned}
            \inp*{\mathcal{T}^\alpha (\nabla\cdot v)}{\mathcal{T}^\alpha f}_\Omega &= \inp*{\nabla\cdot(\mathscr{T}^\alpha v) - \mathscr{H}^\alpha v}{\mathcal{T}^\alpha f}_\Omega\\
            & = \inp*{\mathcal{T}^\alpha (v\cdot \nu)}{\mathcal{T}^\alpha f}_{\partial\Omega} - \inp*{\mathscr{T}^\alpha v}{\nabla \mathcal{T}^\alpha f}_{\Omega} - \inp*{\mathscr{H}^\alpha v}{\mathcal{T}^\alpha f}_\Omega.
        \end{aligned}
    \end{equation*}
    On the other hand,
    \begin{equation*}
        \begin{aligned}
            \inp*{\mathcal{T}^\alpha (\Delta f)}{\mathcal{T}^\alpha f}_\Omega  &= \inp*{\nabla\cdot(\mathscr{T}^\alpha \nabla f) - \mathscr{H}^\alpha (\nabla f)}{\mathcal{T}^\alpha f}_\Omega\\
            & = \inp*{\mathcal{T}^\alpha (\partial_\nu f)}{\mathcal{T}^\alpha f}_{\partial\Omega} - \inp*{\mathscr{T}^\alpha \nabla f}{\nabla \mathcal{T}^\alpha f}_{\Omega} - \inp*{\mathscr{H}^\alpha (\nabla f)}{\mathcal{T}^\alpha f}_\Omega.
        \end{aligned}
    \end{equation*}
    Combining \eqref{PN-eq-1}, one has
    \begin{equation*}
        \inp*{\mathscr{T}^\alpha \nabla f}{\nabla \mathcal{T}^\alpha f}_{\Omega} = 
        \inp*{\mathscr{T}^\alpha v}{\nabla \mathcal{T}^\alpha f}_{\Omega} + \inp*{\mathscr{H}^\alpha (v-\nabla f)}{\mathcal{T}^\alpha f}_\Omega.
    \end{equation*}
        Thus, we get
    \begin{equation*}
        \begin{aligned}
            \|\mathscr{T}^\alpha \nabla f\|_{L^2(\Omega)}^2 & = 
            \inp*{\mathscr{T}^\alpha f}{\mathscr{T}^\alpha \nabla f - \nabla \mathcal{T}^\alpha f}_{\Omega} + \inp*{\mathscr{T}^\alpha \nabla f}{\nabla \mathcal{T}^\alpha f}_{\Omega}\\
            & = & \inp*{\mathscr{T}^\alpha \nabla f}{(\mathscr{T}^\alpha- \mathcal{T}^\alpha)\nabla f + [\mathcal{T}^\alpha, \nabla]f}_\Omega \\
            & + \inp*{\mathscr{T}^\alpha v}{\nabla \mathcal{T}^\alpha f}_{\Omega} + \inp*{\mathscr{H}^\alpha(v-\nabla f)}{\mathcal{T}^\alpha f}_\Omega\\
            & \lesssim  & \|\mathscr{T}^\alpha \nabla f\|_{L^2(\Omega)} \cdot \|D_{\mathcal{T}}^{k-1} \partial_x^1 f\|_{L^2(\Omega)} + \|\mathscr{T}^\alpha v\|_{L^2(\Omega)} \cdot \|D_{\mathcal{T}}^k \partial_x f\|_{L^2(\Omega)} \\
            & + \left|\inp*{\mathscr{H}^\alpha v}{\mathcal{T}^\alpha f}_{\Omega}\right| + \left|\inp*{\mathscr{H}^\alpha (\nabla f)}{\mathcal{T}^\alpha f}_{\Omega}\right|,
        \end{aligned}
    \end{equation*}
    Recall that Lemma~\ref{lemma-J2} implies
    \begin{equation*}
        \inp*{\mathscr{H}^\alpha v}{\mathcal{T}^\alpha f}_{\Omega} = \inp*{D_{\mathcal{T}}^k v}{D_{\mathcal{T}}^{k-1}\partial_x^1 f}_{\Omega}, \quad \inp*{\mathscr{H}^\alpha (\nabla f)}{\mathcal{T}^\alpha f}_{\Omega} = \inp*{D_{\mathcal{T}}^k (\nabla f)}{D_{\mathcal{T}}^{k-1}\partial_x^1 f}_{\Omega}.
    \end{equation*}
    Then, by Cauchy's inequality, and taking the sum of all $\alpha \in \mathbb{N}^2, |\alpha|= k$, one has $$\|\nabla f\|^2_{H^{k}_{\tg}(\Omega)} \lesssim \|\nabla f\|_{H^{k-1}_{\tg}(\Omega)}^2 + \|v\|_{H^k_{\tg}(\Omega)}^2.$$ By taking induction of $k=1,\ldots, s$, one has
    $$\|\nabla f\|^2_{H^{s}_{\tg}(\Omega)} \lesssim \|v\|_{H^s_{\tg}(\Omega)}^2 + \|\nabla f\|^2_{L^2(\Omega)} \lesssim \|v\|_{H^s_{\tg}(\Omega)}^2.$$

    For $m \ge 1$, we rewrite \eqref{PN-eq-1} as
    \begin{equation}\label{PN-eq-2}
        \begin{cases}
        \Delta (\mathcal{T}^\alpha f) = \mathcal{T}^\alpha(\nabla\cdot v) + [\Delta, \mathcal{T}^\alpha] f, & \text{in } \Omega,\\
            \partial_\nu(\mathcal{T}^\alpha f) = \mathcal{T}^\alpha(v\cdot \nu) + [\partial_\nu, \mathcal{T}^\alpha] f, & \text{on } \partial\Omega\backslash\{q_1,\ldots, q_N\},
        \end{cases}
    \end{equation}
    and apply the estimates \eqref{PN-non-est} to the problem \eqref{PN-eq-2}. Then by taking the sum of $\alpha$ over $|\alpha|=k$ and taking inductions on $k=1,\ldots, s$, one has $$\|\nabla f\|_{H^m_x(H^s_{\tg})(\Omega)} \le \|v\|_{H^m_x(H^s_{\tg})(\Omega)} + \|\nabla f\|_{H^m_x(\Omega)} \lesssim \|v\|_{H^m_x(H^s_{\tg})(\Omega)}.$$
    This ends the proof.
\end{proof}

\vspace{.1in}
\hspace{-.25in}{\bf Acknowledgements.} This research was supported by National Key R\&D Program of China under grant 2024YFA1013302, and National Natural Science Foundation of China under grant Nos.12331008, 12171317 and 12250710674.

\newpage
\printbibliography
\end{document}